\documentclass[10pt]{amsart}
\usepackage{amsmath}
\usepackage{amssymb}
\usepackage{amsbsy, amsthm,amstext, amsopn}
\usepackage[all]{xy}
\usepackage{amsfonts}
\usepackage{amscd}
\usepackage{stmaryrd}
\newtheorem{thm}{Theorem} [section]
\newtheorem{lemma}[thm]{Lemma}

\newtheorem{corollary}[thm]{Corollary}
\newtheorem{prop}[thm]{Proposition}

\theoremstyle{definition}

\newtheorem{defn}[thm]{Definition}

\newtheorem{example}[thm]{Example}

\newtheorem{ansatz}[thm]{Ansatz}

\newtheorem{remark}[thm]{Remark}

\begin{document}

\numberwithin{equation}{section}

\newcommand{\hs}{\mbox{\hspace{.4em}}}
\newcommand{\ds}{\displaystyle}
\newcommand{\bd}{\begin{displaymath}}
\newcommand{\ed}{\end{displaymath}}
\newcommand{\bcd}{\begin{CD}}
\newcommand{\ecd}{\end{CD}}

\newcommand{\on}{\operatorname}
\newcommand{\proj}{\operatorname{Proj}}
\newcommand{\bproj}{\underline{\operatorname{Proj}}}

\newcommand{\spec}{\operatorname{Spec}}
\newcommand{\Spec}{\operatorname{Spec}}
\newcommand{\bspec}{\underline{\operatorname{Spec}}}
\newcommand{\pline}{{\mathbf P} ^1}
\newcommand{\aline}{{\mathbf A} ^1}
\newcommand{\pplane}{{\mathbf P}^2}
\newcommand{\aplane}{{\mathbf A}^2}
\newcommand{\coker}{{\operatorname{coker}}}
\newcommand{\ldb}{[[}
\newcommand{\rdb}{]]}

\newcommand{\Sym}{\operatorname{Sym}^{\bullet}}
\newcommand{\Symp}{\operatorname{Sym}}
\newcommand{\Pic}{\bf{Pic}}
\newcommand{\Aut}{\operatorname{Aut}}
\newcommand{\PAut}{\operatorname{PAut}}

\newcommand{\too}{\twoheadrightarrow}
\newcommand{\C}{{\mathbf C}}
\newcommand{\bD}{{\mathbf D}}
\newcommand{\Z}{{\mathbf Z}}
\newcommand{\Q}{{\mathbf Q}}
\newcommand{\Cx}{{\mathbf C}^{\times}}
\newcommand{\Cbar}{\overline{\C}}
\newcommand{\Cxbar}{\overline{\Cx}}
\newcommand{\cA}{{\mathcal A}}
\newcommand{\cS}{{\mathcal S}}
\newcommand{\cV}{{\mathcal V}}
\newcommand{\cM}{{\mathcal M}}
\newcommand{\bA}{{\mathbf A}}
\newcommand{\cB}{{\mathcal B}}
\newcommand{\cC}{{\mathcal C}}
\newcommand{\cD}{{\mathcal D}}
\newcommand{\D}{{\mathcal D}}
\newcommand{\cs}{{\mathbf C} ^*}
\newcommand{\boldc}{{\mathbf C}}
\newcommand{\cE}{{\mathcal E}}
\newcommand{\cF}{{\mathcal F}}
\newcommand{\bF}{{\mathbf F}}
\newcommand{\cG}{{\mathcal G}}
\newcommand{\G}{{\mathbb G}}
\newcommand{\cH}{{\mathcal H}}
\newcommand{\CI}{{\mathcal I}}
\newcommand{\cJ}{{\mathcal J}}
\newcommand{\cK}{{\mathcal K}}
\newcommand{\cL}{{\mathcal L}}
\newcommand{\baL}{{\overline{\mathcal L}}}

\newcommand{\fI}{{\mathfrak I}}
\newcommand{\fJ}{{\mathfrak J}}
\newcommand{\fF}{{\mathfrak F}}
\newcommand{\Mf}{{\mathfrak M}}
\newcommand{\bM}{{\mathbf M}}
\newcommand{\bm}{{\mathbf m}}
\newcommand{\cN}{{\mathcal N}}
\newcommand{\theo}{\mathcal{O}}
\newcommand{\cP}{{\mathcal P}}
\newcommand{\cR}{{\mathcal R}}
\newcommand{\Pp}{{\mathbb P}}
\newcommand{\boldp}{{\mathbf P}}
\newcommand{\boldq}{{\mathbf Q}}
\newcommand{\bbL}{{\mathbf L}}
\newcommand{\cQ}{{\mathcal Q}}
\newcommand{\cO}{{\mathcal O}}
\newcommand{\Oo}{{\mathcal O}}
\newcommand{\cY}{{\mathcal Y}}
\newcommand{\OX}{{\Oo_X}}
\newcommand{\OY}{{\Oo_Y}}
\newcommand{\otY}{{\underset{\OY}{\ot}}}
\newcommand{\otX}{{\underset{\OX}{\ot}}}
\newcommand{\cU}{{\mathcal U}}\newcommand{\cX}{{\mathcal X}}
\newcommand{\cW}{{\mathcal W}}
\newcommand{\boldz}{{\mathbf Z}}
\newcommand{\qgr}{\operatorname{q-gr}}
\newcommand{\gr}{\operatorname{gr}}
\newcommand{\rk}{\operatorname{rk}}
\newcommand{\SH}{{\underline{\operatorname{Sh}}}}
\newcommand{\End}{\operatorname{End}}
\newcommand{\uEnd}{\underline{\operatorname{End}}}
\newcommand{\Hom}{\operatorname{Hom}}
\newcommand{\uHom}{\underline{\operatorname{Hom}}}
\newcommand{\uHomY}{\uHom_{\OY}}
\newcommand{\uHomX}{\uHom_{\OX}}
\newcommand{\Ext}{\operatorname{Ext}}
\newcommand{\bExt}{\operatorname{\bf{Ext}}}
\newcommand{\Tor}{\operatorname{Tor}}

\newcommand{\inv}{^{-1}}
\newcommand{\airtilde}{\widetilde{\hspace{.5em}}}
\newcommand{\airhat}{\widehat{\hspace{.5em}}}
\newcommand{\nt}{^{\circ}}
\newcommand{\del}{\partial}

\newcommand{\supp}{\operatorname{supp}}
\newcommand{\GK}{\operatorname{GK-dim}}
\newcommand{\hd}{\operatorname{hd}}
\newcommand{\id}{\operatorname{id}}
\newcommand{\res}{\operatorname{res}}
\newcommand{\lrar}{\leadsto}
\newcommand{\im}{\operatorname{Im}}
\newcommand{\HH}{\operatorname{H}}
\newcommand{\TF}{\operatorname{TF}}
\newcommand{\Bun}{\operatorname{Bun}}

\newcommand{\F}{\mathcal{F}}
\newcommand{\Ff}{\mathbb{F}}
\newcommand{\nthord}{^{(n)}}
\newcommand{\Gr}{{\mathfrak{Gr}}}

\newcommand{\Fr}{\operatorname{Fr}}
\newcommand{\GL}{\operatorname{GL}}
\newcommand{\gl}{\mathfrak{gl}}
\newcommand{\SL}{\operatorname{SL}}
\newcommand{\ff}{\footnote}
\newcommand{\ot}{\otimes}
\def\Ext{\operatorname {Ext}}
\def\Hom{\operatorname {Hom}}
\def\Ind{\operatorname {Ind}}
\def\bbZ{{\mathbb Z}}

\newcommand{\nc}{\newcommand}
\nc{\ol}{\overline} \nc{\cont}{\on{cont}} \nc{\rmod}{\on{mod}}
\nc{\Mtil}{\widetilde{M}} \nc{\wb}{\overline} 
\nc{\wh}{\widehat}  \nc{\mc}{\mathcal}
\nc{\mbb}{\mathbb}  \nc{\K}{{\mc K}} \nc{\Kx}{{\mc K}^{\times}}
\nc{\Ox}{{\mc O}^{\times}} \nc{\unit}{{\bf \on{unit}}}
\nc{\boxt}{\boxtimes} \nc{\xarr}{\stackrel{\rightarrow}{x}}

\nc{\Ga}{\G_a}
 \nc{\PGL}{{\on{PGL}}}
 \nc{\PU}{{\on{PU}}}

\nc{\h}{{\mathfrak h}} \nc{\kk}{{\mathfrak k}}
 \nc{\Gm}{\G_m}
\nc{\Gabar}{\wb{\G}_a} \nc{\Gmbar}{\wb{\G}_m} \nc{\Gv}{G^\vee}
\nc{\Tv}{T^\vee} \nc{\Bv}{B^\vee} \nc{\g}{{\mathfrak g}}
\nc{\gv}{{\mathfrak g}^\vee} \nc{\BRGv}{\on{Rep}\Gv}
\nc{\BRTv}{\on{Rep}T^\vee}
 \nc{\Flv}{{\mathcal B}^\vee}
 \nc{\TFlv}{T^*\Flv}
 \nc{\Fl}{{\mathfrak Fl}}
\nc{\BRR}{{\mathcal R}} \nc{\Nv}{{\mathcal{N}}^\vee}
\nc{\St}{{\mathcal St}} \nc{\ST}{{\underline{\mathcal St}}}
\nc{\Hec}{{\bf{\mathcal H}}} \nc{\Hecblock}{{\bf{\mathcal
H_{\alpha,\beta}}}} \nc{\dualHec}{{\bf{\mathcal H^\vee}}}
\nc{\dualHecblock}{{\bf{\mathcal H^\vee_{\alpha,\beta}}}}
\newcommand{\ramBun}{{\bf{Bun}}}
\newcommand{\ramBuno}{\ramBun^{\circ}}

\nc{\Buntheta}{{\bf Bun}_{\theta}} \nc{\Bunthetao}{{\bf
Bun}_{\theta}^{\circ}} \nc{\BunGR}{{\bf Bun}_{G_\BR}}
\nc{\BunGRo}{{\bf Bun}_{G_\BR}^{\circ}}
\nc{\HC}{{\mathcal{HC}}}
\nc{\risom}{\stackrel{\sim}{\to}} \nc{\Hv}{{H^\vee}}
\nc{\bS}{{\mathbf S}}
\def\BRep{\operatorname {Rep}}
\def\Conn{\operatorname {Conn}}

\nc{\Vect}{{\operatorname{Vect}}}
\nc{\Hecke}{{\operatorname{Hecke}}}

\newcommand{\ZZ}{{Z_{\bullet}}}
\nc{\HZ}{{\mc H}\ZZ} \nc{\eps}{\epsilon}

\nc{\CN}{\mathcal N} \nc{\BA}{\mathbb A}

 \nc{\BB}{\mathbb B}

\nc{\ul}{\underline}

\nc{\bn}{\mathbf n} \nc{\Sets}{{\on{Sets}}} \nc{\Top}{{\on{Top}}}
\nc{\IntHom}{{\mathcal Hom}}

\nc{\Simp}{{\mathbf \Delta}} \nc{\Simpop}{{\mathbf\Delta^\circ}}

\nc{\Cyc}{{\mathbf \Lambda}} \nc{\Cycop}{{\mathbf\Lambda^\circ}}

\nc{\Mon}{{\mathbf \Lambda^{mon}}}
\nc{\Monop}{{(\mathbf\Lambda^{mon})\circ}}

\nc{\Aff}{{\on{Aff}}} \nc{\Sch}{{\on{Sch}}}

\nc{\bul}{\bullet}
\nc{\module}{{\operatorname{-mod}}}

\nc{\dstack}{{\mathcal D}}

\nc{\BL}{{\mathbb L}}

\nc{\BD}{{\mathbb D}}

\nc{\BR}{{\mathbb R}}

\nc{\BT}{{\mathbb T}}

\nc{\SCA}{{\mc{SCA}}}
\nc{\DGA}{{\mc DGA}}

\nc{\DSt}{{DSt}}

\nc{\lotimes}{{\otimes}^{\mathbf L}}

\nc{\bs}{\backslash}

\nc{\Lhat}{\widehat{\mc L}}

\newcommand{\Coh}{\on{Coh}}

\nc{\QCoh}{\operatorname{QCoh}}
\nc{\QC}{QC}
\nc{\Perf}{\on{Perf}}
\nc{\Cat}{{\on{Cat}}}
\nc{\dgCat}{{\on{dgCat}}}
\nc{\bLa}{{\mathbf \Lambda}}

\nc{\BRHom}{\mathbf{R}\hspace{-0.15em}\on{Hom}}
\nc{\BREnd}{\mathbf{R}\hspace{-0.15em}\on{End}}
\nc{\colim}{\on{colim}}
\nc{\oo}{\infty}
\nc{\Mod}{\on{Mod} }

\nc\fh{\mathfrak h}
\nc\al{\alpha}
\nc\la{\alpha}
\nc\BGB{B\bs G/B}
\nc\QCb{QC^\flat}
\nc\qc{\mathit{qc}}

\def\w{\wedge}
\nc{\vareps}{\varepsilon}

\nc{\fg}{\mathfrak g}

\nc{\Map}{\on{Map}} \nc{\fX}{\mathfrak X}

\nc{\ch}{\check}
\nc{\fb}{\mathfrak b} \nc{\fu}{\mathfrak u} \nc{\st}{{st}}
\nc{\fU}{\mathfrak U}
\nc{\fZ}{\mathfrak Z}
\nc{\fB}{\mathfrak B}

 \nc\fc{\mathfrak c}
 \nc\fs{\mathfrak s}

\nc\fk{\mathfrak k} \nc\fp{\mathfrak p}
\nc\fq{\mathfrak q}

\nc{\BRP}{\mathbf{RP}} \nc{\rigid}{\text{rigid}}
\nc{\glob}{\text{glob}}

\nc{\cI}{\mathcal I}

\nc{\La}{\mathcal L}

\nc{\quot}{/\hspace{-.25em}/}

\nc\aff{\mathit{aff}}
\nc\BS{\mathbb S}

\nc\Loc{{\mc Loc}}
\nc\Tr{{\on{Tr}}}
\nc\Ch{{\mc Ch}}

\nc\ftr{{\mathfrak {tr}}}
\nc\fM{\mathfrak M}

\nc\Id{\operatorname{Id}}

\nc\bimod{\on{-bimod}}

\nc\ev{\operatorname{ev}}
\nc\coev{\operatorname{coev}}

\nc\pair{\operatorname{pair}}
\nc\kernel{\operatorname{kernel}}

\nc\Alg{\operatorname{Alg}}

\nc\init{\emptyset_{\text{\em init}}}
\nc\term{\emptyset_{\text{\em term}}}

\nc\Ev{\on{Ev}}
\nc\Coev{\on{Coev}}

\nc\es{\emptyset}
\nc\m{\text{\it min}}
\nc\M{\text{\it max}}
\nc\cross{\text{\it cr}}
\nc\tr{\on{tr}}

\nc\perf{\on{-perf}}
\nc\inthom{\mathcal Hom}
\nc\intend{\mathcal End}

\newcommand{\Sh}{\mathit{Sh}}

\nc{\Comod}{\on{Comod}}
\nc{\cZ}{\mathcal Z}

\def\interiorsymbol {\on{int}}

\nc\frakf{\mathfrak f}
\nc\fraki{\mathfrak i}
\nc\frakj{\mathfrak j}
\nc\BP{\mathbb P}
\nc\stab{st}
\nc\Stab{St}

\nc\fN{\mathfrak N}
\nc\fT{\mathfrak T}
\nc\fV{\mathfrak V}

\nc\Ob{\on{Ob}}

\nc\fC{\mathfrak C}
\nc\Fun{\on{Fun}}

\nc\Null{\on{Null}}

\nc\BC{\mathbb C}

\nc\loc{\on{Loc}}

\nc\hra{\hookrightarrow}
\nc\fL{\mathfrak L}
\nc\R{\mathbb R}
\nc\CE{\mathcal E}

\nc\sK{\mathsf K}
\nc\sL{\mathsf L}
\nc\sC{\mathsf C}

\nc\Cone{\mathit Cone}

\nc\fY{\mathfrak Y}
\nc\fe{\mathfrak e}
\nc\ft{\mathfrak t}

\nc\wt{\widetilde}
\nc\inj{\mathit{inj}}
\nc\surj{\mathit{surj}}

\nc\Path{\mathit{Path}}
\nc\Set{\mathit{Set}}
\nc\Fin{\mathit{Fin}}

\nc\cyc{\mathit{cyc}}

\nc\per{\mathit{per}}

\nc\sym{\mathit{symp}}
\nc\con{\mathit{cont}}
\nc\gen{\mathit{gen}}
\nc\str{\mathit{str}}
\nc\rsdl{\mathit{res}}
\nc\impr{\mathit{impr}}
\nc\rel{\mathit{rel}}
\nc\pt{\mathit{pt}}
\nc\naive{\mathit{nv}}
\nc\forget{\mathit{For}}

\nc\sH{\mathsf H}
\nc\sW{\mathsf W}
\nc\sE{\mathsf E}
\nc\sP{\mathsf P}
\nc\sB{\mathsf B}
\nc\sS{\mathsf S}
\nc\fH{\mathfrak H}
\nc\fP{\mathfrak P}
\nc\fW{\mathfrak W}
\nc\fE{\mathfrak E}
\nc\sx{\mathsf x}
\nc\sy{\mathsf y}

\nc\ord{\mathit{ord}}

\nc\sm{\mathit{sm}}

\nc\rhu{\rightharpoonup}
\nc\dirT{\mathcal T}
\nc\dirF{\mathcal F}
\nc\link{\mathit{link}}
\nc\cT{\mathcal T}

\newcommand{\ssupp}{\mathit{ss}}
\newcommand{\cyl}{\mathit{Cyl}}
\newcommand{\ball}{\mathit{B(x)}}

 \nc\ssf{\mathsf f}
 \nc\ssg{\mathsf g}
\nc\sq{\mathsf q}
 \nc\sQ{\mathsf Q}
 \nc\sR{\mathsf R}

\nc\fa{\mathfrak a}
\nc\fA{\mathfrak A}

\nc\trunc{\mathit{tr}}
\nc\pre{\mathit{pre}}
\nc\expand{\mathit{exp}}

\nc\Sol{\mathit{Sol}}
\nc\direct{\mathit{dir}}

\nc\out{\mathit{out}}
\nc\Morse{\mathit{Morse}}
\nc\arb{\mathit{arb}}
\nc\prearb{\mathit{pre}}

\nc\BZ{\mathbb Z}
\nc\proper{\mathit{prop}}
\nc\torsion{\mathit{tors}}
\nc\Perv{\mathit{Perv}}
\nc\IC{\operatorname{IC}}

\nc\Conv{\operatorname{Conv}}
\nc\Span{\operatorname{Span}}

\nc\image{\operatorname{image}}
\nc\lin{\mathit{lin}}

\nc\inverse{\operatorname{inv}}

\nc\real{\operatorname{Re}}
\nc\imag{\operatorname{Im}}

\nc\fw{\mathfrak w}

\nc\sign{\on{sgn}}
\nc\conic{\mathit{con}}

\nc\even{\mathit{ev}}
\nc\odd{\mathit{odd}}
\nc\add{\mathit{add}}

\nc\orient{\mathit{or}}

\nc\op{\mathit{op}}
\nc\sing{\mathit{sing}}
\nc\MF{\operatorname{MF}}
\nc\grade{\mathit{gr}}
\nc\periodic{\mathit{per}}

\nc\fdMod{\operatorname{fdMod}}
\nc\arbor{\mathit{arb}}

\nc\un{\diamondsuit}
\nc\BW{\mathbb W}
\nc\BH{\mathbb H}
\nc\sfW{\mathsf W}

\nc\dgcat{\on{dgcat}}
\nc\dgst{\on{dgst}}
\nc\dgSt{\on{dgSt}}

\nc\segment{\mathit{seg}}
\nc\ray{\mathit{ray}}

\nc\crit{\mathit{Cr}}
\nc\skel{\on{Skel}}

\nc\imaginary{\mathit{im}}
\nc\realabbrev{\mathit{re}}
\nc\Log{\on{Log}}

\nc\dist{\on{dist}}

\title[Wrapped microlocal sheaves on pairs of pants]{Wrapped microlocal sheaves on pairs of pants}

\author{David Nadler}
\address{Department of Mathematics\\University of California, Berkeley\\Berkeley, CA  94720-3840}
\email{nadler@math.berkeley.edu}

\begin{abstract}
Inspired by the geometry of wrapped Fukaya categories, 
we introduce the notion of wrapped microlocal sheaves.
We show that traditional microlocal sheaves are equivalent to functionals on wrapped microlocal sheaves,
in analogy with the expected relation of infinitesimal to wrapped Fukaya categories. 
As an application, we calculate 
wrapped microlocal sheaves 
on higher-dimensional pairs of pants,  confirming expectations from  mirror symmetry.  
\end{abstract}

\maketitle


\tableofcontents


\section{Introduction}

The aim of this paper is twofold: to introduce a notion of wrapped microlocal sheaves parallel to wrapped
Fukaya categories as introduced by Abouzaid-Seidel~\cite{AS} and studied by Auroux~\cite{aurouxpw}, and to establish a homological mirror symmetry equivalence for higher-dimensional pairs of pants extending
the results of 
Seidel~\cite{seidelgenustwo}, Sheridan~\cite{sheridan}, and the collaboration~\cite{AAEKO}. 
More speculatively, we are also motivated by the expectation that wrapped microlocal sheaves will offer a good model for $A$-branes
in a Betti version of Geometric Langlands.

Wrapped microlocal sheaves  arise naturally
from both  geometric and categorical considerations, and  consequently enjoy many appealing features.
Most notably,  traditional microlocal sheaves are equivalent to functionals on wrapped microlocal sheaves
(see Theorem~\ref{thm:introduality} below)
in analogy with the expected relation of infinitesimal to wrapped Fukaya categories,
or the relation of perfect complexes with compact support to coherent sheaves~\cite{BNP}.
Returning to a more classical setting, one might also keep in mind the identification of compactly-supported cohomology
with functionals on Borel-Moore homology.

To provide a collection of simple examples of the theory, we
first calculate wrapped microlocal sheaves on exact symplectic surfaces
equipped with skeleta. In particular,  in the case of punctured spheres, we obtain versions
(see Theorem~\ref{thm:introsurface} below) of the mirror symmetry results of~\cite{AAEKO}.
We then turn to our main application
and calculate wrapped microlocal sheaves  supported along natural skeleta within
higher-dimensional  pairs of pants.
We establish their equivalence 
 with the expected Landau-Ginzburg $B$-models
in their guise as matrix factorizations (see Theorem~\ref{thm:intromirror} below).
This can be viewed as a generalization, in the language of microlocal sheaves, of  results of~\cite{AAEKO}
to higher dimensions, and the results of \cite{seidelgenustwo, sheridan} from compact to arbitrary branes.
It also provides the ``opposite direction"
of homological mirror symmetry to that undertaken in~\cite{Nlg3d, Nlg}.


\subsection{Wrapped microlocal sheaves}

Roughly speaking, to an exact symplectic manifold $M$ equipped with
 a Lagrangian  skeleton $L$,
 there are associated two versions
of the Fukaya category: 
the infinitesimal Fukaya-Seidel category~\cite{Seidel} with
Lagrangian branes in $M$ running along $L$, 
and the wrapped Fukaya category~\cite{AS}
with Lagrangian branes  in $M$ transverse to $L$. 
(When the skeleton is noncompact, the wrapped  variant is often called
 partially wrapped~\cite{aurouxpw}.)

There is a broad expectation that the infinitesimal Fukaya-Seidel category may be modeled by microlocal sheaves on $M$ supported along $L$. While there is not yet a general account of such an equivalence, there is a convincing 
and rapidly growing body of evidence, originally motivated by the intersection theory of Lagrangian cycles, and including the far from exhaustive list of references~\cite{B, ccc, Nequiv, NZ, tamarkin}. (We include below a brief review of microlocal sheaves, but for further details, all roads lead to the pioneering work  of Kashiwara-Schapira~\cite{KS}.)

In this paper, we propose a notion of wrapped microlocal sheaves 
which we   can  verify in many situations similarly models
  the wrapped Fukaya category.
%
%
We  primarily focus on the traditional microlocal setting of conic open subspaces $\Omega \subset T^*Z$ of the cotangent bundle
of  a real analytic manifold $Z$.
This suffices for the study of wrapped microlocal sheaves on exact symplectic manifolds $M$
that arise from
  such conic open subspaces $\Omega \subset T^*Z$ by Hamiltonian reduction. 
  In particular, it is the approach we take to calculate 
  wrapped microlocal sheaves on
  higher-dimensional pairs of pants.
  We also calculate wrapped microlocal sheaves on exact symplectic surfaces
  equipped with skeleta
   where a more general approach
  via the gluing of local constructions  is not too involved and fits within the scope of this paper.

Throughout, we fix an algebraically closed   field $k$ of characteristic zero (though more general settings are possible).
We work with 
  $k$-linear differential graded (dg) categories,
$k$-linear differential $\BZ/2$-graded ($\BZ/2$-dg) categories,
 and derived functors, though our language may not explicitly reflect this.
For example, by a $k$-module, we will mean a dg $k$-module, and by a perfect $k$-module, we will mean a dg $k$-module with finite-dimensional cohomology.

%
%
%


Given the cotangent bundle $T^*Z$ of
of  a real analytic manifold $Z$,
let us begin by listing some  prominent features   wrapped microlocal sheaves enjoy. 
(We recommend the analogy with coherent sheaves discussed in Remark~\ref{rem:introanalogy}
as an organizing framework.)
Fix a closed conic Lagrangian subvariety $\Lambda \subset T^*Z$ which we will refer to as a  {\em support Lagrangian}.
To  a conic open subspace $\Omega \subset T^*Z$, there is a dg category $\mu\Sh^w_{\Lambda}(\Omega)$
  of wrapped microlocal sheaves on $\Omega$ supported along $\Lambda$. 
 Given an inclusion of conic open subspaces $\Omega' \subset \Omega$,
  there is a natural   corestriction functor $\mu\Sh^w_{\Lambda}(\Omega')\to \mu\Sh^w_{\Lambda}(\Omega)$.
  These assignments assemble into a cosheaf 
  $\mu\Sh^w_{\Lambda}$
 of dg categories supported along $\Lambda$. 
  There exists a stratification of $\Lambda$ such that the restriction of $\mu\Sh^w_{\Lambda}$  to each stratum is locally constant. The stalk of $\mu\Sh^w_{\Lambda}$ at a smooth point of $\Lambda$ is (not necessarily canonically) equivalent to perfect $k$-modules.
   Given a closed embedding of support Lagrangians $\Lambda\subset \Lambda'$,
  there is a natural localization $\mu\Sh_{\Lambda'} \to \mu\Sh_{\Lambda}$.

\begin{example}\label{ex:introlocsys}
For the zero-section $\Lambda = Z$ and the entire cotangent bundle $\Omega = T^*Z$, 
one can interpret $\mu\Sh^w_{Z}(T^*Z)$ as a dg category of certain locally constant sheaves. More precisely, it is  equivalent
to the dg category of perfect modules over chains on the Poincar\'e $\oo$-groupoid of~$Z$.
In particular, if $Z$ is connected, it is equivalent 
to the dg category of perfect modules over chains on the based loop space of $Z$. 
(One could compare with the parallel result for the wrapped Fukaya category proved by Abouzaid~\cite{abouzaid}.)
In contrast, the traditional dg category of local systems is  
equivalent
to the dg category of those modules over chains on the Poincar\'e $\oo$-groupoid of~$Z$ whose underlying $k$-module
is perfect.\end{example}

\begin{example}\label{ex:lgsingthimble}
In the papers~\cite{Nlg3d, Nlg}, we calculated the Landau-Ginzburg 
$A$-model $M= \BC^n, W= z_1\cdots z_n$ 
taking its branes in the form of microlocal sheaves.
In the most basic formulation, we took as support Lagrangian 
the  ``singular thimble" $L_c = \Cone(T^{n-1}) \subset M$ given by the cone over a compact torus
in a regular fiber.
We established a mirror equivalence between traditional microlocal sheaves
 on $M$ supported along $L_c$ and   perfect complexes with proper support on the $(n-2)$-dimensional pair of pants $P_{n-2}$.
  If one repeats  the arguments of ~\cite{Nlg3d, Nlg} with 
  the wrapped microlocal sheaves of this paper, one 
will arrive at a mirror equivalence
with all coherent complexes on $P_{n-2}$.
\end{example}

The essential nature of wrapped microlocal sheaves is perhaps best captured by their relation with traditional
microlocal sheaves.
To  pursue this, to a conic open subspace $\Omega \subset T^*Z$, let us denote by  $\mu\Sh^\un_{\Lambda}(\Omega)$
the cocomplete dg category 
  of {\em large microlocal sheaves} on $\Omega$ whose microstalks are arbitrary $k$-modules supported along $\Lambda$. 
Thus we require objects of $\mu\Sh^\un_{\Lambda}(\Omega)$ to be geometrically tame in the sense that their singular support lies within $\Lambda$, but we do not impose any algebraic restriction on their size.

We will focus on  two natural ways to cut out a small dg subcategory of  $\mu\Sh^\un_{\Lambda}(\Omega)$
 by imposing finiteness conditions.

First, there is the full dg subcategory 
$$
\xymatrix{
\mu\Sh_{\Lambda}(\Omega)\subset \mu\Sh^\un_{\Lambda}(\Omega)
}
$$
   of objects whose microstalks are perfect $k$-modules supported along $\Lambda$.    
   We refer to these as {\em traditional microlocal sheaves}, though more  technically
    they  might also be termed microlocally constructible. They are the version of microlocal sheaves closely related to
    the infinitesimal Fukaya-Seidel category, and can be represented locally by constructible sheaves.

Traditional microlocal sheaves enjoy  the following contravariant versions of the features of wrapped microlocal sheaves listed above.
Given an inclusion of conic open subspaces $\Omega' \subset \Omega$,
  there is a natural   restriction functor $\mu\Sh_{\Lambda}(\Omega)\to \mu\Sh_{\Lambda}(\Omega')$.
  These assignments assemble into a sheaf 
  $\mu\Sh_{\Lambda}$
 of dg categories supported along $\Lambda$. 
  There exists a stratification of $\Lambda$ such that the restriction of $\mu\Sh_{\Lambda}$  to each stratum is locally constant. The stalk of $\mu\Sh_{\Lambda}$ at a smooth point of $\Lambda$ is (not necessarily canonically) equivalent to perfect $k$-modules.
   Given a closed embedding of support Lagrangians  $\Lambda\subset \Lambda'$,
  there is a natural fully faithful embedding $\mu\Sh_{\Lambda} \to \mu\Sh_{\Lambda'}$.

Second, we can  cut out a small dg subcategory of  $\mu\Sh^\un_{\Lambda}(\Omega)$ by imposing the standard categorical finiteness of compact objects. Recall that an  object $c\in\cC$ of a stable dg category  is compact if and only if the functor it corepresents $\Hom(c, -):\cC\to \Mod_k$ preserves coproducts. 
If $\cC$ is cocomplete, in particular it contains coproducts,  
then we may recover it as the ind-category $\cC\simeq \Ind\cC_c$ of its full dg subcategory of compact objects.

\begin{defn} \label{defn:introwrapped}
The dg category of {\em wrapped microlocal sheaves} is the full dg subcategory
$$
\xymatrix{
\mu\Sh^w_{\Lambda}(\Omega)\subset \mu\Sh^\un_{\Lambda}(\Omega)
}
$$
of compact objects within the cocomplete dg category of all microlocal sheaves. 
   \end{defn}

The above definition is useful for the clean characterization it provides, but there is an equivalent
and more geometric way to approach wrapped microlocal sheaves that better illuminates their relation to the wrapped Fukaya category.
%
%
%
%

   To explain this, let us briefly recall the geometry of microstalks.
   For simplicity, let us focus on microstalks at a generic point $(z, \xi) \in \Lambda$, where $z\in Z$ and $\xi\in T^*_z Z$, 
   so that in particular $\Lambda \subset T^*Z$ is a smooth Lagrangian submanifold near to $(z, \xi)$.
   
   Choose a small open ball $B\subset Z$ around $z\in Z$, 
   and a smooth function $f:B\to \BR$ such that $f(z) = 0$, and $df_z = \xi$. The graph $L = \Gamma_{df} \subset T^*Z$ is a small Lagrangian ball centered at $(z, \xi)$. Let us assume it intersects $\Lambda \subset T^*Z$ transversely at the single point  $(z, \xi)$. 
   
   Then for $\Omega \subset T^*Z$ a conic open subspace containing $(z, \xi)$,
   the microstalk along $L\subset \Omega$ of a microlocal sheaf $\cF \in \mu\Sh_\Lambda^\un(\Omega)$ is the vanishing cycles
  $$
 \xymatrix{
 \phi_L:\mu\Sh_\Lambda^\un(\Omega)\ar[r] & \Mod_k
}
$$
$$
\xymatrix{
 \phi_L(\cF) = \Gamma_{\{f\geq 0\}}(B, \tilde \cF|_B) = \Cone( \Gamma(B, \tilde \cF) \to \Gamma(\{f<0\}, \tilde \cF|_{\{f<0}\}))[-1]
 }
 $$
where the sheaf $\tilde \cF$ represents the restriction of the microlocal sheaf $\cF$
 to a small conic open neighborhood of $(z, \xi)$.

By abstract formalism, there is an object $\cF_L \in \mu\Sh_\Lambda^\un(\Omega)$,
which we call a {\em microlocal skyscraper},
corepresenting the microstalk in the sense of a natural equivalence
$$
\xymatrix{
 \phi_L(\cF) \simeq \Hom(\cF_L, \cF) & \cF \in \mu\Sh_\Lambda^\un(\Omega)
}
$$

We have the following geometric alternative to the above categorical definition.

\begin{lemma}[Lemma~\ref{lem:splitgen} below]
The  microlocal skyscrapers $\cF_L\in \mu\Sh_\Lambda^\un(\Omega)$  form a collection of compact generators,
and thus wrapped microlocal sheaves
$$
\xymatrix{
\mu\Sh_\Lambda^w(\Omega)
\subset \mu\Sh_\Lambda^\un(\Omega)}
$$ 
 form the full dg subcategory split-generated by the microlocal skyscrapers. 
\end{lemma}

\begin{remark}
From the above geometric characterization, we can see more clearly the relation with the wrapped Fukaya category. The microlocal skyscraper $\cF_L\in \mu\Sh_\Lambda^w(\Omega)$
corresponds to the object of the wrapped Fukaya category given by the  Lagrangian ball $L\subset \Omega$ transverse to 
 $\Lambda$ equipped with a suitable brane structure.

The morphisms between microlocal skyscrapers
involve the potentially complicated global geometry of the exact symplectic manifold
 $\Omega$
 and  support Lagrangian $\Lambda$.
This is
  in parallel with  the potentially complicated global  dynamics of the  wrapping Hamiltonian $H:\Omega\to \BR$  
  found in
  the construction of the wrapped Fukaya category. 
   \end{remark}

%

   Now let us focus on how the above versions of microlocal sheaves are related. Since 
   the full dg subcategory 
   $
\mu\Sh^w_{\Lambda}(\Omega)\subset \mu\Sh^\un_{\Lambda}(\Omega)
$
of wrapped microlocal sheaves comprises the compact objects,  we can recover all microlocal sheaves
from it  by forming the ind-category 
   $$
   \xymatrix{
   \mu\Sh^\un_\Lambda(\Omega) \simeq \Ind \mu\Sh_\Lambda^w(\Omega)
   }
   $$

   It turns out we can also recover the  small dg category
   $
\mu\Sh_{\Lambda}(\Omega)\subset \mu\Sh^\un_{\Lambda}(\Omega)
$
of traditional microlocal sheaves from wrapped microlocal sheaves (though not vice versa, see Remark~\ref{rem:introasymmetry}).
   The natural hom-pairing between
   $\mu\Sh_{\Lambda}(\Omega)$ and the opposite dg category $\mu\Sh^w_{\Lambda}(\Omega)^{op}$ lands in perfect $k$-modules.
   We prove this is in fact a perfect pairing in the following sense.
   
\begin{thm}[Theorem~\ref{thm:duality} below]\label{thm:introduality}

The natural hom-pairing provides an equivalence
$$
\xymatrix{
\mu\Sh_\Lambda(\Omega)  \ar[r]^-\sim & \Fun^{ex}(\mu\Sh_\Lambda^w(\Omega)^{op}, \Perf_k) 
}
$$
where $ \Fun^{ex}$ denotes the dg category of exact functors, and $\Perf_k$ that of perfect $k$-modules.
%
\end{thm}

\begin{remark}\label{rem:introanalogy}
With mirror symmetry in mind, we advocate the following informal analogy.

To the  conic open subspace $\Omega \subset T^*Z$
and support Lagrangian $\Lambda \subset T^*Z$, let us imagine assigning a variety $X_{\Lambda}(\Omega)$.
We can think of wrapped microlocal sheaves $\mu\Sh^w_\Lambda(\Omega)$ as coherent sheaves $\Coh(X_\Lambda(\Omega))$,
and thus all microlocal sheaves $\mu\Sh^{\un}_\Lambda(\Omega)$ as ind-coherent sheaves $\Ind\Coh(X_\Lambda(\Omega))$. In line with the results of~\cite{BNP}, 
 Theorem~\ref{thm:introduality} says we can  compatibly  think of traditional microlocal sheaves $\mu\Sh_\Lambda(\Omega)$
as perfect complexes with proper support $\Perf_\proper(X_\Lambda(\Omega))$. 

Let us further extend the analogy to  the natural functoriality in 
the  conic open subspace $\Omega \subset T^*Z$ and support Lagrangian  $\Lambda \subset T^*Z$.
To an open inclusion $\Omega' \subset \Omega$, and closed embedding $\Lambda\subset \Lambda'$,
we 
can think of assigning a correspondence
$$
\xymatrix{
X_{\Lambda}(\Omega) & \ar[l]_-p Y  \ar[r]^-q &  X_{\Lambda'}(\Omega')
}
$$
where $p$ is proper and Gorenstein, and $q$ is smooth.
 Then we can think of the natural functors $\mu\Sh^w_{\Lambda'}(\Omega') \to \mu\Sh^w_\Lambda(\Omega)$
and $\mu\Sh_\Lambda(\Omega)\to \mu\Sh_{\Lambda'}(\Omega')$ as corresponding respectively to  functors
$$
\xymatrix{
p_*q^*:\Coh( X_{\Lambda'}(\Omega')) \ar[r] & \Coh( X_{\Lambda}(\Omega)) 
&
q_*p^!:\Perf_\proper ( X_{\Lambda}(\Omega)) \ar[r] & \Perf_\proper( X_{\Lambda'}(\Omega')) 
}
$$
Note that these are the restrictions of adjoint functors on all ind-coherent sheaves.
 %


\end{remark}

%
%

\begin{remark}\label{rem:introasymmetry}
While objects of $\mu\Sh_\Lambda^w(\Omega)$ similarly give functionals on
$\mu\Sh_\Lambda(\Omega)^{op}$, it is not true that they produce all possible functionals. 
For example, take $Z=T^1$ to be the circle, $\Lambda = T^1$ the zero-section, and $\Omega = T^*T^1$ the entire cotangent bundle.
Then we have equivalences
 $$
 \xymatrix{
 \mu\Sh_{T^1}(T^*T^1) \simeq \Perf_\proper(\G_m)
 &
 \mu\Sh^w_{T^1}(T^*T^1) \simeq \Coh(\G_m)
 }
 $$
and by~\cite{BNP}, the hom-pairing gives an equivalence
$$
\xymatrix{
 \Perf_\proper(\G_m)  \ar[r]^-\sim & \Fun^{ex}(\Coh(\G_m)^{op}, \Perf_k) 
}
$$
compatibly with Theorem~\ref{thm:introduality}.
But clearly there are more functionals on $\Perf_\proper(\G_m)$ than those coming from $\Coh(\G_m)$ alone.
For example, one could take the hom-pairing with a direct sum of skyscraper sheaves at infinitely many points.  
\end{remark}

The proof  we give of Theorem~\ref{thm:introduality}
is an application of the theory of arboreal singularities developed in~\cite{Narb, Nexp}. Let us sketch the argument here. First, by abstract formalism, it suffices to prove the theorem locally in $\Lambda$, 
so we may focus on the germ  of $\Lambda$ at a  point.  Then 
applying~\cite{Nexp}, we may non-characteristically deform  the germ of $\Lambda$ to a nearby conic Lagrangian subvariety
$\Lambda_\arb$ with arboreal singularities. Thus microlocal sheaves along the germ are equivalent
to microlocal sheaves along $\Lambda_\arb$. 
Once again, by abstract formalism, it suffices to prove the theorem locally in $\Lambda_\arb$,
so we may focus further on the germ  of $\Lambda_\arb$ at a  point. 
 Now 
applying~\cite{Narb}, we find that large microlocal sheaves on the germ  of $\Lambda_\arb$  are equivalent to modules over  a directed tree~$\cT$.
Here the  perfect and coherent modules coincide, and form a smooth and proper dg category, 
and thus the assertion of the theorem holds.


\subsection{Mirror symmetry for pairs of pants}

Now let us turn to a concrete application in the setting of homological mirror symmetry.
We will introduce natural skeleta within  pairs of pants
and calculate wrapped microlocal sheaves supported along them. 
We can not immediately invoke the preceding theory since pairs of pants are not conic open subspaces of
cotangent
bundles. Thus we will first  pass to the symplectification
of the contactification of neighborhoods of their skeleta and identify these with  conic open subspaces of
cotangent
bundles.
Since wrapped microlocal sheaves are invariant under this modification, it provides a natural avenue to
the preceding theory.

Let $\BP^{n+1}_\BC = \proj \BC[z_0, \ldots, z_{n+1}]$ be complex projective space, and $T^{n+1}_\BC \subset \BP^{n+1}_\BC$
the  complex torus complementary to the coordinate hyperplanes $H_a = \{z_a = 0\} \subset  \BP^{n+1}_\BC$,
for $a = 0, \ldots, n+1$.

By the $n$-dimensional {\em pair of pants}, we will mean the smooth affine variety 
$$
\xymatrix{
P_{n} = \{1 +  z_1 + \cdots + z_{n+1}   = 0 \} \subset  T^{n+1}_\BC
}
$$
or equivalently but more symmetrically, the open complement of $n+2$ hyperplanes in general position in an $n$-dimensional projective space
  $$
\xymatrix{
P_n =  H \setminus (\bigcup_{a = 0}^{n+1} H \cap H_a)
&
H  = \{ z_0 + \cdots + z_{n+1} = 0\} \simeq \BP_\BC^{n}
}
$$  
Note that the symmetric group $\Sigma_{n+2}$ naturally acts on $P_n$ by permuting the homogeneous coordinates.
Since $P_n$ is a smooth affine variety, it
  is naturally a Stein manifold and hence a Liouville manifold,
and we will  be  interested in this latter structure.

Following Mikhalkin~\cite{mik}, we will work with the pair of pants 
in  a slightly modified form where we alter its embedding near infinity.
This modified form is particularly suited to gluing constructions,
and our results have natural extensions to smooth toric hypersurfaces
(see Remark~\ref{rem: symmetry} below for a brief discussion). 
By the $n$-dimensional {\em tailored pair of pants}. 
we will mean the Liouville manifold $Q_n\subset T^{n+1}_\BC$
constructed in~\cite[Proposition 4.6]{mik}
under the name ``localized" pair of pants,
 and recalled in Section~\ref{s: tailored} below.
It is isotopic to the usual pair of pants $P_n \subset T^{n+1}_\BC$,
and enjoys the same  permutation symmetries,
but near infinity the ends of~$Q_n$ have a technically useful conic structure
compatible with that of $T^{n+1}_\BC$.

The  mirror of the pair of pants is
 the Landau-Ginzburg $B$-model with background $\BA^{n+2}$ and superpotential the product of the coordinates
$$
\xymatrix{
W_{n+2} :\BA^{n+2} \ar[r] &  \BA^1
&
W_{n+2} = z_1 \cdots z_{n+2}
}
$$

We will establish a homological mirror equivalence taking the $A$-model in the form of wrapped microlocal sheaves along a natural skeleton $L_n \subset Q_n$, and the 
 Landau-Ginzburg $B$-model in its guise as matrix factorizations. 
 The matching of a distinguished compact $A$-brane on the pair of pants with the skyscraper $B$-brane at the origin of $\BA^{n+2}$  was accomplished by Sheridan~\cite{sheridan}, building on the groundbreaking case $n=1$ 
of Seidel~\cite{seidelgenustwo}. Going beyond compact $A$-branes, a mirror equivalence in the case $n=1$ 
 was a main result
of the collaboration~\cite{AAEKO}.

To construct  the skeleton $L_n \subset Q_n$,  
we will break symmetry as follows. (See Remark~\ref{rem: symmetry} below for a brief discussion of forthcoming work returning full symmetry to the picture.)
Introduce the compact torus $T^{n+1} \subset T^{n+1}_\BC$, and the moment map for its
natural  Hamiltonian translation action
$$
\xymatrix{
\Log_{n+1}:T^{n+1}_\BC \ar[r] & \BR^{n+1}
&
\Log_{n+1}(z_1, \ldots, z_{n+1}) = (\log|z_1|, \ldots, \log|z_{n+1}|)
}
$$
The amoeba $\Log_{n+1}(Q_n) \subset \BR^{n+1}$ retracts onto its tropical spine, and its complement consists
of the disjoint union of $n+2$ open contractible domains. 
Fix a point $x_\ell= (-\ell, \ldots, -\ell)\in \BR^{n+1} \setminus \Log_{n+1}(Q_n)$ sufficiently far from the amoeba,
and consider the Liouville form 
$$
\xymatrix{
\beta_{{Q_n}} =  \sum_{a = 1}^{n+1} (\xi_a + \ell) d\theta_a |_{Q_n}
}
$$
expressed in polar coordinates $z_a = e^{\xi_a + i\theta_a}$, for $a= 1, \ldots, n+1$.
The corresponding Liouville vector field is gradient-like for the Morse-Bott function
given by the squared-distance 
$$
\xymatrix{
 |\Log_{n+1} - x_\ell|^2 :Q_n \ar[r] &  \BR
}
$$

We will work with the skeleton $L_n \subset Q_n$ given by the union of stable manifolds
for the resulting Liouville flow.
All of the above constructions including the skeleton $L_n \subset Q_n$ are invariant under the subgroup $\Sigma_{n+1} \subset \Sigma_{n+2}$ permuting the coordinates of $T^{n+1}_\BC$ and fixing the $0$th  homogeneous coordinate of $\BP^{n+1}_\BC$.
To describe the geometry of a neighborhood of $L_n \subset Q_n$,
let us outline some further  constructions. 

Let $T^1_\Delta \subset T^{n+1}$  be the diagonal compact torus,
and introduce the quotient $\BT^n \simeq T^{n+1}/T^1_\Delta$.
For any $\chi \in \BR$, we have  a Hamiltonian reduction correspondence
$$
\xymatrix{
T^*T^{n+1} & \ar@{_(->}[l]_-{q_\chi} \mu^{-1}_\Delta(\chi) \ar@{->>}[r]^-{p_\chi}
&
T^*\BT^n
}
$$
where $ \mu_\Delta: T^* T^{n+1} \to \BR$ is the moment map for the natural $T^1_\Delta$-action by translations.
In particular, when $\chi= 0$, we recover the usual Hamiltonian reduction correspondence
$$
\xymatrix{
T^*T^{n+1} & \ar@{_(->}[l]_-{q} T^*_{T^1_\Delta} T^{n+1} \ar@{->>}[r]^-{p}
&
T^*\BT^n
}
$$

Introduce the conic Lagrangian subvariety
$$
\xymatrix{
\Lambda_{1} = \{(\theta, 0) \, | \, \theta\in T^1\} \cup \{(0, \xi) \, |\, \xi \in \BR_{\geq 0}\} \subset T^1 \times \BR \simeq
T^*T^1
}$$
and the product conic Lagrangian subvariety
$$
\xymatrix{
\Lambda_{n+1} = (\Lambda_1)^{n+1} \subset (T^* T^1)^{n+1} \simeq T^* T^{n+1}
}$$

For $\chi>0$, define the Lagrangian subvariety
$$
\xymatrix{
\fL_n = p_\chi(q_\chi^{-1}( \Lambda_{n+1}) \subset T^*\BT^n
}
$$

Let $\ft_n^* =\{\sum_{a=1}^{n+1} \xi_a = 0 \} \subset \BR^{n+1}$ be the dual of the Lie algebra of $\BT^n$.
Under the natural moment projection 
$$
\xymatrix{
T^* \BT^n \simeq  \BT^n \times \ft_n^* \ar[r] & \ft_n^*
}
$$
the image of $\fL_n \subset T^*\BT^n$
 is a closed simplex $\Xi_n \subset \BR^n$,
and  we have a simple combinatorial description
$$
\xymatrix{
 \fL_n = \bigcup_{I} \BT^I \times \Xi_I \subset \BT^n \times \ft_n^*
 \simeq  T^* \BT^n 
}
$$ 
where each proper subset $I \subset \{1, \ldots, n+1\}$ naturally indexes  a relatively open subsimplex $\Xi_I \subset \Xi_n$
and an orthogonal subtorus $\BT^I \subset \BT^n$.

Now we can describe the geometry of a neighborhood of the skeleton $L_n \subset Q_n$.

\begin{thm}[Theorem~\ref{thm: symplecto} below]\label{introthm: symplecto}
Fix $\chi>0$.
There is an open neighborhood $U_n \subset Q_n$ of the skeleton $L_n \subset Q_n$,
 an open neighborhood $\fU_n \subset T^*\BT^n$ of the Lagrangian subvariety $\fL_n \subset T^* \BT^n$,
and 
a symplectomorphism
$$
\xymatrix{
\mathfrak j:U_n \ar[r]^-\sim & \fU_n 
}
$$
restricting to an isomorphism
$$
\xymatrix{
\mathfrak j |_{L_n}:L_n \ar[r]^-\sim & \fL_{n}
}
$$
\end{thm}
%
%

Now let us turn to wrapped microlocal sheaves on
the tailored pair of pants  $Q_n$ supported along the skeleton $L_n$.
Thanks to Theorem~\ref{introthm: symplecto}, we may equivalently study 
wrapped microlocal sheaves on
the cotangent bundle  $T^* \BT^n$ supported along the Lagrangian subvariety $\fL_n$.
But we  must be careful: under Theorem~\ref{introthm: symplecto}, the Liouville form on 
$Q_n$ does not match the 
canonical Liouville form on $T^*\BT^n$.
In particular,  while the skeleton $L_n$ is exact, the Lagrangian subvariety $\fL_n$ is not exact much less conic.
We must take take one further step to transport the exact symplectic geometry of $Q_n$ into that of a cotangent bundle.
To achieve this, we will take advantage of the basic fact that microlocal sheaves are invariant under the operation of
 forming the symplectification of the contactification and lifting support Lagrangians.

On the one hand,  let us form the 
 circular contactification $Q_n \times T^1$, and then its symplectification 
$\tilde Q_n = Q_n \times T^1 \times \BR$,
with their natural projections
 $$
 \xymatrix{
\tilde Q_n = Q_n \times T^1 \times \BR \ar[r]^-s &   Q_n \times T^1 \ar[r]^-c &  Q_n
 }
 $$
 The skeleton $L_n \subset Q_n$ lifts under $c$ to the Legendrian subvariety $L_n \times\{0\}\subset  Q_n\times S^1$, and we can take its inverse-image under $s$ to obtain a conic Lagrangian subvariety
  $$
  \xymatrix{
  \tilde L_n = s^{-1}(L_n \times\{0\}) \subset\tilde Q_n
  }
  $$

 On the other hand,
 introduce the conic open subspace 
 $$
 \xymatrix{
 \Omega_{n+1} = \mu_\Delta^{-1}(\BR_{>0}) = \{\sum_{a=1}^{n+1} \xi_a > 0\} \subset T^*T^{n+1}
 }$$
and recall the conic Lagrangian subvariety $\Lambda_{n+1} \subset T^* T^{n+1}$.
 
Then Theorem~\ref{introthm: symplecto} lifts to the following.

\begin{thm}[Therorem~\ref{thm: contacto} below]
Fix $\chi=n+1$.
There is a conic open neighborhood $\tilde U_n \subset \tilde Q_n$ of
the  Lagrangian subvariety  $\tilde L_n  \subset \tilde Q_n$,
 a conic open neigborhood $\Upsilon_{n+1} \subset \Omega_{n+1}$
 of
the  intersection  $\Lambda_{n+1} \cap \Omega_{n+1}$,
 and an exact symplectomorphism
$$
\xymatrix{
 \tilde{\mathfrak \jmath}:\tilde U_n\ar[r]^-\sim & \Upsilon_{n+1} 
 }
$$
restricting to an isomorphism
$$
\xymatrix{
 \tilde{\mathfrak \jmath} |_{\tilde L_n}:\tilde L_n \ar[r]^-\sim & \Lambda_{n+1} \cap \Omega_{n+1}
}
$$
\end{thm}

Now we can access wrapped microlocal  sheaves 
by passing to the conic open subspace $\Upsilon_{n+1}$
and support Lagrangian $\Lambda_{n+1}$. Note that the precise shape of $\Upsilon_{n+1}$
is not important, and we may indeed replace it with the more explicit $\Omega_{n+1}$.
Since wrapped microlocal sheaves along  $\Lambda_{n+1}$
form a cosheaf  along  $\Lambda_{n+1}$, and $\Upsilon_{n+1}$
is a conic open neighborhood of  $\Lambda_{n+1} \cap \Omega_{n+1}$,
their evaluation over $\Upsilon_{n+1}$
is equal to their evaluation over $\Omega_{n+1}$.
 
 \begin{ansatz}\label{pofpansatz}
  Set the dg category $\mu\Sh_{L_n}^w(Q_n)$ of wrapped microlocal sheaves on the 
  tailored pair of pants $Q_n$
supported along the skeleton $L_n$ to be  that of wrapped
microlocal sheaves on $\Omega_{n+1}$ supported  along
 $\Lambda_{n+1}$
in the sense of Definition~\ref{defn:introwrapped}.
\end{ansatz}

\begin{remark}
Following established  patterns in the subject (for example~\cite{kquant, KSdq, PS}), 
we expect it is possible  to develop a general definition of
a $\BZ/2$-dg category of  wrapped microlocal sheaves
on an exact symplectic manifold $M$ supported along  an exact Lagrangian $L$.
Roughly speaking, one should pass to the symplectification of the contactification of $M$,
choose local identifications of it with conic open subspaces of cotangent bundles, and
then glue together local dg categories of wrapped microlocal sheaves
using  the theory of contact transformations.  
(It is possible the technical demands of such a construction are already available in the literature, but  we have not
attempted to thoroughly understand the state of the art.)
 
 In the relatively simple case of symplectic surfaces, we directly check the required invariance 
 of wrapped microlocal sheaves under such choices of   local identifications and thus establish a general definition
 (see Section~\ref{s:introsurf} below). 
In higher dimensions, 
we will not pursue the details of a general definition here, but 
focus on the situation of Ansatz~\ref{pofpansatz} where 
 the symplectification of the contactification of $M$ is itself a conic open subspace $\Omega$ of a cotangent bundle.
 In this case, any general definition of wrapped microlocal sheaves on $M$ will agree with wrapped microlocal sheaves on
  $\Omega$ in the sense of Definition~\ref{defn:introwrapped}.
   (For a basic example of this consistency, see Remark~\ref{rem:unambiguous} below.)
\end{remark}

Finally, to calculate the dg category $\mu\Sh_{L_n}^w(Q_n)$ of wrapped microlocal sheaves, we proceed as follows. We begin by verifying the elementary mirror equivalence 
$$
\xymatrix{
\mu\Sh^w_{\Lambda_1}(T^* T^1) \simeq \Coh(\BA^1)
}
$$
along with straightforward  analogues of it in higher dimensions.
Here  already is an instance where the wrapped theory,
as opposed to the traditional,
can be appreciated for providing an object mirror to the structure sheaf of the non-proper variety $\BA^1$.
Then we invoke that wrapped microlocal sheaves form a cosheaf to reduce their global calculation
to the colimit of a diagram of such building blocks. Finally, we match this diagram with a natural descent diagram 
whose colimit calculates matrix factorizations.

We arrive at the anticipated mirror equivalence. 

\begin{thm}[Corollary~\ref{cor:mirrorsym} below]\label{thm:intromirror}
There is an  equivalence of  $\BZ/2$-dg categories
$$
\xymatrix{
\mu\Sh_{L_n}^w(Q_n)_{\BZ/2} \simeq \MF(\BA^{n+2}, W_{n+2})
}
$$
between  the underlying  $\BZ/2$-dg category
of wrapped microlocal sheaves on the tailored pair of pants $Q_n$ supported along the 
 skeleton $L_n$,
and  that of matrix factorizations for the superpotential  $W_{n+2} = z_1\cdots z_{n+2} :\BA^{n+2}\to \BA^1$.
\end{thm}

\begin{remark}\label{rem: symmetry}
The invariance of microlocal sheaves under mutations of support Lagrangians
is an important general question.
Let us briefly mention how it plays a prominent role in gluing together microlocal sheaves on copies of $Q_n$  into
microlocal sheaves on smooth toric hypersurfaces.
First, by applying elements of the symmetric group $\Sigma_{n+2}$ to the initial skeleton~$L_n$, one obtains  a collection of alternative skeleta with
analogous structure. Each of these skeleta is  well suited to gluings of $Q_n$ in different directions along its ends.  
To pursue such gluings  simultaneously, we must understand an interpolating family of  support Lagrangians,
and verify the invariance of  microlocal sheaves  as we move in the family.

In forthcoming work~\cite{Nperm}, we construct such a family and verify the sought-after invariance.
The symmetric group $\Sigma_{n+2}$ naturally acts on the family,
and in particular preserves its   central member $L_n^0$.
As a  topological space, we can obtain $L_n^0$ from the boundary of the $(n+1)$-dimensional permutohedron  by the equivalence relation
determined by its tessellation of $\BR^{n+1}$. For example, for $n=1$, it is obtained by 
from the boundary of a hexagon by identifying its opposite sides. 
Moreover, the singularities of $L_n^0$ are arboreal in the sense of \cite{Narb, Nexp},
and consequently wrapped microlocal sheaves supported along $L_n^0$ admit a simple combinatorial description.
Combined with the results of \cite{mik} and this paper, one arrives at a satisfying mirror description of  microlocal sheaves
on smooth toric hypersurfaces.
\end{remark}

\subsection{Microlocal sheaves on surfaces}\label{s:introsurf}

To provide a collection of simple additional examples, 
we also study microlocal sheaves on  exact symplectic surfaces  $\Sigma$.
We prove they are equivalent to well-known combinatorial constructions
(found for example in \cite{DK, HKK, Ncyc}),
and in particular  depend only on the induced orientation of $\Sigma$.

Fix an exact Lagrangian skeleton $\Gamma \subset \Sigma$ which we may regard as a locally finite embedded graph.
Our constructions will  depend only upon $\Sigma$ in a small neighborhood of $\Gamma$.

We introduce  $\BZ/2$-dg categories 
$
\mu\Sh_\Gamma(\Sigma)$, $\mu\Sh^w_\Gamma(\Sigma)
$
of  respective traditional and wrapped microlocal sheaves
on the surface $\Sigma$ supported along $\Gamma$. They are the  global sections
of a respective sheaf and cosheaf $\mu\Sh_\Gamma$, $\mu\Sh^w_\Gamma$ supported along $\Gamma$, similarly
to the structure of traditional and wrapped microlocal sheaves on conic open subspaces of a cotangent bundle discussed above.
In particular, they may be  calculated as the respective limit and colimit of their sections over small open balls. 
Over small open balls, their sections are given by passing to the symplectification of the contactification, 
choosing a local identification  with a conic open subspace of a cotangent bundle, and
then taking traditional and wrapped microlocal sheaves there. To insure this makes sense, we check
their required invariance
   under such choices of   local identifications (see Lemma~\ref{lem:inv} below).  
This is particularly simple here for surfaces where the moduli of such choices essentially reduces to a circle. 
 
 To calculate the $\BZ/2$-dg categories
$\mu\Sh_\Gamma(\Sigma)$, $\mu\Sh^w_\Gamma(\Sigma)$, we compare them with simple combinatorial models.

First, the exact Lagrangian skeleton $\Gamma \subset \Sigma$
 carries a natural constructible cosheaf $\cO$ of  cyclically ordered finite sets whose stalk at a point $x\in \Gamma$ 
comprises the components
$$
\xymatrix{
 \cO_x = \pi_0((\Sigma \setminus \Gamma) \cap B_x)
}$$
where $B_x \subset \Sigma$ is a small open ball around $x$, and $\cO_x$ inherits a cyclic ordering
from the orientation of $\Sigma$.

Let $\Lambda$ denote the cyclic category, and $\BZ/2\on{-}\dgst_k$ the $\oo$-category of small stable
$\BZ/2$-dg categories. By the results of~\cite{DK, Ncyc}, there is a functor
$$
\xymatrix{
\cC_{st}: \Lambda^{op}\ar[r] &  \BZ/2\on{-}\dgst_k
}
$$
that assigns to  the 
standard cycle $\Lambda_n  = \{ 1 \to 2\to  \cdots \to n\to n+1 \to 1\} \in \Lambda^{op}$
a  $\BZ/2$-dg category equivalent to  perfect $\BZ/2$-dg modules over the $A_n$-quiver
$$
\xymatrix{
\cC_{st}(\Lambda_n) \simeq (A_{n}\on{-\Perf}_k)_{\BZ/2}
}
$$
Passing to left adjoints,  we obtain an additional functor of opposite variance
$$
\xymatrix{
\cC^w_{st}:   \Lambda\ar[r] &  \BZ/2\on{-}\dgst_k
}
$$

Introduce the  respective composite sheaf and cosheaf   $\cF_{\Gamma} = \cC_{st} \circ \cO$,
$\cF_{\Gamma}^w =  \cC^w_{st} \circ \cO$, and form their respective
 global sections
$$
\xymatrix{
\cF_{\Gamma}(\Sigma) = \lim_{\Gamma} \cF_{\Gamma}
&
\cF_{\Gamma}^w(\Sigma) =  \colim_{\Gamma}  \cF_{\Gamma}^w
}
$$

\begin{thm}[Theorem~\ref{thm:surface} below]\label{thm:introsurface}
There are canonical equivalences of $\BZ/2$-dg categories 
$$
\xymatrix{
\cF_\Gamma(\Sigma) \ar[r]^-\sim & \mu\Sh_\Gamma(\Sigma)
&
\cF^w_\Gamma(\Sigma) \ar[r]^-\sim & \mu\Sh^w_\Gamma(\Sigma)
}
$$
\end{thm}

\begin{corollary}[Corollary~\ref{cor:topol} below]\label{cor:introtopol}
 The  $\BZ/2$-dg categories 
  $\mu\Sh_\Gamma(\Sigma)$, $\mu\Sh^w_\Gamma(\Sigma)$,
only depend on the exact symplectic structure on $\Sigma$ through the orientation
it defines. 
\end{corollary}

\begin{remark}
Suppose we restrict to compact Lagrangian skeleta, or more generally,
Lagrangian skeleta  with fixed structure near  the circular ends of $\Sigma$.
Then Dyckerhoff-Kapranov~\cite{DK} have explained that the 
the  $\BZ/2$-dg categories 
  $\cF_\Gamma(\Sigma)$, $\cF^w_\Gamma(\Sigma)$ are canonically independent
  of the specific choice of  skeleton. They provide invariants
  of the oriented surface $\Sigma$ equipped with a finite subset of its circular ends.
\end{remark}

The proof of Theorem~\ref{thm:introsurface} is local in the following sense.  As mentioned above, 
there is a respective sheaf and cosheaf
$\mu\Sh^w_{\Gamma}$, $\mu\Sh^w_{\Gamma}$ supported along $\Gamma$,
and canonical restriction equivalences of $\BZ/2$-dg categories
$$
\xymatrix{
\mu\Sh_{\Gamma}(\Sigma) \simeq \lim_{\Gamma} \mu\Sh_\Gamma
&
\mu\Sh^w_{\Gamma}(\Sigma) \simeq  \colim_{\Gamma}   \mu\Sh^w_\Gamma
}
$$
We show there are canonical equivalences of respective sheaves and cosheaves
$$
\xymatrix{
\cF_{\Gamma}\ar[r]^-\sim &  \mu\Sh_{\Gamma}
&
\cF_{\Gamma}^w \ar[r]^-\sim &   \mu\Sh^w_{\Gamma}
}
$$
thus implying the equivalences of Theorem~\ref{thm:introsurface}. Moreover, this provides a
 combinatorial model for calculating $\mu\Sh_{\Gamma}$, $\mu\Sh^w_{\Gamma}$.

We also obtain a simple proof of the following analogue of Theorem~\ref{thm:introduality}.
One could compare with the result of~\cite{HKK} classifying objects of $\cF_{\Gamma}^w(\Sigma)$
in terms of compact immersed curves in $\Sigma$ equipped with local systems.

\begin{corollary}[Corollary~\ref{cor:surfduality} below]
The natural hom-pairing provides an equivalence
$$
\xymatrix{
\mu\Sh_\Gamma(\Sigma)  \ar[r]^-\sim & \Fun^{ex}(\mu\Sh_\Gamma^w(\Sigma)^{op}, \Perf_k) 
}
$$
\end{corollary}

\begin{remark}
For  the choice of a bicanonical  trivialization of $\Sigma$, 
it is possible to lift the above theory from $\BZ/2$-dg categories to dg categories.
On the one hand, the contact transformations providing the local invariance of
the respective sheaf  and cosheaf $\mu\Sh_\Gamma$, $\mu\Sh^w_\Gamma$ will no longer have a ``metaplectic anomaly" 
requiring the trivialization of the shift $[2]$. On the other hand, 
the combinatorial functors $\cC_{st}$,  $\cC_{st}^w$,  and hence the respective sheaf  
and cosheaf
$\cF_{\Gamma}^w $, $\cF_{\Gamma}$  will lift to dg categories. 
Finally, the proof of Theorem~\ref{thm:introsurface} can be repeated 
to give an equivalence of dg categories of global sections for these respective lifts.
\end{remark}

%
%
%
%
%

As a concrete example, we describe traditional and wrapped microlocal sheaves on the $n$-punctured sphere $\Sigma_n = S^2 \setminus \{n \mbox{ points}\}$, for $n\geq 2$, 
with respect to a natural compact Lagrangian skeleton $\Gamma_n \subset \Sigma_n$.
 To state the answer,
for $n = 2$, set $Q_2 = \G_m$, 
and for $n \geq 3$, set
$$
\xymatrix{
Q_n = \BA^1 \cup_{pt} \BP^1 \cup_{pt} \BP^1  \cdots   \cup_{pt} \BP^1 \cup_{pt} \BA^1
}
$$
with $n-3$ copies of $\BP^1$, and where the inclusions of $pt= \Spec k$
 into each copy of $\BP^1$ from the  left and right  have distinct images.
 
 Then we have the following version of the results of~\cite{AAEKO},
 but with the $B$-side given in the form of coherent sheaves
 rather than matrix factorizations.

\begin{thm}[Theorem~\ref{thm:surfacemirror} below]\label{thm:introsurfacemirror}
There are mirror equivalences of $\BZ/2$-dg categories
$$
\xymatrix{
 \mu\Sh_{\Gamma_n}(\Sigma_n) \ar[r] &\Perf_\proper(Q_n)_{\BZ/2}
&
\mu\Sh^w_{\Gamma_n}(\Sigma_n) \ar[r]^-\sim &  \Coh(Q_n)_{\BZ/2}
}
$$
\end{thm}

\begin{remark}\label{rem:unambiguous}
Theorems~\ref{thm:intromirror} and~\ref{thm:introsurfacemirror} overlap when we take the one-dimensional
 pair of pants $P_1$
and the thrice-punctured sphere $\Sigma_3$. 

On the $B$-side, the statements are related by a natural  $\BZ/2$-dg  equivalence
$$
\xymatrix{
\Coh(Q_3)_{\BZ/2} \ar[r]^-\sim &  \MF(\BA^3, W_3)
}
$$
found in Proposition~\ref{prop:mfcohequiv} below.

On the $A$-side, they are related by a natural  $\BZ/2$-dg  equivalence
$$
\xymatrix{
\mu\Sh^w_{\Gamma_3}(\Sigma_3) \ar[r]^-\sim & \mu\Sh^w_{L_1}(P_1)
}
$$
Here the left hand side is defined by covering $\Sigma_3$ by open balls equivalent 
to conic open subspaces of a cotangent bundle and gluing local dg categories of microlocal sheaves. 
The right hand side is defined by realizing   $P_1$ as a Hamiltonian reduction of a single conic open subspace of a cotangent bundle and taking its dg category of microlocal sheaves.
In both cases, their calculation reduces to the property that they complete the diagram 
of $\BZ/2$-dg categories
  $$
 \xymatrix{
\ar[d] \Coh(\{0\})_{\BZ/2} \ar[r] & \Coh(\BA^1)_{\BZ/2}  \\ 
 \Coh(\BA^1)_{\BZ/2}  & 
 }
 $$ 
 to a pushout square,
 where the maps of the diagram are given by the evident pushforwards along the inclusion $ \Spec k = \{0\} \to \BA^1$.

\end{remark}


\subsection{Acknowledgements}
I  thank  D. Auroux, D. Ben-Zvi, K. Cieliebak, Y. Eliashberg, D. Gaistgory, N. Sheridan, D. Treumann, Z. Yun,
and E. Zaslow for their interest, encouragement, and valuable comments. I am particularly indebted to S. Ganatra
for suggesting the viewpoint of the amoeba, and pointing me toward Mikhalkin's construction of ``tailored pairs of pants".

I am grateful to the NSF for the support of grant DMS-1502178.


\section{Landau-Ginzburg $B$-model}

This  primary aim of this section
is to present  the Landau-Ginzburg $B$-model  with background $\BA^{n+2}$ and superpotential $W_{n+1}= z_1, \ldots, z_{n+1}$ in a form suited to  match our $A$-model calculations.


\subsection{Preliminaries}

Fix a characteristic zero algebraically closed field $k$.

By a dg category $\cC$, we will  always mean a stable differential $\BZ$-graded category.
The morphisms $\Hom_\cC(c_1, c_2)$,
for $c_1, c_2\in \cC$, form $\BZ$-graded cochain complexes, and 
we have the shift functors $[n]$, for $n\in \BZ$.

By a 2-periodic dg category $\cC$, we will mean a stable differential $\BZ$-graded category such that the shift functor $[2]$
is equivalent to the identity.
The morphisms $\Hom_\cC(c_1, c_2)$,
for $c_1, c_2\in \cC$, continue to form $\BZ$-graded cochain complexes,
though
they are invariant under even shifts.

By a $\BZ/2$-dg category, we will  always mean a stable differential $\BZ/2$-graded category.
 The morphisms $\Hom_\cC(c_1, c_2)$,
for $c_1, c_2\in \cC$, form $\BZ/2$-graded cochain complexes, and
the shift $[2]$ is equivalent to the identity.
 
 To any dg category $\cC$, we can assign a $\BZ/2$-dg category $\cC_{\BZ/2}$ by taking the same objects
 and only remembering the underlying $\BZ/2$-grading on morphism complexes
 $$
 \xymatrix{
 \Hom_{ \cC_{\BZ/2}}^0(c_1, c_2) = \oplus_{n\in \BZ} \Hom_{\cC}^{2n} (c_1, c_2)
&
 \Hom_{ \cC_{\BZ/2}}^1(c_1, c_2) = \oplus_{n\in \BZ} \Hom_{\cC}^{2n+1} (c_1, c_2)
 }
 $$
We will refer to $\cC_{\BZ/2}$ as the folding of $\cC$.

 To any $\BZ/2$-dg category $\cC$, we can assign a  2-periodic dg category $\cC_{2\BZ}$ by taking the same objects
 and taking the morphism complexes
 $$
 \xymatrix{
 \Hom_{ \cC_{2\BZ}}^n(c_1, c_2) = \Hom_{\cC}^{\ol n} (c_1, c_2)
 &n\in \BZ, \ol n\in\BZ/2
 }
 $$
We will refer to $\cC_{2\BZ}$ as the unfurling of $\cC$. 
%
This   provides an equivalence between 2-periodic and $\BZ/2$-dg categories,
and we will go back and forth between them without much comment.


\subsection{Matrix factorizations}

Consider the background $M = \Spec A$, with $A= k[z_1, \ldots, z_n] $, and a superpotential $W\in A$ such that $0\in \BA^1$ is its only possible critical value.

Introduce the special fiber $X = W^{-1}(0) = \Spec B$, with $B  = A /(W)$. 

Let $\Perf(X)$ be the dg category of perfect complexes on $X$,
and $\Coh(X)$ the dg category of bounded coherent complexes of sheaves on $X$.

Let $\D_\sing(X) = \Coh(X)/\Perf(X)$ be the 2-periodic dg quotient category of singularities. 


Let $\MF(M, W)$ be the  $\BZ/2$-dg category of matrix factorizations.
Its objects are pairs $(V, d)$ of a $\BZ/2$-graded  free $A $-module $V$
of finite rank
equipped with an odd endomorphism $d$
such that $d^2 =W\id$.
Thus we have $V=  V^0 \oplus V^1$, $d= (d_0, d_1) \in \Hom(V^0, V^1)\oplus \Hom(V^1, V^0)$, and
$d^2 =  (d_1 d_0, d_0 d_1) =(W\id, W\id)  \in \Hom(V^0, V^0)\oplus \Hom(V^1, V^1)$.
We denote the data of a matrix factorization by a diagram
$$
\xymatrix{
V^0 \ar[r]^-{d_0} &  V^1 \ar[r]^-{d_1}   & V^0
}
$$

Let $\MF(M, W)_{2\BZ} $ denote the unfurling of $\MF(M, W)$.
Then thanks to Orlov~\cite{orlov}, there is an equivalence of 2-periodic dg categories
$$\xymatrix{
\MF(M, W)_{2\BZ} \ar[r]^-\sim &\D_\sing(X)
& 
(V = V^0 \oplus V^1, d = (d_0, d_1)) \ar@{|->}[r] & \coker(d_1)
}
$$


\subsection{Coordinate hyperplanes}

For $n\in \mathbb N$, set  $[n] = \{1, \ldots, n\}$.

We will focus on 
 the background $\BA^{n+1} = \Spec A$, with $A_{n+1} = k[z_a \, |\, a\in [n+1]] $, and the superpotential 
  $$
  \xymatrix{
  W_{n+1} = z_1 \cdots z_{n+1} \in A_{n+1}
  }
  $$
  
  Introduce the union of coordinate hyperplanes 
  $$
  \xymatrix{
  X_n= W_{n+1}^{-1}(0) = \Spec B_n
  }$$ 
  where we set $B_n =A_{n+1} /(W)$.

It will also be convenient to set $W_{n+1}^a = W_{n+1}/z_a \in A$, for $a\in [n+1]$.

For $a\in [n+1]$, let $X_{n}^a =  \Spec A/(z_a) \subset X_n$ denote the coordinate hyperplane,  
and  $\cO_n^a$  its structure sheaf. As an object of $\Perf(\BA^{n+1})$, it admits the free resolution
$$
\xymatrix{
A_{n+1} \ar[r]^-{z_a} & A_{n+1} \ar[r] & \cO_n^a
}
$$
and as an object of $\Coh(X_{n})$, it admits the infinite resolution
$$
\xymatrix{
\cdots \ar[r]^-{W_{n+1}^a} & B_n \ar[r]^-{z_a} & B_n \ar[r]^-{W_{n+1}^a} & B_n \ar[r]^-{z_a} & B_n \ar[r] & \cO_n^a
}
$$

For $a\in [n+1]$,  let $\ul \cO_{n}^a\in \MF(\BA^{n+1}, W_{n+1})$ denote the matrix factorization
$$
\xymatrix{
A_{n+1} \ar[r]^-{W_{n+1}^a} &  A_{n+1}  \ar[r]^-{z_a}   & A_{n+1}
}
$$

\begin{prop}\label{prop:mfhoms}
The $\BZ/2$-dg category $\MF(\BA^{n+1}, W_{n+1})$ is split-generated by the collection of objects
$ \ul \cO_n^a$, for $a\in [n]$. There are equivalences of $\BZ/2$-graded $k$-modules
$$
\xymatrix{
H^*(\Hom(\ul \cO_n^a, \ul \cO_n^a)) \simeq  A_{n+1}/(z_a, W_{n+1}^a)
&
a\in [n+1]
}
$$
$$
\xymatrix{
H^*(\Hom(\ul \cO_n^a, \ul \cO_n^b)) =
 A_{n+1}/(z_a, z_b)[-1]
&
a\not = b\in [n+1]
}
$$
\end{prop}

\begin{proof}
 The collection of objects $ \cO_n^a$, for $a\in [n+1]$, generates $\Coh(X_n)$,    
and $\ul \cO_n^{n+1}$ is in the triangulated envelope of the collection of objects $\ul \cO_a$, for $a\in [n]$,
hence
the collection of objects $\ul \cO_n^a$, for $a\in [n]$, generates $\MF(\BA^{n+1}, W_{n+1})$.
The cohomology of morphism complexes is a straightforward calculation.
\end{proof}

Next, let us reduce the dimension by one, and consider the space $\BA^{n} = \Spec A_n$, with $A_n = k[z_a \, |\, a\in [n]]$,
and the function 
  $$
  \xymatrix{
  W_{n} = z_1 \cdots z_{n} \in A_{n}
  }
  $$

Introduce the union of coordinate hyperplanes  
$$
\xymatrix{
X_{n-1}= W_n^{-1}(0) = \Spec B_{n-1}
}
$$
where we set $B_{n-1}  =A_n /(W_n)$.

It will also be convenient to set $W_n^a = W_n/z_a \in A_n$, for $a\in [n]$.

For $a\in [n]$, let $X_{n-1}^a =  \Spec A_n/(z_a)$ denote the coordinate hyperplane,  
and  $\cO_{n-1}^a$  its structure sheaf. As an object of $\Perf(\BA^n)$, it admits the resolution
$$
\xymatrix{
A_n \ar[r]^-{z_a} & A_n \ar[r] & \cO_{n-1}^a
}
$$
and as an object of $\Coh(X_{n-1})$, it admits the infinite resolution
$$
\xymatrix{
\cdots \ar[r]^-{W_n^a} &  B_{n-1} \ar[r]^-{z_a} &  B_{n-1} \ar[r]^-{W_n^a} &  B_{n-1} \ar[r]^-{z_a} & B_{n-1} \ar[r] & \cO_{n-1}^a 
}
$$

 \begin{prop}\label{prop:cohhoms}
 Let $u$ be a variable of cohomological degree $2$.

The dg category $\Coh(X_{n-1})$ is generated by the collection of objects
$ \cO^a_{n-1}$, for $a\in [n]$. 
There are equivalences of $\BZ$-graded $k$-modules
$$
\xymatrix{
H^*(\Hom( \cO_{n-1}^a,  \cO_{n-1}^a)) \simeq  A_{n}[u]/(z_a, uW_{n}^a)
&
a\in [n]
}
$$
$$
\xymatrix{
H^*(\Hom( \cO_{n-1}^a,  \cO_{n-1}^b)) =
 A_{n}[u]/(z_a, z_b)[-1]
&
a\not = b\in [n]
}
$$
 \end{prop}

\begin{proof}
 The collection of objects $ \cO_{n-1}^a$, for $a\in [n]$, clearly generates,
and the cohomology of morphism complexes is a straightforward calculation.
\end{proof}

Now let us consider the natural coordinate projection
$$
\xymatrix{
\pi:X_{n} \ar[r] & X_{n-1}
 }$$

Recall that we write $\Coh(X_{n-1})_{\BZ/2}$ for the folding of $\Coh(X_{n-1})$.

 \begin{prop}\label{prop:mfcohequiv}
 The pullback of coherent sheaves 
 $$
\xymatrix{
 \pi^*:\Coh(X_{n-1}) \ar[r] & \Coh(X_{n}) 
 }
 $$
 induces an equivalence of $\BZ/2$-dg categories
$$
\xymatrix{
\Coh(X_{n-1})_{\BZ/2} \ar[r]^-\sim & \MF(\BA^{n+1}, W_{n+1})
}
$$
such that $\cO_{n-1}^a$ maps to $\ul \cO_n^a$, for $a\in [n]$, and such that $u$ maps to $z_{n+1}$. 
 \end{prop}
 
 \begin{proof}
 The pullback $\pi^*$ induces a functor via passage to $D_\sing(X_n)$. By the first assertion of Proposition~\ref{prop:mfhoms}, it is essentially surjective. Thus it suffices to check it is an isomorphism on the cohomology of the morphism complexes
  of the generating collection of Proposition~\ref{prop:cohhoms}.
This   is a straightforward calculation using Propositions~\ref{prop:mfhoms} and \ref{prop:cohhoms}.
 \end{proof}


\subsection{Descent description}

We will collect here useful descent statements proved in~\cite{GR} in the setting of stable dg categories~\cite{dgdg}.
They are specializations of   analogous statements possible in  the setting
of stable $\oo$-categories~\cite{htopos, halgebra}.

Let $\dgst_k$ be the $\oo$-category of  $k$-linear small stable dg categories with exact functors.

Let $\dgSt_k$ be the $\oo$-category of $k$-linear cocomplete dg categories with continuous functors.
Let $\dgSt_k^{c}$ be the not full $\oo$-subcategory of  $\dgSt^c_k$ of $k$-linear cocomplete dg categories with functors those that preserve compact objects.
Taking ind-categories provides an equivalence   
$$
\xymatrix{
\Ind:\dgst_k \ar[r]^-\sim & \dgSt_k^{c} &
}
$$
and taking compact objects provides an inverse equivalence
$$
\xymatrix{
\kappa:\dgSt_k^{c} \ar[r]^-\sim & \dgst_k
}
$$
%

Let $\fX_k$ denote the category of affine lci $k$-schemes and closed embeddings. (We work in this setting
for concreteness but far more generality is possible.)

Passing to coherent sheaves and pushforwards  provides a functor
$$
\xymatrix{
\Coh_*: \fX_k\ar[r] &  \dgst_k
}
$$
Passing to perfect complexes and $*$-pullbacks provides a functor 
$$
\xymatrix{
\Perf^*: \fX_k^{op}\ar[r] &  \dgst_k
}
$$
and similarly for perfect complexes with proper support
$$
\xymatrix{
\Perf^*_\proper: \fX_k^{op}\ar[r] &  \dgst_k
}
$$
  since morphisms of $\fX_k$ are closed embeddings.
  
Passing to ind-coherent sheaves and pushforwards provides a functor
$$
\xymatrix{
\Ind\Coh_*: \fX_k\ar[r] &  \dgSt_k^{c}
}
$$
Passing to quasi-coherent sheaves and $*$-pullbacks provides a functor 
$$
\xymatrix{
\QCoh^* \simeq \Ind \Perf^*: \fX_k^{op}\ar[r] &  \dgSt_k^{c}
}
$$

Passing to ind-coherent sheaves and $*$-pullbacks or $!$-pullbacks provide respective functors
$$
\xymatrix{
\Ind\Coh^*: \fX_k^{op}\ar[r] &  \dgSt_k
&
\Ind\Coh^!: \fX_k^{op}\ar[r] &  \dgSt_k
}
$$
Tensoring with the dualizing complex provides a natural intertwining equivalence
$$
\xymatrix{
\otimes \omega:\Ind\Coh^*\ar[r]^-\sim & \Ind\Coh^!
}
$$

Now let us quote the folllowing result of Gaitsgory-Rozenblyum~ \cite{GR} and then note the further assertions it implies.

\begin{thm}[Theorem A.1.2 of Chapter IV.4 of \cite{GR}]\label{thm:descent}
A pushout square 
in $\fX_k$ is taken to a pullback square by $\Ind\Coh^!$.
\end{thm}

\begin{corollary}\label{cor:descent}
A pushout square 
in $\fX_k$ is taken to a pushout square by $\Coh_*$ and $\Ind \Coh_*$, and a pullback square by
$\Perf^*$,  $\Perf^*_\proper$, and $\Ind \Coh^*$.
\end{corollary}

\begin{proof}
The intertwining equivalence $\Ind\Coh^*\simeq \Ind\Coh^!$ implies the assertion for $\Ind\Coh^*$.

The full embeddings  $\Perf^*_\proper \subset \Perf^* \subset \Ind\Coh^*$, and the fact that perfect complexes
are locally characterized,
 implies the assertions for $\Perf^*_\proper$ and $\Perf$.

The pullback square for $\Ind\Coh^!$ is equivalently calculated in the 
 $\oo$-category  $\dgSt_k^{R}$ of $k$-linear  cocomplete dg categories with right adjoints.
Passing to left adjoints, gives an equivalence 
 $$
 \xymatrix{
 (\dgSt_k^{R})^{op} \ar[r]^-\sim &  \dgSt_k
 }
 $$ 
and thus implies  
  the pushout square for  for $\Ind\Coh_*$.

The fact that $\Ind$ preserves colimit diagrams, and conversely, taking compact objects preserves colimit diagrams with quasi-compact morphisms, implies the assertion for $\Coh_*$.
\end{proof}

 \begin{remark}
 Let $i_1 = i_2:\Spec k = \{0\} \to \BA^1$ be the  inclusion of the origin.
It is not true that 
the pushout square 
$$
\xymatrix{
\ar[d]_-{i_1} \{0\} \ar[r]^-{i_2} & \BA^1 \ar[d]^-{\tilde i_1} \\
\BA^1  \ar[r]^-{\tilde i_2} & \BA^1 \cup_{\{0\}} \BA^1
}
$$
is taken to a pullback square by $\QCoh^*$. The restrictions $i_1^*, i_2^*:\QCoh(\BA^1)\to \QCoh(\{0\})$ preserve limits, so the natural maps
from the pullback to each $\QCoh(\BA^1)$ must also preserve limits. But the restrictions $\tilde i_1^*, \tilde i_2^*:\QCoh(\BA^1 \cup_{\{0\}} \BA^1) \to \QCoh(\BA^1) $ do not preserve limits.  
\end{remark}

Now let us specialize  to our situations of interest.

Let us return to the space $\BA^{n} = \Spec A_n$, with $A_n = k[z_a \, |\, a\in [n]]$,
and the function 
  $$
  \xymatrix{
  W_{n} = z_1 \cdots z_{n} \in A_{n}
  }
  $$
Recall the union of coordinate hyperplanes  
$$
\xymatrix{
X_{n-1}= W_n^{-1}(0) = \Spec B_{n-1}
}
$$
where we set $B_{n-1}  =A_n /(W_n)$.

Let $\cI_n^\circ$ denote the poset of proper subsets $I\subset [n]$ under inclusions. 
For $I \in \cI_n^\circ$, consider 
 the corresponding coordinate subspace
$$
\xymatrix{
X_I =  \Spec A_n/(z_a \, |\, a\not \in I)
}
$$

We have a colimit diagram of closed embeddings
\begin{equation}\label{eq:colim}
\xymatrix{
\colim_{I \in \cI_n^\circ} X_I \ar[r]^-\sim & X_{n-1}
}
\end{equation}

\begin{prop}\label{prop:descent}
The colimit diagram~\eqref{eq:colim}  
 is taken to a colimit diagram by $\Coh_*$ and $\Ind \Coh_*$, and a limit diagram by
$\Perf^*$, $\Perf^*_\proper$,  $\Ind\Coh^!$, and $\Ind \Coh^*$.
\end{prop}

\begin{proof}

For $I \in \cI_{n-1}^\circ$, we have $X_I = (X_I \cap \{z_n = 0\}) \times \BA^1_{z_n}$,  and a colimit diagram of closed embeddings
$$
\xymatrix{
\colim_{I \in \cI_{n-1}^\circ} X_I \simeq (\colim_{I \in \cI_{n-1}^\circ}  (X_I \cap \{z_n = 0\}) ) \times \BA^1_{z_n}
\ar[r]^-\sim & X_{n-2} \times \BA^1_{z_n}
}
$$

Similarly, for $I \in \cI_{n-1}^\circ$, we have $X_{I\cup\{z_n\}} = (X_I \cap \{z_n = 0\})$,  and a colimit diagram of closed embeddings
$$
\xymatrix{
\colim_{I \in \cI_{n-1}^\circ} X_{I\cup\{z_n\}} \simeq \colim_{I \in \cI_{n-1}^\circ}  X_I \cap \{z_n = 0\}) 
\ar[r]^-\sim & X_{n-2}
}
$$

By induction, we may reduce from proving the assertions for  the colimit diagram~\eqref{eq:colim} to the pushout square
of closed embeddings
$$
\xymatrix{
\ar[d] X_{n-2} \times \{z_n = 0\} \ar[r] &X_{n-2} \times \BA^1_{z_n} \ar[d] \\
  \ar[r]  \BA^{n-1} \times \{z_n = 0\} &  X_{n-1}
}
$$
Now the assertions follow immediately from Theorem~\ref{thm:descent} and Corollary~\ref{cor:descent}.
\end{proof}

Finally, let us record another specific example of such descent.

Consider the iterated pushout
\begin{equation}\label{eq:itpushout}
\xymatrix{
Q_m = \BA^1 \cup_{pt} \BP^1 \cup_{pt} \BP^1  \cdots   \cup_{pt} \BP^1 \cup_{pt} \BA^1
}
\end{equation}
with $m$ copies of $\BP^1$, and where the inclusions of $pt= \Spec k$
 into each copy of $\BP^1$ from the  left and right  have distinct images.
 
 By repeated applications of Theorem~\ref{thm:descent}. and Corollary~\ref{cor:descent}, we obtain the following.

\begin{prop}\label{prop:moredescent}
The iterated pushout~\eqref{eq:itpushout}  
 is taken to an iterated pushout by $\Coh_*$ and $\Ind \Coh_*$, and an iterated pullback by
$\Perf^*$, $\Perf^*_\proper$,  $\Ind\Coh^!$, and $\Ind \Coh^*$.
\end{prop}

\section{Microlocal $A$-model}


\subsection{Setup}

Let $Z$ be a real analytic manifold. 

Consider the cotangent bundle  and 
its spherical projectivization 
$$
\xymatrix{
\pi:T^*Z\ar[r] & Z
&
\pi^\oo:S^\oo Z = (T^*Z \setminus Z)/\BR_{>0}\ar[r] & Z
}
$$
with their respective standard exact symplectic and contact structures.

For convenience, fix a Riemannian metric on $Z$, so that in particular we have an identification
with the unit cosphere bundle
$$
\xymatrix{
S^\oo Z\simeq U^*Z \subset T^*Z
}
$$

We will often work with a closed conic Lagrangian subvariety and its Legendrian spherical projectivization
$$
\xymatrix{
\Lambda \subset T^*Z
&
\Lambda^\oo = (\Lambda\cap  (T^*Z \setminus Z))/\BR_{>0} \subset S^\oo Z
}
$$

By the front projection of $\Lambda^\oo$, we will mean the image 
$$
\xymatrix{
Y = \pi^\oo(\Lambda^\oo)\subset Z
}
$$ 
In the generic situation, the restriction 
$$
\xymatrix{
\pi^\oo|_{\Lambda^\oo}:\Lambda^\oo\ar[r] & Y 
}
$$ is finite so that the front projection is a hypersurface. 

We will also often fix  $\cS=\{Z_\alpha\}_{\alpha\in A}$  a Whitney stratification  of $Z$ such that $Y\subset Z$ is a union of strata. Hence
we have  inclusions
$$
\xymatrix{
\Lambda \subset T^*_\cS Z = \coprod_{\alpha\in A} T^*_{Z_\alpha} Z
&
\Lambda^\oo \subset S^\oo_\cS Z = \coprod_{\alpha\in A} S^\oo_{Z_\alpha} Z
}
$$ 
into the union of conormal bundles to strata 
and their spherical projectivizations.


Given   a Whitney stratification $\cS=\{Z_\alpha\}_{\alpha\in A}$, 
by a small open ball $B\subset Z$ around a point $z\in Z$, we will always mean an open ball $B= B(r) \subset Z$ of a radius $r>0$
such that the corresponding spheres $S(r') \subset Z$, for all $0<r'<r$, are transverse to the strata of $\cS$.


\subsection{Sheaves}

Fix a field $k$ of characteristic zero. 

Let $\Sh^\un(Z)$ denote the dg category of  all complexes of sheaves of $k$-vector spaces on $Z$
such that there exists a Whitney stratification
$\cS=\{Z_\alpha\}_{\alpha\in A}$
such that  for each stratum $Z_\alpha \subset Z$, the total cohomology sheaf of the restriction
$\cF|_{Z_\alpha}$ is locally constant. 

For a fixed Whitney stratification
$\cS=\{Z_\alpha\}_{\alpha\in A}$,
we denote by  $\Sh^\un_\cS(Z) \subset \Sh^\un(Z)$  the full dg subcategory of  objects $\cF\in \Sh^\un(Z)$ 
such that  for each stratum $Z_\alpha \subset Z$, the total cohomology sheaf of the restriction
$\cF|_{Z_\alpha}$ is locally constant.

Thus we have  $\Sh^\un(Z) = \cup_{\cS}  \Sh_\cS^\un(Z)$.
Note that we do not impose any constraint on the rank of the total cohomology sheaf 
in either of the above definitions.

Let $\Sh(Z) \subset \Sh^\un(Z)$ denote the full dg subcategory of  constructible complexes of sheaves of $k$-vector spaces on $Z$, or in other words, 
complexes of sheaves of $k$-vector spaces on $Z$
such that there exists a Whitney stratification
$\cS=\{Z_\alpha\}_{\alpha\in A}$
such that  for each stratum $Z_\alpha \subset Z$, the total cohomology sheaf of the restriction
$\cF|_{Z_\alpha}$ is locally constant of finite rank.

For  a fixed Whitney stratification
$\cS=\{Z_\alpha\}_{\alpha\in A}$,
we denote by $\Sh_\cS(Z) = \Sh(Z) \cap \Sh^\un_\cS(Z) \subset \Sh^\un(Z)$ the full dg subcategory of   $\cS$-constructible complexes, or in other words,  
complexes of sheaves of $k$-vector spaces on $Z$
such that for each stratum $Z_\alpha \subset Z$, 
 the total cohomology sheaf of the restriction
$\cF|_{Z_\alpha}$ is locally constant of finite rank.

Thus we have  $\Sh(Z) = \cup_{\cS}  \Sh_\cS(Z)$. These are the dg versions of traditional constructible complexes.

We will abuse terminology and refer to objects of $\Sh^\un(Z)$ as large constructible sheaves, and
 objects of $\Sh(Z)$ as  constructible sheaves.
When $U\subset Z$ is an open subset, we will abuse notation and write 
 $\Sh^\un_\cS(U) \subset \Sh^\un(U)$, respectively $\Sh_\cS(U) \subset \Sh(U)$, for large constructible sheaves, respectively constructible sheaves, with respect to $\cS \cap U$.

 All functors between dg categories of  sheaves will be derived in the dg sense, though the notation may not explicitly reflect it.  
 We will be careful when applying functors to large constructible sheaves to work with continuous versions.
 For example, for a closed embedding $i:Y\to Z$, by the $!$-restriction $i^!: \Sh^\un(Z)\to \Sh^\un(Y)$, we will mean
 the shifted cone $i^! \simeq \Cone(\cF\to j_*j^*\cF)[-1]$ where $j:U\to Z$ is the inclusion of the open complement $U = Z\setminus Y$. For a smooth map $f:Y\to Z$, by the $!$-pullback $f^!: \Sh^\un(Z)\to \Sh^\un(Y)$, we will mean
 the twist of the $*$-pullback $f^!\cF \simeq f^*\cF \otimes \omega_f$ where $\omega_f \simeq \orient_{f}[\dim Y/Z]
 $ is the relative dualizing complex.

\begin{example}
For $Z=pt$ a point, we have $\Sh^\un(pt) \simeq \Mod_k$, the dg category of  $k$-modules, and 
$\Sh(pt) \simeq \Perf_k$, the full dg subcategory of perfect $k$-modules.

More generally, for $Z$ stratified by a single stratum $\cS = \{Z\}$, we have 
 $\Sh_\cS^\un(Z) = \Loc^\un(Z)$, the dg category of  locally constant complexes on $Z$, and 
$\Sh_\cS(Z)  = \Loc(Z)$, the full dg subcategory of finite rank locally constant complexes on $Z$.

\end{example}


\subsection{Singular support}

Fix a point $(z, \xi) \in T^* Z$. 

Let $B\subset Z$ be an open ball around $z\in Z$,
and $f:B\to \BR$ a smooth function such that $f(z) = 0$ and $df|_z = \xi$.  
We will refer to $f$ as a compatible test function.
 
 Define the vanishing cycles functor
 $$
 \xymatrix{
 \phi_f:\Sh^\un(Z)\ar[r] & \Mod
 }
 $$
 $$
 \xymatrix{
 \phi_f(\cF) = \Gamma_{\{f\geq 0\}}(B, \cF|_B) \simeq  
 \Cone(\Gamma(B, \cF|_B) \to \Gamma(\{f< 0\}, \cF|_{ \{f< 0\}}))[-1]
 }
 $$
 where we take $B \subset Z$ sufficiently small.
 In other words, we take sections of $\cF$ over the ball $B$ supported where $f\geq 0$, or equivalently vanishing where $f< 0$. 
If we introduce the correspondences
 $$
 \xymatrix{
 B & \ar[l]_-{q_{+}} \{f\geq 0\} \ar[r]^-{p_{+}} & pt
&
 B & \ar[l]_-{q_{-}} \{f\leq 0\} \ar[r]^-{p_{-}} & pt
 }
 $$
 then the vanishing cycles take the equivalent form as hyperbolic restrictions
  $$
 \xymatrix{
 \phi_f(\cF) \simeq p_{+*} q_{+}^! (\cF|_B)
 \simeq  p_{-!} q_-^* (\cF|_B)
 }
 $$
 
 To any object $\cF \in\Sh^\un(Z) $, define its
singular support  
$$
\xymatrix{
\ssupp(\cF) \subset T^* Z
}$$ 
to be the largest closed subset such that $\phi_f(\cF) \simeq 0$, for any 
$(z, \xi) \in T^*Z \setminus \ssupp(\cF)$, and any compatible test function $f$.
Define its spherical singular support to be the spherical projectivization 
$$\xymatrix{
\ssupp^\oo(\cF) = (\ssupp(\cF)\setminus (T^*Z \setminus Z))/\BR_{>0} \subset S^\oo Z
}
$$
The singular support  is a closed conic Lagrangian subvariety,
 and   the spherical singular support  is a closed Legendrian subvariety.

\begin{example}\label{ex conv}
Suppose $i:U\to Z$ is the inclusion of an open subspace whose closure is a submanifold with boundary modeled on a Euclidean halfspace. Then the singular support  $ \ssupp(i_* k_U) \subset T^*Z$ of the standard extension $i_* k_U\in \Sh(Z)$ 
consists 
of the union of $U \subset Z$ and the inward conormal codirection along the boundary $\partial U \subset Z$.
More precisely, if near a point $z\in \partial U$, we have $U = \{f > 0\}$, for 
a local coordinate $f$, then $ \ssupp(i_* k_U)|_z$ is the closed ray $\R_{\geq 0} \langle df|_z\rangle $. 

More generally, suppose $i:U\to Z$ is the inclusion of an open subspace whose closure is a submanifold with corners
modeled  on a Euclidean quadrant. 
Then the singular support  $ \ssupp(i_* k_U) \subset T^*Z$ consists of the inward conormal cone along the boundary $\partial U \subset Z$. More precisely, if near a point $z\in \partial U$, we have $U = \{f_1, \ldots, f_k > 0\}$, for 
local coordinates $f_1, \ldots, f_k$, then $  \ssupp(i_* k_U)|_z$
 is the closed cone
$\R_{\geq 0} \langle df_1|_z, \ldots, df_k|_z\rangle$. 

\end{example}

For a conic Lagrangian subvariety $\Lambda\subset T^*Z$,
we write $\Sh_{\Lambda}^\un(Z) \subset \Sh^\un(Z)$,
respectively $\Sh_{\Lambda}(Z)    \subset \Sh(Z)$ 
 for the full dg subcategory of objects $\cF \in \Sh^\un(Z)$, respectively $\cF \in \Sh(Z)$,  with singular support satisfying 
$
\ssupp(\cF) \subset \Lambda.
$

Given  a Whitney stratification
$\cS$,
an inclusion $\Lambda\subset T^*_\cS Z$ implies the full inclusions
$\Sh^\un_{\Lambda}(Z)\subset \Sh^\un_\cS(Z)$, 
$\Sh_{\Lambda}(Z)\subset \Sh_\cS(Z)$, 
and more generally, an inclusion $\Lambda\subset \Lambda'$
implies the full inclusions
$\Sh^\un_{\Lambda}(Z)\subset \Sh^\un_{\Lambda'}(Z)$,
$\Sh_{\Lambda}(Z)\subset \Sh_{\Lambda'}(Z)$.

When $U\subset Z$ is an open subset, we will abuse notation and write 
 $\Sh^\un_\Lambda(U) \subset \Sh^\un(U)$, respectively $\Sh_\Lambda(U) \subset \Sh(U)$,
 for the full dg subcategory of objects $\cF \in \Sh^\un(U)$, respectively $\cF \in \Sh(U)$,  with singular support satisfying 
$
\ssupp(\cF) \subset \Lambda\cap \pi^{-1}(U).
$

\begin{example}

For the zero-section $\Lambda = Z \subset T^*Z$, and stratification $\cS = \{Z\}$, 
 we have the coincidences 
$\Sh^\un_{\Lambda}(Z) = \Sh^\un_\cS(Z) \simeq \Loc^\un(Z)$,
$\Sh_{\Lambda}(Z) = \Sh_\cS(Z) \simeq\Loc(Z)$.

\end{example}

\begin{remark}
Let $\omega_Z \simeq \orient_Z[\dim Z] \simeq p^! k_{pt}$, for $p:Z\to pt$, be the Verdier dualizing complex. 
For a conic Lagrangian subvariety $\Lambda\subset T^*Z$, 
and the antipodal  conic Lagrangian subvariety  $-\Lambda \subset T^*Z$, Verdier duality provides an involutive  equivalence
$$
\xymatrix{
\bD_Z:\Sh_{\Lambda}(Z)^{\op} \ar[r]^-\sim & \Sh_{-\Lambda}(Z)
&
\bD_Z(\cF) = \inthom(\cF, \omega_Z)
}
$$
\end{remark}


\subsection{Microlocal sheaves}\label{s:microsh}

Let $\Lambda \subset T^*Z$  be a closed conic Lagrangian subvariety.


%

To  a conic open subspace $\Omega \subset T^*Z$, let us introduce the dg category $\mu\Sh^\un_{\Lambda}(\Omega)$
  of large microlocal sheaves on $\Omega$ supported along~$\Lambda$. 
 To describe its construction, let us first mention some of its formal properties. 
 
 Given an inclusion of conic open subspaces $\Omega'_Z \subset \Omega$,
  there is a natural restriction functor $\mu\Sh^\un_{\Lambda}(\Omega)\to \mu\Sh^\un_{\Lambda}(\Omega_{Z'})$.
  These assignments assemble into a sheaf 
  $\mu\Sh^\un_{\Lambda}$
 of dg categories supported along $\Lambda$. 
  Furthermore, there exists a Whitney stratification of $\Lambda$ such that the restriction of $\mu\Sh^\un_{\Lambda}$  to each stratum is locally constant. Thus we can reconstruct $\mu\Sh^\un_{\Lambda}$ from the assignments
  $\mu\Sh^\un_{\Lambda}(\Omega)$ for small conic open neighborhoods $\Omega \subset T^*Z$
  of points $(x, \xi) \in \Lambda$.
  Finally, given a closed embedding of conic Lagrangian subvarieties $\Lambda'\subset \Lambda$,
  there is a natural full embedding $\mu\Sh^\un_{\Lambda'} \subset \mu\Sh^\un_{\Lambda}$
  of sheaves of dg categories.

  All of the above facts follow from the local description of $\mu\Sh^\un_{\Lambda}(\Omega)$
 which we now recall.
Note for a point $(x, \xi) \in \Lambda$ there are two local cases to consider: either 1) $\xi = 0$ so that locally
$\Omega$ is  the cotangent bundle $T^*B$ of a small open ball $B\subset Z$, or  2) $\xi \not = 0$ so that locally $\Omega$ is the cone over a small 
open ball  $\Omega^\oo \subset S^\oo Z$.

\medskip

Case 1) For $B= \pi(\Omega)$, there is always a canonical functor $\Sh^\un_\Lambda(B) \to \mu\Sh^\un_\Lambda(\Omega)$,
and when $\Omega = T^*B$, this functor is in fact an equivalence
$$
\xymatrix{
\Sh^\un_\Lambda(B) \ar[r]^-\sim & \mu\Sh^\un_\Lambda(T^*B)
}
$$

\medskip

Case 2) Suppose $\Omega \subset T^*Z$ is the cone over  
a small open ball $\Omega^\oo\subset S^\oo Z$.
 
 Set $B= \pi(\Omega)$, and let $\Sh^\un_{\Lambda}(B, \Omega)\subset \Sh^\un(B)$
denote the full dg subcategory of objects $\cF\in \Sh^\un(B)$ with singular support satisfying $\ssupp(\cF) \cap \Omega \subset \Lambda$.
Then there is a natural equivalence
$$
\xymatrix{
\Sh^\un_{\Lambda}(B, \Omega)/K^\un(B, \Omega)\ar[r]^-\sim  & \mu \Sh^\un_{\Lambda}(\Omega)
}
$$
where  $K^\un(B, \Omega)\subset\Sh^\un_{\Lambda}(B, \Omega)$ denotes the full dg subcategory of 
 objects $\cF\in \Sh^\un(B)$ with singular support satisfying $\ssupp(\cF) \cap \Omega = \emptyset$.

It is possible to make this more concrete, and in particular avoid the appearance of a dg quotient category.
 By applying a contact transformation, we may arrange to be in the generic situation where  the
front projection 
$$
\xymatrix{
\pi^\oo|_{\Lambda^\oo}:\Lambda^\oo\ar[r] & Y 
}
$$ 
is finite so that $Y= \pi^\oo(\Lambda^\oo )\subset Z$ is a hypersurface.
%
%
%

In this case,  the natural functor
$\Sh^\un_{\Lambda}(B)\to \mu \Sh^\un_{\Lambda}(\Omega)$
 induces  an
equivalence on the quotient dg category
$$
\xymatrix{
\Sh^\un_{\Lambda}(B)/\Loc^\un(B)\ar[r]^-\sim  & \mu \Sh^\un_{\Lambda}(\Omega)
}
$$
where  $\Loc^\un(B) \subset \Sh^\un(B)$ denotes the full dg subcategory of locally constant complexes, or in other words complexes with  singular support lying in the zero-section $B\subset T^*B$.

Now introduce the respective full dg subcategories 
$$
\xymatrix{
\Sh^\un_{\Lambda}(B)^0_* \subset  \Sh^\un_{\Lambda}(B) &
\Sh^\un_{\Lambda}(B)^0_! \subset  \Sh^\un_{\Lambda}(B) }
$$
of complexes $\cF \in \Sh^\un_{\Lambda}(B) $
with no sections and no  compactly-supported sections 
$$
\xymatrix{
\Gamma(B, \cF) \simeq 0
&
\Gamma_c(B, \cF) \simeq 0
}$$
Then the natural functor $\Sh^\un_{\Lambda}(B)\to \mu \Sh^\un_{\Lambda}(\Omega)$ restricts to equivalences
$$
\xymatrix{
\Sh^\un_{\Lambda}(B)^0_*\ar[r]^-\sim  & \mu \Sh^\un_{\Lambda}(\Omega)
&
\Sh^\un_{\Lambda}(B)^0_!\ar[r]^-\sim  & \mu \Sh^\un_{\Lambda}(\Omega)
}
$$

\begin{remark}\label{rem:cases}
In fact, one can reduce Case 1)  to Case 2) by passing to $Z' = Z\times \BR$,  $\Lambda' = \Lambda \times T^*_{\{0\}} \BR \subset T^*Z \times T^* \BR \simeq T^* Z' $, and $((z, 0), (0, 1)) \in T^*Z \times T^* \BR \simeq T^* Z'$ . 
\end{remark}

We similarly introduce  the full dg subcategory $\mu\Sh_{\Lambda}(\Omega)\subset \mu\Sh^\un_{\Lambda}(\Omega)$
  of microlocal sheaves on $\Omega$ supported along~$\Lambda$. 
It is constructed as above but working with constructible sheaves rather than large constructible sheaves. So for example, in case 2) above, if we denote by $\Sh_{\Lambda}(B, \Omega)\subset \Sh(B)$
 the full dg subcategory of objects $\cF\in \Sh(B)$ with singular support satisfying $\ssupp(\cF) \cap \Omega \subset \Lambda$,
then there is a natural equivalence
$$
\xymatrix{
\Sh_{\Lambda}(B, \Omega)/K(B, \Omega)\ar[r]^-\sim  & \mu \Sh_{\Lambda}(\Omega)
}
$$
where  $K(B, \Omega)\subset\Sh_{\Lambda}(B, \Omega)$ denotes the full dg subcategory of 
 objects $\cF\in \Sh(B)$ with singular support satisfying $\ssupp(\cF) \cap \Omega = \emptyset$.

The following analogous  formal properties again follow directly from the constructions. 
 The dg category $\mu\Sh_{\Lambda}(\Omega)$  is the sections of a subsheaf $\mu\Sh_{\Lambda}\subset  \mu\Sh^\un_{\Lambda}$
 of full dg subcategories supported along $\Lambda$.
   Given a Whitney stratification of $\Lambda$ such that the restriction of $\mu\Sh^\un_{\Lambda}$  to each stratum is locally constant, the restriction of $\mu\Sh_{\Lambda}$  to each stratum will also be locally constant.
  Finally, given a closed embedding of conic Lagrangian subvarieties $\Lambda'\subset \Lambda$,
  the full embedding $\mu\Sh^\un_{\Lambda'} \subset \mu\Sh^\un_{\Lambda}$ restricts to a full embedding
  $\mu\Sh_{\Lambda'} \subset \mu\Sh_{\Lambda}$.

\begin{example}\label{ex cod 1}
Suppose $Z = \BR$.
Inside of $T^*\BR \simeq \BR\times \BR$, introduce the conic Lagrangian subvariety and  conic open subspace 
$$
\xymatrix{
 \Lambda =   \{(t, 0) \, |\, t\in \BR\}  \cup \{(0, \eta) \, |\, \eta \in \BR_{\geq 0}\}    &
  \Omega   =    \{(t, \eta) \, |\, t\in \BR,  \eta\in \BR_{\geq 0}\} 
 }
 $$

Then there are canonical equivalences
$$
\xymatrix{
\Mod_k \ar[r]^-\sim &  \Sh^\un_\Lambda ( Z)^0_! \ar[r]^-\sim & \mu \Sh^\un_\Lambda(\Omega)
&
V \ar@{|->}[r] & j_{+*}p_+^*V
}
$$
induced by the correspondence
$$
\xymatrix{
pt & \ar[l]_-{p_+}  \BR_{<0} \ar@{^(->}[r]^-{j_+}  & \BR
}
$$

Similarly,  there are canonical equivalences
$$
\xymatrix{
\Mod_k \ar[r]^-\sim &  \Sh^\un_\Lambda ( Z)^0_* \ar[r]^-\sim & \mu \Sh^\un_\Lambda(\Omega)
&
V \ar@{|->}[r] & j_{-!}p_-^! V
}
$$
induced by the correspondence
$$
\xymatrix{
pt & \ar[l]_-{p_-}  \BR_{<0} \ar@{^(->}[r]^-{j_-}  & \BR
}
$$

Furthemore, the composite functors are naturally equivalent
$$
\xymatrix{
j_{-!}p_-^! \simeq j_{+*}p_+^* :\Mod_k \ar[r]^-\sim &  \mu\Sh^\un_\Lambda(\Omega)
}
$$
An inverse equivalence is given by the hyperbolic localization
$$
\xymatrix{
\eta: \mu\Sh^\un_\Lambda( Z) \ar[r]^-\sim &\Mod_k  
&
\eta(\cF) = p_*q^! \cF
}
$$
built from the correspondence
$$
\xymatrix{
pt & \ar[l]_-{p}  \BR_{\geq 0} \ar@{^(->}[r]^-{q}  & \BR
}
$$
%
 
\end{example}

\begin{remark}

For a conic Lagrangian subvariety $\Lambda\subset T^*Z$, 
with  antipodal  subvariety  $-\Lambda \subset T^*Z$,
and conic open subspace $\Omega\subset T^*Z$, 
with antipodal subspace $-\Omega\subset T^*Z$,
Verdier duality induces an involutive  equivalence
$$
\xymatrix{
\bD_Z:\mu\Sh_{\Lambda}(\Omega)^{\op} \ar[r]^-\sim & \mu\Sh_{-\Lambda}(-\Omega)
}
$$
\end{remark}


\subsection{Microstalks}

Let $\Lambda \subset T^*Z$  be a closed conic Lagrangian subvariety.

Fix  $\cS=\{Z_\alpha\}_{\alpha\in A}$  a Whitney stratification  of $Z$ such that 
$$
\xymatrix{
\Lambda \subset T^*_\cS Z = \coprod_{\alpha \in A} T^*_{Z_\alpha} Z
}
$$

To each stratum $Z_\alpha  \subset Z$, introduce the frontier 
$\partial (T^*_{Z_\alpha} Z)  = \ol{T^*_{Z_\alpha} Z} \setminus T^*_{Z_\alpha} Z \subset T^*Z$
of its conormal bundle,
and  the dense, open, smooth locus of their complement
$$
\xymatrix{
(T^*_\cS Z)^\circ = T^*_\cS Z \setminus \bigcup_{\alpha \in A} \partial (T^*_{Z_\alpha} Z) 
}
$$
Introduce 
  the corresponding dense, open, smooth locus
  $$
  \xymatrix{
  \Lambda^\circ =  \Lambda \cap (T^*_\cS Z)^\circ  \subset \Lambda
  }
  $$
Note that $\Lambda^\circ$ depends on the stratification $\cS$, in particular refining $\cS$ leads to smaller  $\Lambda^\circ$, but it is always 
 dense, open, and smooth within $\Lambda$.

Fix a point $(z, \xi) \in \Lambda^\circ$.
Let $B\subset Z$ be a small open ball around $z\in Z$,
and $f:B\to \BR$ a compatible test function in the sense that that $f(z) = 0$ and $df|_z = \xi$.  
Let $L\subset T^*Z$ be the graph of $df$, and assume that
$L$ intersects $\Lambda^\circ$  transversely at the single point $(z, \xi) \in \Lambda^\circ$.
Note that from the graph $L\subset T^*Z$, we can reconstruct the ball $B = \pi(L) \subset Z$, and the compatible test function $f:B\to \BR$ by recalling $f(z) = 0$ and $L\subset T^*Z$ is the graph of $df$.

 \begin{defn}
 Let $\Omega\subset T^*Z$ be a  conic open subspace containing $(z, \xi) \in \Lambda^\sm$.
 
 Define the {\em microstalk} along $L\subset T^*X$  to be the vanishing cycles
  $$
 \xymatrix{
 \phi_L:\mu\Sh_\Lambda^\un(\Omega)\ar[r] & \Mod_k
&
 \phi_L(\cF) = \Gamma_{\{f\geq 0\}}(B, \tilde \cF|_B) 
 }
 $$
 where $\tilde \cF \in \Sh_\Lambda^\un(B, \Omega_B)$ represents the restriction of $\cF\in \mu\Sh_\Lambda^\un(\Omega)$
 to a small conic open neighborhood $\Omega_B\subset \Omega$ of the point $(z, \xi) \in T^*Z$.
  \end{defn}

\begin{remark}
The microstalk is well-defined since by construction it vanishes on the kernel of the localization $\Sh_\Lambda^\un(B, \Omega_B)\to \mu\Sh_\Lambda^\un(\Omega_B)$ with respect to singular support.
\end{remark}

\begin{remark}
Let us view the graph $L\subset T^*Z$
of the compatible test function $f:B\to \BR$
as a small Lagrangian ball centered at $(x, \xi) \in T^*Z$.
Introduce the symplectic tangent space $V = T_{(z, \xi)} T^*Z$ 
and its Lagrangian subspaces
$$
\xymatrix{
\ell = T_{(z, \xi)} L
&
F = T_{(z, \xi)} (T^*_z Z)
&
\lambda =  T_{(z, \xi)} \Lambda^\circ
}$$ 

As a small Lagrangian ball centered at $(x, \xi) \in T^*Z$,
the only constraints on $L\subset T^*Z$
 are that $\ell$ is transverse to both $F$ and $\lambda$ inside of $V$.
We can regard $\lambda$ as the graph
of a symmetric linear map $F\to \ell \simeq F^*$, and hence as a quadratic form  $q\in \on{Sym}^2(F^*)$
to which we can assign the index $i(q) = \#\{\mbox{negative eigenvalues}\}$. 
If $Z_\alpha \subset Z$ is the stratum containing $z\in Z$,  
then the restriction $f|_{B\cap Z_\alpha}$ has a Morse singularity at $z\in Z_\alpha$ of  the form
 $ - \sum_{i=1}^{i(q)} z_i^2 + \sum_{i=i(q) +1}^{\dim Z_\alpha} z_i^2$.

The microstalk $\phi_L$ only depends on the Hamiltonian isotopy class of $L\subset T^*Z$
as a small Lagrangian ball centered at $(x, \xi) \in T^*Z$.
The (not necessarily contractible) connected components of the  Hamiltonian isotopy classes are parameterized by the index
$i(q)$. 
 For two such graphs $L, L'\subset T^*Z$, one can find an equivalence of the shifted functors
 $$
 \xymatrix{
 \phi_{L'}[i(q)] \simeq \phi_{L}[i(q')]
 }
 $$
\end{remark}

\begin{lemma}\label{lemma: microstalks}
An object $\cF\in \mu\Sh_\Lambda^\un(\Omega)$ is trivial if and only if all of its microstalks
are trivial.
An object $\cF\in \mu\Sh_\Lambda^\un(\Omega)$
 lies in the full subcategory 
$\mu\Sh_\Lambda(\Omega)\subset \mu\Sh_\Lambda^\un(\Omega)$ if and only if all of its microstalks 
are perfect $k$-modules.
\end{lemma}

\begin{proof}
It suffices to prove the assertion locally in $\Omega$. Then as discussed in Remark~\ref{rem:cases},
we may place ourselves in the generic situation of Case 2) outlined in Section~\ref{s:microsh} above.
Adopting the notation therein, the localization functor 
restricts to an equivalence
$$
\xymatrix{
\Sh^\un_{\Lambda}(B)^0_*\ar[r]^-\sim  & \mu \Sh^\un_{\Lambda}(\Omega)
}
$$
Thus it suffices to prove the assertions for
an object $\cF\in \Sh^\un_{\Lambda}(B)^0_*$.

The condition that $\Gamma(B, \cF) \simeq 0$ implies that $\cF$ vanishes on the component of the complement of the front projection $Y \subset B$ to which the coorientation points. Crossing to other components in the complement of $Y\subset B$, we see that any stalk of $\cF$ is an inductive extension of  its microstalks. This implies the assertions.
\end{proof}


\subsection{Wrapped microlocal sheaves}

Now we arrive at the main objects of study of this paper.

 \begin{defn}Fix $\Lambda \subset T^*Z$   a closed conic Lagrangian subvariety,
 and let $\Omega\subset T^*Z$ be a conic open subspace. 
 
 Define the small dg category  $\mu\Sh_\Lambda^w(\Omega)$ of {\em wrapped microlocal sheaves} on $\Omega$ supported along $\Lambda$ to be the full dg subcategory
 of   compact objects within the dg category $\mu\Sh_\Lambda^\un(\Omega)$ of large microlocal sheaves.
  \end{defn}

\begin{remark}
Given the full dg subcategory $\cC_c\subset \cC$ of compact objects in a stable cocomplete dg category,
the canonical map is an equivalence $ \Ind \cC_c \risom \cC$. Thus we have 
$$
\xymatrix{ 
\Ind \mu\Sh_\Lambda^w(\Omega) \ar[r]^-\sim & \mu\Sh_\Lambda^\un(\Omega) 
}
$$
\end{remark}

 The notion of compact objects in the above definition is useful in theoretical developments, but there is a more concrete geometric characterization. 
 Recall the microstalk functors
  $
   \phi_L:\mu\Sh_\Lambda^\un(\Omega)\to \Mod_k
 $
 introduced in the previous section.
 Note that $\phi_L$ preserves products, hence admits a left adjoint
 $
   \phi^\ell_L: \Mod_k\to \mu\Sh_\Lambda^\un(\Omega),
 $
 and also preserves coproducts, hence the left adjoint $ \phi^\ell_L$ preserves compact objects.
 
\begin{defn}
Define the {\em microlocal skyscraper} $\cF_L = \phi_L^\ell(k)  \in \mu\Sh_\Lambda^w(\Omega)$ to be the object
 corepresenting the microlocal stalk
$$
\xymatrix{
\phi_L(\cF) \simeq \Hom(\cF_L, \cF) &
\cF\in  \mu\Sh_\Lambda^\un(\Omega)
}
$$
\end{defn}

\begin{lemma}\label{lem:splitgen}
The dg category  $\mu\Sh_\Lambda^w(\Omega)$ of wrapped microlocal sheaves is split-generated by the 
 microlocal skyscrapers $\cF_L\in \mu\Sh_\Lambda^w(\Omega)$.
\end{lemma}

\begin{proof}
By Lemma~\ref{lemma: microstalks}, the  microlocal skyscrapers $\cF_L\in \mu\Sh_\Lambda^w(\Omega)$
compactly generate  $\mu\Sh_\Lambda^\un(\Omega) \simeq \Ind \mu\Sh_\Lambda^w(\Omega) $.
Thus we may invoke the general fact that  if a collection  of objects of a small stable dg category  $\cC_c$ generates the ind-category
$\cC \simeq \Ind \cC_c$, 
then it split-generates $\cC_c$.
\end{proof}

Now let us record some of the evident functoriality enjoyed by wrapped microlocal sheaves.

First, recall that for conic open subspaces $\Omega \subset T^*Z$, the dg category $\mu\Sh_\Lambda^\un(\Omega)$ of large microlocal sheaves
is the sections  of a sheaf $\mu\Sh_\Lambda^\un$ of dg categories supported along $\Lambda$.
 Observe that for an inclusion $\Omega_{Z'} \subset \Omega$ of conic open subspaces,  the restriction map 
 $
\rho: \mu\Sh_\Lambda^\un(\Omega)\to \mu\Sh_\Lambda^\un(\Omega_{Z'})
  $
preserves products,
hence admits a left adjoint
 $
 \rho^\ell: \mu\Sh_\Lambda^\un(\Omega_{Z'}) \to \mu\Sh_\Lambda^\un(\Omega_{Z'}),
 $
and also preserves coproducts, hence the left adjoint $\rho^\ell$ preserves compact objects.
Thus the restriction of $\rho^\ell$ to compact objects provides a natural corestriction functor
$$
\xymatrix{
 \rho^w: \mu\Sh_\Lambda^w(\Omega_{Z'}) \ar[r] & \mu\Sh_\Lambda^w(\Omega_{Z})
 }
 $$

\begin{prop}

The dg categories $\mu\Sh_\Lambda^w(\Omega)$, for conic open subspaces $\Omega \subset T^*Z$, and corestriction functors
$\rho^w\mu\Sh_\Lambda^w(\Omega) \to \mu\Sh_\Lambda^w(\Omega'_Z)$, for inclusions $\Omega \subset \Omega'$, 
 assemble into a cosheaf $\mu\Sh_\Lambda^w$  of dg categories supported along $\Lambda$. 
   Furthermore, there exists a Whitney stratification of $\Lambda$ such that the restriction of $\mu\Sh^w_{\Lambda}$  to each stratum is locally constant. 
\end{prop}

\begin{proof}
Fix a conic open subspace $\Omega \subset T^*Z$.

Fix a cover $\{\Omega_i\}_{i\in I}$ by conic open subspaces $\Omega_i \subset \Omega$.

For the first assertion, it suffices to show that the natural map
$$
\xymatrix{
\colim_{i\in I} \mu\Sh_\Lambda^w(\Omega_i)\ar[r] & \mu\Sh_\Lambda^w(\Omega)
}
$$
 is an equivalence.
Colimits in the $\oo$-category of small stable dg categories and exact functors are calculated by taking the compact objects in the colimit of the induced diagram  in the $\oo$-category of  cocomplete stable dg categories
and continuous functors. Thus it suffices to show the natural map 
$$
\xymatrix{
\colim_{i\in I} \mu\Sh_\Lambda^\un(\Omega_i) \ar[r] &  \mu\Sh_\Lambda^\un(\Omega)\
}
$$
 is an equivalence.
To  calculate colimits of cocomplete stable dg categories
and continuous functors, we may pass to the opposite $\oo$-category of 
cocomplete stable dg categories and right adjoints.
 Thus it suffices to show the natural map 
$$
\xymatrix{
\lim_{i\in I} \mu\Sh_\Lambda^\un(\Omega_i)  & \ar[l]  \mu\Sh_\Lambda^\un(\Omega)\
}
$$
 is an equivalence.
 Since the corestriction maps $\rho^\ell$ were defined to be the left adjoints to the restrictions maps $\rho$, 
  this is precisely a limit diagram satisfied by the sheaf $\mu\Sh^\un_\Lambda$.
  
  Finally, for the second assertion, recall 
   there exists a Whitney stratification of $\Lambda$ such that the restriction of $\mu\Sh^\un_{\Lambda}$  to each stratum is locally constant. In other words,  for small conic open neighborhoods $\Omega\subset T^*Z$
    along a stratum,
   the restriction maps between the sections $\mu\Sh^\un_{\Lambda}(\Omega)$
  are equivalences. 
   Thus  the left adjoints of the restrictions are the inverse equivalences,
   and the full dg subcategory $\mu\Sh^w_{\Lambda}(\Omega)$ of compact objects 
   is thus locally constant along the stratum.
    \end{proof}
    
    Next, recall that given a closed embedding of conic Lagrangian subvarieties $\Lambda'\subset \Lambda$, there
    is a natural full embedding $i:\mu\Sh^\un_{\Lambda'}\to \mu\Sh^\un_{\Lambda}$ of sheaves of dg categories.
    Observe that $i$ preserves products,
hence admits a left adjoint
 $
 i^\ell: \mu\Sh_\Lambda^\un \to \mu\Sh_{\Lambda'}^\un,
 $
and also preserves coproducts, hence the left adjoint $i^\ell$ preserves compact objects.
Thus the restriction of $i^\ell$ to compact objects provides a natural essentially surjective functor
$$
\xymatrix{
 i^w: \mu\Sh_\Lambda^w \ar[r] & \mu\Sh_{\Lambda'}^w
 }
 $$

The evaluation of $i^w$  on a microlocal skyscraper $\cF_L  \in \mu\Sh_\Lambda^w(\Omega)$ 
is straightforward. 
On the one hand, if the small Lagrangian ball $L\subset T^*Z$ is centered at a point $(x,\xi)\in \Lambda^\circ$
that is not contained in $\Lambda' \subset \Lambda$, then $i^w(\cF_L) \simeq 0$. 
On the other hand, if the small Lagrangian ball $L\subset T^*Z$ is centered at a point $(x,\xi)\in \Lambda^\circ$
 contained in $\Lambda' \subset \Lambda$, then $i^w(\cF_L)$ simply represents  the restriction of the microlocal stalk
 functor to sections of  $\mu\Sh_{\Lambda'}^\un \subset  \mu\Sh_\Lambda^\un$.


\subsection{Wrapped constructible sheaves}

We will describe here a special case of the preceding constructions.
 Fix  $\cS=\{Z_\alpha\}_{\alpha\in A}$  a Whitney stratification  of $Z$,
 and recall the conormal Lagrangian 
$ T^*_\cS Z = \coprod_{\alpha \in A} T^*_{Z_\alpha} Z
$

 \begin{defn} 
 Define the small dg category  $\Sh_\cS^w(Z)$ of {\em wrapped constructible sheaves} 
 to be the full dg subcategory  of   compact objects within the dg category $\Sh_\cS^\un(T^* Z)$ of large constructible sheaves.
  \end{defn}
  
  \begin{remark}
  Starting from the canonical equivalence 
  $
  \Sh_\cS^\un(Z) \simeq  \mu\Sh^\un_{T^*_\cS Z}(T^*Z)
  $
 and passing to compact objects, we obtain a canonical equivalence
  $
  \Sh_\cS^w(Z) \simeq \mu \Sh_{T^*_\cS Z}^w(T^*Z).
  $
  Thus we can view  wrapped constructible sheaves as a special case of  wrapped microlocal sheaves.  
  \end{remark}


\subsubsection{Local systems}

Let us first focus on the case of a single stratum $\cS = \{Z\}$.

Let $PZ$ be the Poincar\'e $\oo$-groupoid of $Z$. If we choose a base point $z_i \in Z$ representing each connected component $[z_i]\in \pi_0(Z)$, then there is a canonical equivalence of $\oo$-groupoids 
$$
\xymatrix{
\coprod_{[z_i]\in \pi_0(Z)} \Omega_{z_i} Z \ar[r]^-\sim & PZ
}
$$

Let us introduce the respective versions of large, traditional, and wrapped local systems as
the  functor categories 
$$
\xymatrix{
\Loc^\un(Z)= \Mod_k(PZ)
&
\Loc(Z)= \Perf_k(PZ)
&
\Loc^w(Z)=  PZ\on{-Perf}_k
}
$$
Thus a local system $\cL\in\Loc^\un(Z)$ lies in the small dg subcategory $\Loc(Z) \subset \Loc^\un(Z)$
  if and only if for each $z_i \in Z$ the restriction $\cL(z_i)$ is perfect as a 
$k$-module.
A local system $\cL\in\Loc^\un(Z)$ lies in the small dg subcategory $\Loc^w(Z) \subset \Loc^\un(Z)$
of  compact objects  if and only if for each $z_i \in Z$ the restriction $\cL(z_i)$ is perfect as a 
$ \Omega_{z_i}Z$-module.

We have  canonical equivalences
$$
\xymatrix{
\Sh_Z^\un(Z) \simeq \Loc^\un(Z)
&
\Sh_Z(Z) \simeq \Loc(Z)
&
\Sh_Z^w(Z) \simeq \Loc^w(Z)
}
$$

\begin{example}
For a specific example, take the circle $Z= S^1$. 

Then we have $PS^1 \simeq \Omega S^1 \simeq \BZ$, and set $\G_m = \Spec k[\BZ]$.

We have the canonical equivalences
$$
\xymatrix{
\Loc^\un(Z)= \QCoh(\G_m)
&
\Loc(Z)= \Coh_\proper(\G_m)
&
\Loc^w(Z)=   \Coh(\G_m)
}
$$
where  $\Coh_\proper(\G_m) \subset \QCoh(\G_m)$ denotes the full dg subcategory of 
coherent complexes with proper support.

\end{example}


\subsubsection{Stratifications}

Now we consider the general case
of a Whitney stratification $\cS=\{Z_\alpha\}_{\alpha\in A}$.

Let $P_\cS Z$ be the exit path $\oo$-category of $\cS$. For any point $z\in Z$,
its automorphisms are the based loop space $\Aut_{P_\cS}(z) \simeq \Omega_{z} Z_\alpha$ of the stratum
$Z_\alpha \subset Z$ containing $z\in Z$.

We have  canonical equivalences
$$
\xymatrix{
\Sh_\cS^\un(Z) \simeq \Mod_k(P_\cS Z)
&
\Sh_\cS(Z) \simeq \Perf_k(P_\cS Z)
&
\Sh_\cS^w(Z) \simeq P_\cS Z\on{-\Perf}_k
}
$$

Thus a sheaf $\cF\in\Sh_\cS^\un(Z)$ lies in the small dg subcategory $\Sh_\cS(Z) \subset \Sh_\cS^\un(Z)$
  if and only if for each $z \in Z$, the stalk $\cF_z$ is perfect as a 
$k$-module.
A sheaf $\cF\in \Sh_\cS^\un(Z)$ lies in the small dg subcategory $\Sh_\cS^w(Z)\subset \Sh_\cS^\un(Z)$
of  compact objects  if and only if for each $z \in Z$, the stalk $\cF_z$ is perfect as a 
$ \Omega_{z}Z_\alpha$-module where $Z_\alpha \subset Z$ is the stratum containing $z\in Z$. 

To see this last assertion, note that we have the compact generators $i_{\alpha!} \cL_\alpha \in \Sh_\cS^w(Z)$,
 for the
inclusion  $i_\alpha:Z_\alpha\to Z$ of a stratum, and $\cL_\alpha = p_{\alpha !} k_{\tilde Z_\alpha} \in \Loc^w(Z_\alpha)$ the universal local system coming from the universal cover $p_\alpha:Z_\alpha \to Z$. Their
stalks $(i_{\alpha!} \cL_\alpha)_z$ vanish when $z\not \in Z_\alpha$, and for $z\in Z_\alpha$ are 
the regular
$ \Omega_{z}Z_\alpha$-module.

\begin{example}\label{ex:cylinder}
For a specific example, take the circle $Z= T^1 = \BR/\BZ$, with stratification $\cS_1 = \{\{0\},  (0, 1)\}$,
and write $i:\{0\} \to T^1$, $j:(0, 1) \to T^1$ for the inclusions.

Let $K_1$ be the Kronecker quiver category with two objects $a, b$, and two non-identity morphisms $x, y:a\to b$.
 There is a canonical equivalence  $K_1\risom P_{\cS_1} T^1$ sending $a$ to the point $0 \in T^1$, $b$ to the point $1/2\in T^1$,
$x$ to the exit-path $x:[0,1] \to T^1, x(t) = t/2$, $y$ to the exit-path $y:[0,1] \to T^1, y(t) = -t/2$

Thus we have  canonical equivalences
$$
\xymatrix{
\Sh_{\cS_1}^\un(T^1) \simeq \Mod_k(K_1)
&
\Sh_{\cS_1}(T^1) \simeq \Perf_k(K_1)
&
\Sh_{\cS_1}^w(T^1) \simeq K_1\on{-\Perf}_k
}
$$
with an additional equivalence $\Perf_k(K_1) \simeq K_1\on{-\Perf}_k$.

If we set $\BP^1 = \on{Proj} k[x, y]$, then we can rewrite the above as mirror equivalences
$$
\xymatrix{
\Sh_{\cS_1}^\un(T^1) \simeq \QCoh(\BP^1)
&
\Sh_{\cS_1}(T^1) \simeq \Perf(\BP^1)
&
\Sh_{\cS_1}^w(T^1) \simeq \Coh(\BP^1)
}
$$
matching the objects $i_! k_{\{0\}}, j_! \cD_{(0,1)}$ to the respective objects $\cO_{\BP^1}, \cO_{\BP^1}(-1)$.
%
 
 Consider as well the microlocal stalks for the Lagrangian graphs $L_\pm = \Gamma_{\pm d\theta} \subset T^* T^1$ centered at $(0, \pm 1) \in T^*T^1$. They are calculated by the respective hyperbolic restrictions
$$
\xymatrix{
\eta_\pm:\Sh_{\cS_1}^\un(T^1)\ar[r] &  \Mod_k &
\eta_\pm(\cF) = p_* q_\pm^!\cF
}
$$
built from the respective correspondences
 $$
 \xymatrix{
 pt & \ar[l]_-p [0, 1/2) \ar[r]^-{q_+} & T^1
&
 pt & \ar[l]_-p (-1/2, 0] \ar[r]^-{q_-} & T^1
 }
 $$
Note that $\eta_+$, respectively $\eta_-$, vanishes on an object if and only if the  quiver map $y$, respectively $x$, is invertible, 
and they both satisfy  $\eta_\pm(i_! k_{0}) \simeq k$.
Under the mirror equivalences, they
correspond to the $*$-restrictions to the respective poles 
$$
\xymatrix{
f_\pm^*: \QCoh(\BP^1) \ar[r] &  \Mod_k
&
f_+:\{y = 0\} \ar[r] & \BP^1
&
f_-:\{x = 0\} \ar[r] & \BP^1}
$$

%
%
%
%
%
%
%
%
%
%

\end{example}


\subsection{Duality}

This section presents a duality between  microlocal sheaves and wrapped microlocal sheaves
in the form of Theorem~\ref{thm:duality}. The proof we give
is an  application of the theory developed in~\cite{Narb, Nexp}.
For the reader unfamiliar with~\cite{Narb, Nexp}, we will largely use it as a black box,
and also will not appeal to the results of this section elsewhere in the paper.

Let $\Lambda \subset T^*Z$  be a closed conic Lagrangian subvariety,
 and $\Omega\subset T^*Z$  a conic open subspace.

\begin{thm}\label{thm:duality}

The natural hom-pairing provides an equivalence
$$
\xymatrix{
\mu\Sh_\Lambda(\Omega)  \ar[r]^-\sim & \Fun^{ex}(\mu\Sh_\Lambda^w(\Omega)^{op}, \Perf_k) 
}
$$
where $ \Fun^{ex}$ denotes the dg category of exact functors, and $\Perf_k$ that of perfect $k$-modules.
\end{thm}

\begin{remark}\label{rem:asymmetry}
We caution the reader that while objects of $\mu\Sh_\Lambda^w(\Omega)$ similarly give functionals on
$\mu\Sh_\Lambda(\Omega)$, it is not true that they produce all possible functionals. One could think about the 
specific example 
where $\Lambda = T^1 \subset T^*T^1$ is the zero-section, and $\Omega = T^*T^1$ is the entire cotangent bundle.
Then we have seen
 $\mu\Sh_{T^1}(T^*T^1) \simeq \Perf_\proper(\G_m)$
 and  $\mu\Sh^w_{T^1}(T^*T^1) \simeq \Coh(\G_m)$, and
by~\cite{BNP}, the hom-pairing gives an equivalence
$$
\xymatrix{
 \Perf_\proper(\G_m)  \ar[r]^-\sim & \Fun^{ex}(\Coh(\G_m)^{op}, \Perf_k) 
}
$$
compatibly with Theorem~\ref{thm:duality}.
But clearly there are more functionals on $\Perf_\proper(\G_m)$ than those coming from $\Coh(\G_m)$ alone.
For example, one could take the hom-pairing with a direct sum of skyscraper sheaves at infinitely many points.  
\end{remark}

\begin{remark}
If one prefers  not to have opposite categories appear in Theorem~\ref{thm:duality}, one can take opposite categories on both sides to obtain
$$
\xymatrix{
\mu\Sh_\Lambda(\Omega)^{op}  \ar[r]^-\sim & \Fun^{ex}(\mu\Sh_\Lambda^w(\Omega)^{op}, \Perf_k)^{op}
\simeq \Fun^{ex}(\mu\Sh_\Lambda^w(\Omega), \Perf_k^{op})
}
$$
Then Verdier duality for $\mu\Sh_\Lambda(\Omega)$ identifies its opposite category with
$\mu\Sh_{-\Lambda}(-\Omega)$, and tensor duality for $\Perf_k$ identifies its opposite category with itself.
In this way, one obtains  a canonical equivalence
$$
\xymatrix{
\mu\Sh_{-\Lambda}(\Omega)  \ar[r]^-\sim & \Fun^{ex}(\mu\Sh_\Lambda^w(\Omega), \Perf_k) 
}
$$
\end{remark}


\begin{proof}[Proof of Theorem~\ref{thm:duality}]
First, let us observe that it suffices to prove the assertion locally. Namely, if we choose a cover $\{\Omega_i\}_{i\in I}$ by conic open subspaces $\Omega_i \subset \Omega$, then since $\mu\Sh_\Lambda(\Omega)$
is a sheaf and $\mu\Sh_\Lambda^w(\Omega)$ a cosheaf we have
$$
\xymatrix{
\mu\Sh_\Lambda(\Omega)   \simeq \lim_{i\in I} \mu\Sh_\Lambda(\Omega_i)
&
\mu\Sh^w_\Lambda(\Omega)   \simeq \colim_{i\in I} \mu\Sh^w_\Lambda(\Omega_i)
}
$$
Thus if we have the assertion locally
$$
\xymatrix{
\mu\Sh_\Lambda(\Omega_i)  \ar[r]^-\sim & \Fun^{ex}(\mu\Sh_\Lambda^w(\Omega_i)^{op}, \Perf_k)
}
$$ 
then we have it globally
$$
\xymatrix{
\mu\Sh_\Lambda(\Omega)  
 \simeq \lim_{i\in I} \mu\Sh_\Lambda(\Omega_i)
 \ar[r]^-\sim & \lim_{i\in I} \Fun^{ex}(\mu\Sh_\Lambda^w(\Omega_i)^{op}, \Perf_k) 
}
 $$
$$
\xymatrix{
 \simeq \Fun^{ex}(\colim_{i\in I} \mu\Sh_\Lambda^w(\Omega_i)^{op}, \Perf_k)
\simeq   \Fun^{ex}(\mu\Sh_\Lambda^w(\Omega)^{op}, \Perf_k) 
}
$$ 
(If one were to try to swap the roles of 
$\mu\Sh_\Lambda(\Omega)$
and $\mu\Sh_\Lambda^w(\Omega)$ in the theorem, the above middle equivalence  is what fails and  produces the asymmetry noted in Remark~\ref{rem:asymmetry}.)

Now following Remark~\ref{rem:cases}, we may assume that $\Omega \subset T^*Z$ is the cone over a small open ball $\Omega^\oo \subset S^\oo Z$ centered at a point of $\Lambda^\oo \subset S^\oo Z$.

Invoking the theory developed in~\cite{Narb, Nexp}, we may deform $\Lambda^\oo \subset S^\oo Z$ 
to a Legendrian subvariety $\Lambda^\oo_{\arbor} \subset S^\oo Z$ with arboreal singularities. Taking $\Lambda_\arbor\subset T^* Z$ to be the cone over $\Lambda^\oo_{\arbor} \subset S^\oo Z$,
we have an equivalence $\mu\Sh^\un_\Lambda(\Omega)   \simeq \mu\Sh^\un_{\Lambda_{\arb}}(\Omega)$, restricting to an equivalence $\mu\Sh_\Lambda(\Omega)   \simeq \mu\Sh_{\Lambda_{\arb}}(\Omega)$.
Passing to compact objects in the first equivalence, we obtain an additional equivalence
$\mu\Sh^w_\Lambda(\Omega)   \simeq \mu\Sh^w_{\Lambda_{\arb}}(\Omega)$.

Thus  we may assume that the conic Lagrangian subvariety $\Lambda \subset T^* Z$ has arboreal singularities.
Applying again the argument that it suffices to prove the assertion locally,
we may further assume that  
$\Omega \subset T^*Z$ is the cone over a small open ball $\Omega^\oo \subset S^\oo Z$ centered at a point of $\Lambda^\oo \subset S^\oo Z$ that is an arboreal singularity. In this situation, the calculations of~\cite{Narb, Nexp}
show that $\mu\Sh^\un_\Lambda(\Omega) $ is equivalent to the dg category $\Mod_k(\cT)$ of modules over a directed tree $\cT$ (a finite nonempty connected acyclic graph with an orientation of its edges), and $\mu\Sh_\Lambda(\Omega) $  is equivalent to the dg category $\Perf_k(\cT)$ of perfect modules. Passing to compact objects  under the first equivalence, 
we  have that  $\mu\Sh^w_\Lambda(\Omega) $ is also equivalent to $\Perf_k(\cT)$.

Finally, for perfect modules over a directed tree, it is straightforward to check directly that the hom-pairing provides an equivalence
$$
\xymatrix{
\Perf_k(\cT) \ar[r]^-\sim & \Fun^{ex}(\Perf_k(\cT)^{op}, \Perf_k) 
}
$$
One can also invoke the fact that the associated path algebra is smooth and proper.
This concludes the proof of the theorem.
\end{proof}


\section{Surfaces}

To provide a collection of simple examples of the preceding theory, in this section we
calculate traditional and wrapped microlocal sheaves on surfaces.

Let $\Sigma$ be an exact symplectic surface, with symplectic form $\omega_\Sigma$ and Liouville form $\alpha_\Sigma$.

By a support Lagrangian $\Gamma \subset \Sigma$, we will mean a locally finite, closed embedded graph, with  vertex set denoted by $V$ and edge set by $E$, admitting a continuous function $f:\Gamma\to \BR$,  such that for each edge $e\in E$, 
the restriction $f|_e$ is differentiable with $d(f|_e) = \alpha_\Sigma|_e$. 
We will allow vertices of any non-zero valence,  edges incident to the same vertex at each end, half-infinite edges incident to a single vertex, as well as infinite edges incident to no vertices. We exclude  circular edges  incident to no vertices for notational convenience; 
the geometry will be unchanged by picking a point  on each such an edge to serve as a vertex.

We will assume the following local tameness: for each vertex $v\in V$, there is a small open ball $B_v\subset \Sigma$ around $v$
and a diffeomorphism 
$$
\xymatrix{
\psi:B_v\ar[r]^-\sim &  \BC
}
$$ 
such that $\psi(\Gamma \cap B_v) \subset \BC$ is the union
of distinct closed rays 
$$
\xymatrix{
\BR_{\geq 0} \cdot e^{i\theta_a} \subset \BC &  \theta_a\in \BR/2\pi\BZ,  a\in |v|
}
$$

\subsection{Combinatorial Fukaya categories}

%
%
%


Given a vertex $v\in V$, denote by $|v|\subset E$ the set of edges incident to $v$,
and by $\# v$ the valence of $|v|$.
Let us view the union $ V\coprod E$ as a poset with relation $v< e$, for $v\in V$, $e\in E$, with $e\in |v|$,
and otherwise distinct elements are incomparable. 

Introduce the poset $\cP$ given by ordered chains in $V\coprod E$.
Its elements are the union of
the vertices $V$, edges $E$, and the subset $U \subset V \times E$ of pairs $(v, e)$ with $e\in |v|$.
Its order relation is given by
$
(v, e) <v, (v, e) <e,
$
for $ v\in V$, $e\in E$, $(e, v) \in U$, and otherwise distinct elements are incomparable.

To each vertex $v\in V$, fix a small open ball $B_v \subset \Sigma$ around $v$,
and to each
edge vertex $e\in E$, fix a small open ball $B_e \subset \Sigma$ around $e$. Arrange so that each intersection
$B_v \cap B_{v'}$, $B_e \cap B_{e'}$ is empty when $v \not = v'$, $e\not = e'$, and the intersection
$B_{(v, e)} = B_v\cap B_e$ is itself a ball if $e\in |v|$, and otherwise empty.

We will only care about $\Sigma$ in a small open neighborhood of $\Gamma$,
and so will assume that it is itself the union
$$
\xymatrix{
\Sigma = (\bigcup_{v\in V} B_v) \cup (\bigcup_{e\in E} B_e)
}
$$
Thanks to the intersection properties of the open balls,
we have a homotopy colimit diagram
$$
\xymatrix{
\colim_{ p \in \cP} B_p \ar[r]^-\sim & \Sigma
}
$$
which can be  more concretely written as the homotopy pushout diagram
$$
\xymatrix{
\ar[d] \bigcup_{(v, e) \in U} B_{(v, e)} \ar[r] & \bigcup_{v\in V}  B_{v} \ar[d]  \\
\bigcup_{e \in E} B_{e}  \ar[r] & \Sigma
}
$$

Regard the poset $\cP$ as a category.
Its objects are the union of
the vertices $V$, edges $E$, and the subset $U \subset V \times E$ of pairs $(v, e)$ with $e\in |v|$.
Its non-identity morphisms are given by
$
(v, e) \to v, (v, e) \to e,
$
for $ v\in V$, $e\in E$, $(e, v) \in U$,
and there are no non-identity compositions requiring definition.

Let $\Lambda$ denote the cyclic category.
Introduce the functor 
$$
\xymatrix{
\cO:\cP\ar[r] & \Lambda &
 \cO(p) = \pi_0((\Sigma \setminus \Gamma) \cap B_p)
}
$$
where the finite set $\cO(p)$ is given the natural cyclic order induced by the orientation of $\Sigma$.

By the results of~\cite{DK, Ncyc}, we have a functor
$$
\xymatrix{
\cC_{st}: \Lambda^{op}\ar[r] &  \BZ/2\on{-}\dgst_k
}
$$
of which we will make use of the following properties. 

For $n =1, 2,\ldots$, let $\Lambda_n \in \Lambda^{op}$ denote the 
standard cycle of $n+1$ elements  $\{ 1 \to 2\to  \cdots \to n\to n+1 \to 1\}$.
 Let $A_{n}$ be the directed quiver $1 \to 2 \to \cdots \to n$. For each $a\in \{1, \ldots, n\}$,  let 
 $P_a, k_a, I_a \in A_{n}\on{-\Perf}_k$ denote the respective projective, skyscraper, and injective
 based at $a$.

The choice of a linear ordering underlying the cyclic ordering on $\Lambda_n$
induces  a canonical equivalence of $\BZ/2$-dg categories
$$
\xymatrix{
\cC_{st}(\Lambda_n) \simeq (A_{n}\on{-\Perf}_k)_{\BZ/2}
}
$$

Moreover, for $a\in \{2, \ldots, n+1\}$,  the inclusion of the subcycle $i_a:\Lambda_1 =\{a\to a+1\to a\} \to \Lambda_n$ of two successive elements is taken to the quotient functor
 $$
\xymatrix{
i_{a}^!:(A_{n}\on{-\Perf}_k)_{\BZ/2} \ar[r] &   (A_{2}\on{-\Perf}_k)_{\BZ/2} \simeq ( \Perf_k)_{\BZ/2} 
}
$$ 
sending 
the injective $I_{a} \in A_{n}\on{-\Perf}_k$ 
to $k\in \Perf_k$,
and killing all of the other
injectives  $I_{a'}   \in A_{n}\on{-\Perf}_k$  with $a'\not = a$.

For $a= 1$,
  the inclusion of the subcycle $i_{1}:\Lambda_1 =\{1\to  2\to 1\} \to \Lambda_n$   is taken to the quotient functor  
$$
\xymatrix{
i_{1}^!:(A_{n}\on{-\Perf}_k)_{\BZ/2} \ar[r] &   (A_{2}\on{-\Perf}_k)_{\BZ/2} \simeq ( \Perf_k)_{\BZ/2} 
}
$$ 
sending 
the injective-projective  $I_{\#v-1} \simeq P_1 \in A_{n}\on{-\Perf}_k$
to $k[-1]\in \Perf_k$,
and killing all of the
 other projectives $P_{a'} \in A_{n}\on{-\Perf}_k$  with $a'\not = 1$.

Note that for $a\in \{2, \ldots, n+1\}$, we have the fully faithful left adjoint
$$
\xymatrix{
i_{a!}: ( \Perf_k)_{\BZ/2} 
 \simeq (A_{2}\on{-\Perf}_k)_{\BZ/2} 
\ar[r] &  (A_{n}\on{-\Perf}_k)_{\BZ/2}
&
i_{a!} k = k_a}
$$ 
and for $a=1$, we have the fully faithful left adjoint 
 $$
\xymatrix{
i_{1!}:( \Perf_k)_{\BZ/2} \simeq (A_{2}\on{-\Perf}_k)_{\BZ/2} \ar[r] &   (A_{n}\on{-\Perf}_k)_{\BZ/2}
}
$$
$$
\xymatrix{
i_{1!}(k) = I_{\#v-1}[1] \simeq P_1[1] \simeq [k_1 \risom k_2 \risom \cdots \risom k_n][1]
}
$$

Passing to left adjoints,  we thus obtain an additional functor
$$
\xymatrix{
\cC^w_{st}:   \Lambda\ar[r] &  \BZ/2\on{-}\dgst_k
}
$$

\begin{defn} 1) Define the {\em combinatorial infinitesimal and wrapped Fukaya functors} of $\Sigma$ 
supported along 
$\Gamma$ to be the respective compositions
$$
\xymatrix{
\cF_{\Gamma} = \cC_{st} \circ \cO: \cP^{op} \ar[r] &  \BZ/2\on{-}\dgst_k
&
\cF_{\Gamma}^w = \cC^w_{st} \circ \cO: \cP \ar[r] &  \BZ/2\on{-}\dgst_k
}
$$

2)
Define the {\em combinatorial infinitesimal and wrapped Fukaya categories}  to be the respective $\BZ/2$-dg categories
$$
\xymatrix{
\cF_\Gamma(\Sigma) = \lim_{\cP^{op}} \cO_{st} 
&
\cF^w_\Gamma(\Sigma) = \colim_{\cP} \cO^w_{st} 
}
$$
\end{defn}

\begin{remark}
Recall for $\cF = \cF_{\Gamma}$ or $\cF_{\Gamma}^w$, we have equivalences
of  $\BZ/2$-dg categories
$$
\xymatrix{
\cF(v) \simeq (A_{\# v - 1}\on{-\Perf}_k)_{\BZ/2}
&
\cF(e) \simeq (\Perf_k)_{\BZ/2} & \cF(v, e) \simeq (\Perf_k)_{\BZ/2} 
}
$$ 
 
 Thus we have  respective pullback and pushout diagrams 
$$
\xymatrix{
\oplus_{(v, e) \in U} (\Perf_k)_{\BZ/2} & \ar[l]  \oplus_{v\in V}  (A_{\# v  - 1}\on{-\Perf}_k)_{\BZ/2}     \\
\ar[u] \oplus_{e \in E}( \Perf_k )_{\BZ/2}  & \ar[l]  \cF_\Gamma(\Sigma)  \ar[u]& 
}
$$
$$
\xymatrix{
\ar[d] \oplus_{(v, e) \in U} (\Perf_k)_{\BZ/2}  \ar[r] & \oplus_{v\in V}  (A_{\# v - 1}\on{-\Perf}_k)_{\BZ/2}    \ar[d]  \\
\oplus_{e \in E} (\Perf_k)_{\BZ/2}   \ar[r] & \cF^w_\Gamma(\Sigma)
}
$$

Furthermore, we can choose the initial equivalences so that the vertical arrows of the left columns
of the diagram are the 
evident functors. Thus all of the structure of the diagrams is contained in the top rows of horizontal arrows.
\end{remark}


\subsection{Microlocal sheaves via gluing}

Let $N = \Sigma \times \BR$ be the contactification of $\Sigma$ with contact form $\lambda = dt + \alpha_\Sigma$
where $t$ denotes the coordinate on $\BR$.
 
 Let $\cL = \Gamma_{-f} \subset N$ be the Legendrian lift of the exact Lagrangian skeleton $\Gamma\subset \Sigma$
 given by the graph of the negative of a primitive $f:\Gamma\to \BR$.

Consider the plane $\BR^2$ with coordinates $x, y$, and let $\pi:T^*\BR^2\to \BR^2$ be its cotangent bundle
with canonical coordinates $x, y, \xi, \eta$,
and $\pi^\oo :S^\oo \BR^2\to \BR^2$ its spherical projectivization.
 
For each point $\ell\in \cL$, we may find a small open ball $U \subset N$ around $\ell$, and an oriented contact embedding
$$
\xymatrix{
\varphi:U \ar@{^(->}[r] & S^\oo \BR^2
}
$$

Set $\Lambda^\oo = \varphi(\cL\cap U)\subset S^\oo \BR^2$ to be the transported Legendrian.
 We may choose $\varphi$ so that the front projection of $\Lambda^\oo$ is a finite map
$$
\xymatrix{
\Lambda^\oo \ar[r] & Y = \pi^\oo(\Lambda^\oo) \subset \BR^2
}
$$ 
and $\varphi(\ell) = ((0, 0), [0,1]) \in \Lambda^\oo$.

Introduce the conic open subspace
$$
\xymatrix{
\Omega =  \{ ((x, y), [\xi, \eta]) \, |\, \eta>0\} \subset T^*\BR^2
}$$
and 
the conic Lagrangian subvariety
$\Lambda \subset T^* \BR^2$
given by the cone over 
$\Lambda^\oo \subset S^\oo \BR^2$.

Then we have the respective dg categories
$$
\xymatrix{
\mu\Sh^{\un}_{\Lambda_p}(\BR^2, \Omega)
&
\mu\Sh_{\Lambda_p}(\BR^2, \Omega)
&
\mu\Sh^w_{\Lambda_p}(\BR^2, \Omega)
}
$$
of large, traditional, and wrapped microlocal sheaves.

To calculate them, let us consider two cases: $\ell = (v, -f(v))\in \cL$, with $v\in \Gamma$ a vertex,
and $\ell = (w, -f(w))\in \cL$, with $w\in \Gamma$  on an edge $e\in E$.

\medskip

1) When $\ell = (w, -f(w))\in \cL$, with $w\in \Gamma$  on an edge $e\in E$, we may choose $\varphi$ so that 
the front projection is the axis
$$
\xymatrix{
Y =  \{ y = 0\} \subset \BR^2 
}
$$

Then we have  canonical equivalences
$$
\xymatrix{
 \Mod_k\ar[r]^-\sim & \Sh^\un_{\Lambda}(\BR^2)_!^0 \ar[r]^-\sim & \mu\Sh^\un_{\Lambda}(\Omega)
}
$$
induced by sending $k \in \Mod_k$ to the constant sheaf $k_{U_+} \in \Sh^\un_{\Lambda}(\BR^2)_!^0$
on the closed half-space $U_+ = \{y>0\} $.

\medskip

2) When $\ell = (v, -f(v))\in \cL$, with $v\in \Gamma$ a vertex, recall there is a small open ball $B_v\subset \Sigma$ around $v$
and a diffeomorphism 
$$
\xymatrix{
\psi:B_v\ar[r]^-\sim &  \BC
}
$$ 
such that $\psi(\Gamma \cap B_v) \subset \BC$ is the union
of distinct closed rays 
$$
\xymatrix{
\BR_{\geq 0} \cdot e^{i\theta_a} \subset \BC &  \theta_a\in \BR/2\pi\BZ,  a\in |v|
}
$$
Without changing microlocal sheaves, we may apply an isotopy to $\Gamma$ to assume all of the rays are as close as we like
to the positive real ray $\BR_{\geq 0} \subset \BC$.

With this arranged, we may choose $\varphi$ so that  for $e\in |v|$, there are functions
$$
\xymatrix{
g_e:\BR_{\geq 0} \ar[r] & \BR
}
$$
with $g_e(0) = 0$, and $dg_e(x) \to 0$ as $x\to 0$, so that the front projection takes the form
$$
\xymatrix{
Y = \bigcup_{e\in |v|}  Y_{e} \subset \BR^2
}
$$
where $Y_{e} \subset \BR_{\geq 0} \times \BR$ denotes the graph of $g_e$,
and $Y_e$ intersects $Y_{e'}$ at the origin alone, for $e\not = e'$.

Moreover, the linear ordering defined by 
$$
\xymatrix{
e<e' & \iff & g_{e}(x) > g_{e'}(x),  x\in \BR_{>0}
}
$$
 induces the cyclic ordering on $|v|$.

Then we have  canonical equivalences
$$
\xymatrix{
 A_{\# v- 1}\on{-\Mod}_k\ar[r]^-\sim & \Sh^{\un}_{\Lambda}(\BR^2)_*^0 \ar[r]^-\sim & \mu\Sh^\un_{\Lambda}(\Omega)
}
$$
induced by the following.
For $e_a \in |v|$, with $a<\# v$, 
we send the skyscraper module $k_{a} \in A_{\# v- 1}\on{-\Mod}_k$ to the  extension  $i_{a!} k_{U_a} \in \Sh_{\Lambda}^\un(\BR^2)_*^0$
along the  inclusion
$$
\xymatrix{
i_a: U_a = \{ x> 0, \,  g_a(x) > y \geq g_{a+1}(x)\} \ar[r] &
  \BR^2  }
  $$
This also sends the projective $P_a \in A_{\# v- 1}\on{-\Mod}_k$,
for $e_a \in |v|$, with $a<\# v$,  to the extension  $ j_{a!} k_{V_a} \in \Sh_{\Lambda}^\un(\BR^2)_*^0$
along the  inclusion
$$
\xymatrix{
j_a: V_a = \{ x> 0, \,  g_a(x) >  y \geq g_{\# v}(x)\} \ar[r] &
 \BR^2
  }
  $$
It also sends the injective $I_a \in A_{\# v- 1}\on{-\Mod}_k$,
for $e_a \in |v|$, with $a<\# v$,  to the extension  $ h_{a!} k_{W_a} \in \Sh_{\Lambda}^\un(\BR^2)_*^0$
along the  inclusion
$$
\xymatrix{
h_a: W_a = \{ x> 0, \,  g_1(x) >  y \geq g_{a}(x)\} \ar[r] &
 \BR^2
  }
  $$

Going further, consider the open conic subspace 
$$
\xymatrix{
\Omega_+ = \Omega \times_{\BR^2} \{x> 0\}  \subset T^*\BR^2
}
$$
and for $e\in|v|$, the closed conic Lagrangian subspace  
$$
\xymatrix{
\Lambda_e = \Lambda \times_{\BR^2} Y_e \subset T^*\BR^2
}
$$

Consider the natural microlocal restriction
$$
\xymatrix{
A_{\# v- 1}\on{-\Mod}_k \simeq \mu\Sh^\un_\Lambda(\Omega) \ar[r] & \mu\Sh^\un_{\Lambda_e}(\Omega_+) \simeq \Mod_k
}
$$

For $e>1 $, it is the quotient functor  sending 
the injective $I_{e} \in A_{\# v- 1}\on{-\Mod}_k$
to $k\in \Mod_k$,
and killing all of the
other injectives $I_{a} \in A_{\# v- 1}\on{-\Mod}_k$  with $a\not = e$,

For $e= 1$, it is the quotient functor  sending 
the injective-projective $I_{\#v-1} \simeq P_1 \in A_{\# v- 1}\on{-\Mod}_k$
to $k[-1]\in \Mod_k$,
and killing all of the
 other projectives $P_{a} \in A_{\# v- 1}\on{-\Mod}_k$  with $a\not = 1$.

\medskip

Now recall  the poset $\cP$ given by ordered chains in $V\coprod E$.

Fix $p\in \cP$,  and apply the above constructions with $\ell = (v, -f(v))$, if $p = v$, 
or a choice of $\ell = (w, -f(w))$, with $w\in e$, if $p = e$, or $p = (v, e)$.

 Let $\Lambda_p^\oo \subset S^\oo \BR^2$ denote
the resulting Legendrian, 
and $\Lambda_p\subset T^*\BR^2$
the conic Lagrangian subvariety obtained by taking
the cone over $\Lambda_p^\oo \subset S^\oo \BR^2$.

\begin{lemma}\label{lem:inv}
The  respective underlying $\BZ/2$-dg categories 
$$
\xymatrix{
\mu\Sh^{\un}_{\Lambda_p}(\BR^2, \Omega)_{\BZ/2}
&
\mu\Sh_{\Lambda_p}(\BR^2, \Omega)_{\BZ/2}
&
\mu\Sh^w_{\Lambda_p}(\BR^2, \Omega)_{\BZ/2}
}
$$
of large, traditional, and wrapped microlocal sheaves
are canonically independent
of any of the preceding choices.
\end{lemma}

\begin{proof}
By introducing a new vertex along an edge,
the situation when $\ell = (w, -f(w))$ can be viewed as a special case of when $\ell = (v, -f(v))$.
We will focus on the latter situation and leave the former to the reader.

When $\ell = (v, -f(v))$, we must  check what happens if we return to the beginning of the above constructions
and  apply an alternative isotopy to $\Gamma$ to gather all of the edges at $v$  close together.
Then up to homeomorphism, the 
resulting front projection
$$
\xymatrix{
Y = \bigcup_{e\in |v|}  Y_{e} \subset \BR^2
}
$$
will change only possibly by a cyclic permutation of the labels $|v|$.

It suffices to analyze the case where the labels $|v|$ change by a simple cyclic permutation. Suppose with respect to the initial linear ordering $e_1 <e_2 <\cdots< e_{\# v}$,  each $e_a\in |v|$ goes to the next $e_{a+1}\in |v|$, and the maximum $e_{\# v}\in |v|$ goes
to the minimum $e_1\in |v|$.

Then under the corresponding contact transformation, the dg category of microlocal sheaves undergoes the mutation:
 for $a<\# v-1$, each sheaf $i_{a!} k_{U_a}$ goes to the next $i_{a+1!} k_{U_a}$, and for $a = \# v$,
 the sheaf $i_{\# v-1 !} k_{U_{\#v-1}}$ goes to the sheaf $j_{1 !} k_{V_a}[1]$.

Note this mutation is the same as that assigned to the simple cyclic permutation by the functor 
$\cC_{st}:\Lambda^{op} \to \BZ/2\on{-}\dgst_k$.
In particular, if we iterate it $\#v$ times, we obtain the shift $[2]$. Thus passing 
to underlying $\BZ/2$-dg categories, we obtain the sought-after invariance.
\end{proof}

Regard  the poset $\cP$ as a category.

\begin{defn}
1) Define the functor of {\em large microlocal sheaves}  
 on $\Sigma$ 
supported along 
$\Gamma$
 to be given by the assignments
$$
\xymatrix{
\mu\Sh^\un_\Gamma :\cP \ar[r] & \BZ/2\on{-}\dgSt_k 
&
\mu\Sh^\un_\Gamma(p) = \mu\Sh_{\Lambda_p}(\BR^2, \Omega)_{\BZ/2}
}$$
with morphisms taken to the natural microlocal restrictions.

Define the functor of {\em traditional microlocal sheaves}  
 to be given by the assignments
$$
\xymatrix{
\mu\Sh_\Gamma :\cP \ar[r] & \BZ/2\on{-}\dgst_k 
&
\mu\Sh_\Gamma(p) = \mu\Sh_{\Lambda_p}(\BR^2, \Omega)_{\BZ/2}
}$$
with morphisms given by the natural microlocal restrictions.

Define the  functor of  {\em wrapped microlocal sheaves} 
to be given by the assignments
$$
\xymatrix{
\mu\Sh^w_\Gamma :\cP \ar[r] & \BZ/2\on{-}\dgst_k 
&
\mu\Sh^w_\Gamma(p) = \mu\Sh_{\Lambda_p}(\BR^2, \Omega)_{\BZ/2}
}$$
with morphisms given by the left adjoints to the natural microlocal restrictions.

2) 
Define the  2-periodic dg categories of {\em traditional and wrapped microlocal sheaves}  to be the respective limit and colimit
$$
\xymatrix{
\mu\Sh_\Gamma(\Sigma) = \lim_{\cP^{op}} \mu\Sh_\Gamma
&
\mu\Sh^w_\Gamma(\Sigma) = \colim_{\cP} \mu\Sh_\Gamma^w
}
$$
\end{defn}

Comparing the above explicit descriptions, we obtain the following.

\begin{thm}\label{thm:surface}
There are canonical equivalences of functors
$$
\xymatrix{
\cF_\Gamma \ar[r]^-\sim & \mu\Sh_\Gamma 
&
\cF^w_\Gamma \ar[r]^-\sim & \mu\Sh^w_\Gamma
}
$$
and thus canonical equivalences of their respective limits and colimits
$$
\xymatrix{
\cF_\Gamma(\Sigma) \ar[r]^-\sim & \mu\Sh_\Gamma(\Sigma)
&
\cF^w_\Gamma(\Sigma) \ar[r]^-\sim & \mu\Sh^w_\Gamma(\Sigma)
}
$$
\end{thm}

\begin{corollary}\label{cor:topol}
The functors of traditional and wrapped microlocal sheaves
$
\mu\Sh_\Gamma, 
\mu\Sh^w_\Gamma
$
and  their respective $\BZ/2$-dg categories of global sections
$
\mu\Sh_\Gamma(\Sigma) 
$,
$\mu\Sh^w_\Gamma(\Sigma)
$
only depend on the exact symplectic structure on $\Sigma$ through the orientation
it defines. 
\end{corollary}

\begin{remark}
Suppose we restrict to compact support Lagrangians, or more generally,
support Lagrangians  with fixed structure near  the circular ends of $\Sigma$.
Then Dyckerhoff-Kapranov~\cite{DK} have explained that the 
the  $\BZ/2$-dg categories 
  $\cF_\Gamma(\Sigma)$, $\cF^w_\Gamma(\Sigma)$ are canonically independent
  of the specific choice of  skeleton $\Gamma\subset \Sigma$. They provide invariants
  of the oriented surface $\Sigma$ equipped with a finite subset of the circular ends of $\Sigma$.
\end{remark}

\begin{corollary}\label{cor:surfduality}
The natural hom-pairing provides an equivalence
$$
\xymatrix{
\mu\Sh_\Gamma(\Sigma)  \ar[r]^-\sim & \Fun^{ex}(\mu\Sh_\Gamma^w(\Sigma)^{op}, \Perf_k) 
}
$$
\end{corollary}

\begin{proof}
As in the proof of Theorem~\ref{thm:duality} below, it suffices to prove the assertion locally.
But perfect and coherent modules over the $A_n$-quiver coincide, and form a smooth and proper dg category,
so the  assertion holds locally.
\end{proof}


\subsection{Example: punctured spheres}
We focus here on the specific example $\Sigma_n = S^2 \setminus S_n$, where $S_n$ consists of $n\geq 2$ distinct points,
and $\Gamma_n \subset \Sigma_n$ is a natural compact Lagrangian skeleton. 

\begin{remark}
By Corollary~\ref{cor:topol},
microlocal sheaves on a surface  $\Sigma$ supported along $\Gamma$  depend on the symplectic structure of $\Sigma$
only 
in so far as it provides an orientation and hence the requisite cyclic orderings.
In what follows, we will keep track of orientations but not the specifics of the symplectic structure.
\end{remark}

Let $T^1 = \BR/ 2\pi\BZ$ be the circle, with coordinate $\theta$,
and $T^*T^1 \simeq T^1 \times \BR$ its cotangent bundle, with canonical coordinates $\theta, \xi$.

%
%
Fix $n\geq 2$. 
Introduce  the finite set
$$
\xymatrix{
\fs_{n-2} = \{ (\pi, 2m-1) \in T^1 \times \BR \, |\, m=1, \ldots, n-2\} 
}
$$

Let us work with the surface $\Sigma_n = S^2 \setminus S_n$ in the form of the open complement
$$
\xymatrix{
\Sigma_n = T^* T^1 \setminus \fs_{n-2} 
}
$$
The canonical exact symplectic structure on $T^*T^1$ restricts to an exact symplectic structure 
 on~$\Sigma_n $. This exact structure is in no way specially tuned to $\Sigma_n$, 
 and in particular, does not present it as a Weinstein manifold,
 but all we require is the induced orientation.
 
Let us work with the skeleton 
$$
\xymatrix{
\Gamma_n = \{ (\theta, 2m-2) \in T^1 \times \BR \, |\, m=1, \ldots, n-1\}
\cup  \{ (0, \xi) \in T^1 \times \BR \, |\, 0\leq \xi \leq  2n-4\}   \subset \Sigma_n
}
$$

For $n = 2$, set $Q_2 = \G_m$.
For $n \geq 3$, set
$$
\xymatrix{
Q_n = \BA^1 \cup_{pt} \BP^1 \cup_{pt} \BP^1  \cdots   \cup_{pt} \BP^1 \cup_{pt} \BA^1
}
$$
with $n-3$ copies of $\BP^1$, and where the inclusions of $pt= \Spec k$
 into each copy of $\BP^1$ from the  left and right  have distinct images.

\begin{thm}\label{thm:surfacemirror}
There are mirror equivalences
$$
\xymatrix{
 \mu\Sh_{\Gamma_n}(\Sigma_n) \ar[r] &\Perf_\proper(Q_n)_{\BZ/2}
&
\mu\Sh^w_{\Gamma_n}(\Sigma_n) \ar[r]^-\sim &  \Coh(Q_n)_{\BZ/2}
}
$$
\end{thm}

\begin{proof}
By Theorem~\ref{thm:surface},
there are canonical equivalences
$$
\xymatrix{
\cF_{\Gamma_n}(\Sigma_n) \ar[r]^-\sim & \mu\Sh_{\Gamma_n}(\Sigma_n)
&
\cF^w_{\Gamma_n}(\Sigma_n) \ar[r]^-\sim & \mu\Sh^w_{\Gamma_n}(\Sigma_n)
}
$$
Thus it suffices to establish the assertion for combinatorial Fukaya categories. 

On the one hand, viewing $Q_n$ as an iterated pushout, by Proposition~\ref{prop:moredescent}, we can calculate $\Perf_\proper(Q_n)_{\BZ/2}$ as  an
 iterated pullback and $\Coh(Q_n)_{\BZ/2}$ as an 
 iterated pushout.

On the other hand, 
for any $a<b$, define the open subspace
$$
\xymatrix{
\Sigma(a, b)  =  \Sigma_n \cap ( T^1 \times (a, b)) 
  \subset \Sigma_n
}
$$
Consider the  cover of $\Sigma_n$ by the open subspaces
$$
\xymatrix{
\Sigma(-\oo, 2), \Sigma(0, 4), \ldots, \Sigma(2n-8, 2n-4), \Sigma(2n-6, \oo)
}
$$
Each only intersects its neighbors in the list and these intersections take the form
$$
\xymatrix{
\Sigma(0, 2), \Sigma(2, 4), \ldots, \Sigma(2n-8, 2n-6), \Sigma(2n-6, 2n-4)
}
$$

Consider $\Sigma_n$ as an iterated pushout diagram of the above cover.
Applying the functors of combinatorial Fukaya categories $\cF_\Gamma$, $\cF^w_\Gamma$,
 it is a straightforward exercise to check that we obtain equivalent diagrams to
those calculating $\Perf_\proper(Q_n)_{\BZ/2}$, $\Coh(Q_n)_{\BZ/2}$ respectively.
To  outline the steps, for the end terms $\Sigma(-\oo, 2), \Sigma(2n-6, \oo)$, we  obtain $\Perf_{\proper}(\BA^1)_{\BZ/2}, \Coh(\BA^1)_{\BZ/2}$ consonant with Lemma~\ref{lemma:cylinder} below.
For the middle terms $\Sigma(0, 4), \ldots, \Sigma(2n-8, 2n-4)$, we obtain $\Perf_{\proper}(\BP^1)_{\BZ/2}, \Coh(\BP^1)_{\BZ/2}$ consonant with Example~\ref{ex:cylinder} above.
For the intersections of neighbors, note that the skeleton reduces to a line so we obtain $(\Perf_k)_{\BZ/2}$.
Finally, see Example~\ref{ex:cylinder} to confirm the maps of the diagram are  the evident pushforwards.
(Note the diagram contains no cycles and so there are no compatibilities of the maps to check.)
\end{proof}


\section{Pairs of pants}\label{s: pants}


\subsection{Liouville manifolds}

We will view pairs of pants as Liouville manifolds and so recall here some standard background in this direction.
Our source for all of the material is~\cite{ce}.

Let $X$ be a smooth manifold. 
Let $v\in\Vect(X)$ be a complete vector field,
and denote its flow by $v_{t}:X\to X$, for $t\in \BR$.
Let $X^0 \subset X$ be the zero locus of $v$, or equivalently the fixed locus of $v_t$.
By the  stable locus of $v$, we will mean the subspace
$$
\xymatrix{
 X^{st}  =  \{ x\in X \, |\, \lim_{t\to \infty} v_{t}(x) \in X^0\} \subset X
}
$$

Now let $W$ be an exact symplectic manifold with Liouville form $\alpha$
and symplectic form $\omega = d\alpha$.
The Liouville vector field $v\in \Vect(W)$ defined by $i_v\omega = \alpha$ is symplectically
 expanding $L_v \omega = \omega$. Conversely, if a vector field $v\in \Vect(W)$ is symplectically expanding $L_v \omega = \omega$,
 then the one-form $\alpha = i_v \omega$ is a Liouville form $d\alpha = \omega$.

\begin{defn}
By a {\em Liouville manifold} $(W, \alpha, \omega)$, we will mean an exact symplectic manifold $W$
with Liouville form $\alpha$
and symplectic form $\omega = d\alpha$ such that the Liouville vector field $v\in \Vect(W)$ is complete,
and
there exists an exhaustion $W = \bigcup_{i = 1}^\oo W^i$ by compact domains $W^i \subset W$
with smooth boundaries along which $v$ is outward pointing.
\end{defn}

Let $W^0 \subset W$ be the zero locus of the Liouville vector field $v$,
or equivalently the fixed locus of the Liouville flow $v_t:W\to W$, for $t\in \BR$.
We will use the term {\em skeleton}
to refer
to the stable locus for the Liouville flow
$$
\xymatrix{
 W^{st} =  \{ x\in W \, |\, \lim_{t\to \infty} v_{t}(x) \in W^0\} \subset W
}
$$

A Liouville manifold is said to be {\em finite-type} if the skeleton is compact. In this case,
we may fix a single compact domain $W^c\subset W$ with smooth boundary 
along which the Liouville vectro field $v$ is outward pointing and such that $W^{st} \subset W^c$. Then the boundary
$\partial W^c \subset W$ is naturally a contact manifold with contact form $\lambda =  \alpha|_{\partial W^c}$.
Applying the
 Liouville flow $v_t$ to the boundary $\partial W^c $  provides an
exact symplectomorphism
$$
\xymatrix{
W \setminus W^{st} \ar[r]^-\sim & \partial W^c \times \BR
}
$$ 
where $ \partial W^c \times \BR$ is the symplectization of $\partial W^c$  
 with Liouville form $e^t\lambda$.

By a {\em Liouville submanifold}, we will mean a closed submanifold $V \subset W$ such that
the restrictions
$(V, \alpha|_V, \omega|_V)$ form a Liouville manifold.

%
%

\begin{defn}
A smooth family $(W, \alpha(s), \omega(s))$,  $s\in [0,1]$, of Liouville manifolds is called a  
{\em simple Liouville homotopy} if there exists a smooth family of exhaustions 
 $W = \bigcup_{i = 1}^\oo W^i(s)$ by compact domains $W^i(s) \subset W$ 
 with smooth boundaries along which the Liouville vector field $v(s)$ is outward pointing. 
A smooth family $(W, \alpha(s), \omega(s))$, $s\in [0,1]$, of Liouville manifolds is called a  
{\em Liouville homotopy} if it is a composition of finitely many simple
 Liouville homotopies.
\end{defn}

\begin{prop}\cite[Proposition 11.8]{ce}
Let $(W, \alpha(s), \omega(s))$,  $s\in [0,1]$, be a  Liouville homotopy.
Then there exists a diffeotopy $h(s):W\to W$ with $h(0) = \Id_W$ such that
$h(s)^*\alpha(s) -\alpha(0)$ is exact, for all $s\in [0,1]$.
Moreover, if the closure of the union $\ol{\bigcup_{s\in [0,1]} W(s)^{st}}$ is compact, then we can achieve that
$h(s)^*\alpha(s_ -\alpha_0 = 0$, for all $s\in [0,1]$, outside of a compact subspace.
\end{prop}

\begin{remark}
A Liouville homotopy of finite-type Liouville manifolds need not have
the property that the closure of the union $\ol{\bigcup_{s\in [0,1]} W(s)^{st}}$ is compact (see \cite[Example 11.7]{ce}).
\end{remark}

\begin{defn}
By a {\em Weinstein manifold} $(W, \alpha, \omega, \phi)$,
we will mean an exact symplectic manifold $W$
with Liouville form $\alpha$, symplectic form $\omega = d\alpha$, and exhausting
Morse-Bott function $\phi:W\to \BR$ such that 
 the Liouville vector field $v\in \Vect(W)$ is complete and gradient-like for $\phi$.
\end{defn}

A Weinstein manifold $(W, \alpha, \omega, \phi)$ defines a Liouville manifold  $(W, \alpha, \omega)$
(though not every Liouville manifold is diffeomorphic to a Weinstein manifold~\cite{mcduff}).
For a Weinstein manifold $(W, \alpha, \omega, \phi)$, the critical locus of $\phi$ coincides with the zero locus of $v$, and
is a union $W^0 =\coprod_{i}  W^0_i \subset W$ of  connected critical manifolds.
The skeleton is a union $W^{st} =\coprod_{i}  W^{st}_i \subset W$ of the respective stable manifolds
  $$
 \xymatrix{
 W^{st}_i =  \{ x\in W \, |\, \lim_{t\to \infty} v_{t}(x) \in W^0_i\} \subset W
 }
 $$
  By construction, each stable manifold $ W^{st}_i\subset W$ is invariant with respect to the Liouville vector field $v$,
  and hence  isotropic for the Liouville form $\alpha|_{W^{st}_i} = 0$.

Let $M$ be a complex manifold, and  $J:TM\to TM$ its complex structure.
To a smooth function $\phi :M\to \BR$,  associate the respective one-form
and two-form
 $$
 \xymatrix{
 d^c\phi = d\varphi \circ J
 &
 \omega_\phi = - d d^c \phi = 2i \partial \ol\partial\phi
 }
 $$  
One says that $\phi$ is  {\em $J$-convex} if the assignment 
$$
\xymatrix{
g_\phi(v, w) = \omega_{\phi}(v, Jw)
&
v, w\in TM
}
$$
 is positive definitive $g_\varphi(v, v) >0$, for $v\not = 0 \in TM$,
 and so defines a Riemannian metric. 

\begin{defn}
By a {\em Stein manifold} $(M, J, \phi)$, we will mean a complex manifold $M$, with complex structure $J$, equipped with an exhausting $J$-convex
function $\phi:M\to \BR$. 
\end{defn}

\begin{remark}
Given a  complex manifold $(M, J)$, the space of exhausting $J$-convex
functions is convex, hence contractible, and open in the $C^2$-topology.
\end{remark}

\begin{remark}
Any  closed complex submanifold  of $\BC^N$ is a Stein manifold when equipped with the 
exhausting $J$-convex function $\phi(z) = |z|^2$. Conversely, thanks to classical results of 
Grauert and Bishop-Narasimhan,
any  Stein manifold can be embedded as a closed complex submanifold  of some $\BC^N$.
\end{remark}

A Stein manifold $(M, J, \phi)$ defines an exact symplectic form $\omega_\phi$ with Liouville form $d^c \phi$.
After composing $\phi$ with a suitable convex function so that the Liouville vector field is complete,
we obtain a Liouville manifold $(M, d^c \phi, \omega_\phi)$.
After perturbing $\phi$ so that it is  Morse-Bott, we obtain  
a Weinstein manifold  $(M, d^c \phi, \omega_\phi, \phi)$.
Conversely, up to suitable homotopy, every Stein manifold comes from a Weinstein manifold~\cite{ce}.

\begin{example}\label{ex: torus}

For $n\in \mathbb N$, set  $[n]= \{1, \ldots, n\}$.


Introduce the torus $T^{n} = (\BR/2\pi \BZ)^{n}$ with coordinates $\theta_a$, for $a\in [n]$.
Fix the usual identification $T^* T^{n}  \simeq T^{n}  \times \BR^{n} $ with canonical coordinates $\theta_a, \xi_a$,
for $a\in [n]$,
so that the Liouville form, symplectic form, and Liouville vector field take the respective forms 
$$
\xymatrix{
\displaystyle
\alpha_{{n}} = \sum_{a = 1}^{n} \xi_a d\theta_a
&
\displaystyle
\omega_{{n}} = d\alpha_{{n}} =  \sum_{a = 1}^{n} d\xi_a d\theta_a
&
\displaystyle
v_{{n}} =  \sum_{a = 1}^{n} \xi_a \partial_{\xi_a}
}
$$
The regular  $T^n$-action on $T^n$ induces a Hamiltonian $T^n$-action on  $T^* T^n$ with moment map
the natural projection
$$
\xymatrix{
\mu_n: T^*T^{n} \ar[r] & \BR^{n}
&
\mu_n(\theta_1, \xi_1,\ldots, \theta_n,\xi_n) = (\xi_1, \ldots, \xi_n)
}
$$
Taking its squared-length provides a Weinstein manifold $(T^*T^n, \alpha_{{n}}, \omega_{{n}}, |\mu_n|^2)$ with skeleton the zero-section
$T^n\subset T^* T^n$.

Introduce the complex torus $T^{n}_\BC = (\BC^\times)^{n}$ with coordinates $z_a = x_a + i y_a = r_a e^{i\theta_a}$,
for $a\in [n]$.
Fix the identification $T^{n}_\BC  \simeq T^* T^{n}$ defined by
$z_a = e^{\xi_a + i\theta_a}$ so that $\xi_a = \log |z_a|$, $r_a = e^{\xi_a}$, for $a\in [n]$.
Note that the moment map $\mu_n:T^*T^n \to \BR^n$ transports to 
the projection 
$$
\xymatrix{
\Log_n: T^{n}_\BC \ar[r] & \BR^{n}
&
\Log_n(z_1, \ldots, z_n) = (\log|z_1|, \ldots, \log|z_n|)
}
$$
Taking its squared-length provides a Stein manifold $(T^n_\BC, J, |\Log_n|^2)$
 with underlying Weinstein manifold $(T^*T^n, \alpha_{{n}}, \omega_{{n}}, |\mu_n|^2)$.

\end{example}



\subsection{Tailored pairs of pants}\label{s: tailored}

We continue with the constructions and notation recorded in Example~\ref{ex: torus} above.
Thus we have the natural Stein structure 
$$
\xymatrix{
(T_\BC^{n+1}, J, |\Log_{n+1}|^2)
}
$$
on a complex torus, and
 the natural Weinstein structure  
$$
\xymatrix{
(T^* T^{n+1}, \alpha_{{n+1}}, \omega_{{n+1}}, |\mu_{n+1}|^2)
}
$$
on the cotangent bundle of a compact torus.
Under the  identification
$T^{n+1}_\BC \simeq T^* T^{n+1}$,
given in coordinates by
$z_a = e^{\xi_a + i\theta_a}$, for $a\in [n+1]$,
 the Stein structure  induces the Weinstein structure.

By the $n$-dimensional {\em pair of pants}, we will mean the Liouville manifold 
$(P_n, \alpha_{P_n}, \omega_{P_n})$ given by the generic hyperplane
$$
\xymatrix{
P_{n} = \{1 +  z_1 + \cdots + z_{n+1}   = 0 \} \subset  T^{n+1}_\BC
}
$$
equipped with  the restricted Liouville form $\alpha_{P_n} = \alpha_n|_{P_n}$ and symplectic form
$\omega_{P_n} =  \omega_n|_{P_n}$.
%
Note this is  the Liouville manifold associated to the  Stein manifold
$(P_n, J, \phi_n)$
given by the restricted exhausting $J$-convex function  $\phi_n = |\Log_n|^2|_{P_n}$.

\begin{remark}
Fix  the standard open embedding 
$$
\xymatrix{
T_\BC^{n+1} \subset \BP^{n+1}_\BC = \on{Proj}( k[z_0, z_1, \ldots, z_{n+1}])
}
$$  
as the complement of the coordinate hyperplanes $H_a = \{z_a = 0\} \subset \BP^{n+1}_\BC$,
 for $a\in \{0\} \cup [n]$.
 Thus the $n$-dimensional  pair of pants lies in the hyperplane
 $$
\xymatrix{
\displaystyle
H  = \{ z_0 + \cdots + z_{n+1} = 0\} \subset \BP_\BC^{n+1}
}
$$  
as the complement of its intersections with the  coordinate hyperplanes
  $$
\xymatrix{
\displaystyle
P_n =  H \setminus (\bigcup_{a = 0}^{n+1} H \cap H_a)
}
$$  
 Note the symmetric group $\Sigma_{n+2}$ naturally  acts on
 the subspaces $P_n \subset T_\BC^{n+1} \subset \BP^{n+1}_\BC$
 by permuting the homogeneous coordinates.
\end{remark}

%

Following Mikhalkin~\cite{mik}, it is useful to work with the pair of pants 
in  a slightly modified form where we alter its embedding near infinity.
By the $n$-dimensional {\em tailored pair of pants}, 
we will mean the Liouville manifold $(Q_n, \alpha_{Q_n}, \omega_{Q_n})$
given by the submanifold $Q_n \subset T^{n+1}_\BC$
constructed in~\cite[Proposition 4.6]{mik}
equipped with  the restricted Liouville form $\alpha_{Q_n} = \alpha_n|_{Q_n}$ and symplectic form
$\omega_{Q_n} =  \omega_n|_{Q_n}$.
%
%
%
%
%

Let us recall some of its key properties. 
For a large constant $R >0$,
 consider the closed $(n+1)$-simplex
$$
\xymatrix{
\Delta_{n+1}(R) = \{ (x_1, \ldots, x_{n+1}) \in \BR^{n+1} \, |\, x_a \geq -R, \mbox{ for } a\in [n+1],
\sum_{a = 1}^{n+1} x_a \leq R\}
}
$$
For a large constant $M >0$, consider the open subspace
$$
\xymatrix{
\BC^\times(M) = \{z\in \BC^\times \, |\, \log|z| < -M\}
}
$$
and more generally, the open subspace
$$
\xymatrix{
T^{n+1}_\BC(M) = \{(z_1, \ldots, z_{n+1}) \in T^{n+1}_\BC \, |\, \log|z_{n+1}| < -M\}
}
$$
Note
the evident identification $T^{n+1}_\BC(M)  \simeq T^n_\BC \times \BC^\times(M)$.

\medskip

1) There is an isotopy  of Liouville submanifolds 
$$
\xymatrix{
Q_n(s) \subset T^{n+1}_\BC,
s\in [0,1],
&
Q_n(0) = P_n, Q_n(1) = Q_n
}
$$
Moreover, the isotopy is constant inside of
the compact region $\Log_{n+1}^{-1}(\Delta_{n+1}(R)) \subset T^{n+1}_\BC$,
and preserved by the  permutation action of the symmetric group $\Sigma_{n+2}$.
In particular, $Q_n$ coincides with $ P_n$ inside  of $\Log_{n+1}^{-1}(\Delta_n(R))$,
 and  $Q_n$ is preserved by the action of $\Sigma_{n+2}$.

\medskip

2) There is the inductive compatibility
$$
\xymatrix{
Q_n \cap T^{n+1}_\BC(M)  = Q_{n-1} \times \BC^\times(M) 
}
$$
In particular,
$Q_n \cap T^{n+1}_\BC(M)$ is preserved by the dilation $z_{n+1} \mapsto cz_{n+1}$,
for $c\in (0, 1)$. Note that the permutation action of $\Sigma_{n+2}$ implies similar
compatibilities in other directions.

\medskip

To provide the tailored pair of pants with a particularly simple skeleton,  
it will be useful to break symmetry and apply a natural isotopy to its Liouville structure. 

For $x = (x_1, \ldots, x_{n+1})\in \BR^{n+1}$, consider the
family of Weinstein structures on $T_\BC^{n+1}\simeq T^*T^{n+1}$
given in coordinates
$z_a = e^{\xi_a + i\theta_a}$, for $a\in [n+1]$,
 by the translated Liouville form and  translated symplectic form 
$$
\xymatrix{
\alpha_{{n}}^x =  \sum_{a = 1}^{n+1} (\xi_a - x_a) d\theta_a
&
\omega_{{n}}^x = d\alpha^x_n =   \sum_{a = 1}^{n+1} (\xi_a - x_a) d\theta_a
}
$$
Note these are  the Weinstein structures associated to the  Stein structures
given by the translated exhausting $J$-convex function  
$$
\xymatrix{
|\Log_{n+1}^x (z_1, \ldots, z_{n+1})|^2=  \sum_{a = 1}^{n+1} (\log|z_a - x_a|)^2
}
$$

The restricted Liouville form $\alpha^x_{P_n} = \alpha^x_n|_{P_n}$ and symplectic form
$\omega^x_{P_n} =  \omega^x_n|_{P_n}$ provide a family of Liouville structures on the pair of pants
$P_n \subset T^{n+1}_\BC$ induced 
by the restricted exhausting $J$-convex function  $\phi_n^x = |\Log_n^x|^2|_{P_n}$.
It follows from the arguments of~\cite[Proposition 4.6]{mik} that we may construct 
the tailored pair of pants $Q_n \subset T^{n+1}_\BC$ so that  
  the restricted Liouville form $\alpha^x_{Q_n} = \alpha^x_n|_{Q_n}$ and symplectic form
$\omega^x_{Q_n} =  \omega^x_n|_{Q_n}$ provide a family of Liouville structures as well.

Now  choose a large $\ell >0$,  and fix the point $x_\ell = (-\ell, \ldots, -\ell) \in \BR^{n+1}$.
Let us focus our attention
on the Liouville structure on the tailored pair of pants 
$Q_n \subset T^{n+1}_\BC$ given by the specific translated Liouville form
$$
\xymatrix{
\displaystyle
\beta_{Q_n} = \alpha_{Q_n}^{x_\ell} =  (\sum_{a = 1}^{n+1} (\xi_a + \ell) d\theta_a)|_{Q_n}
}
$$

We will write 
$
 L_n \subset Q_n
$
for the resulting skeleton. Our aim in the rest of this section is to describe its geometry.
To do so, we will first introduce some further simple  constructions.

%
%
%
%
%
%
%

Introduce the diagonal circle $T^1_\Delta \subset T^{n+1}$. 
The translation  $T^1_\Delta$-action on $T^{n+1}$ induces a Hamiltonian $T^1_\Delta$-action on  $T^* T^{n+1}$ with moment map
$$
\xymatrix{
\mu_\Delta:T^*T^{n+1}\ar[r] & 
\BR
&
\mu_\Delta (\theta_1, \xi_1, \ldots, \theta_{n+1}, \xi_{n+1})= \sum_{a=1}^{n+1}  \xi_a
}
$$

Consider the quotient $\BT^n = T^{n+1}/T^1_\Delta$
consisting of  $(n+1)$-tuples $[\theta_1, \ldots, \theta_{n+1}]$ taken up to simultaneous translation.
If we distinguish the last entry, then we obtain an identification $\BT^n \simeq T^n$ via the coordinates
$\theta_a - \theta_{n+1}$, for $a \in [n]$.

Let $\ft_n^* = \{\sum_{a = 1}^{n+1} \xi_a = 0\} \subset \BR^{n+1}$ denote the dual of the Lie algebra of $\BT^n$.
Consider the cotangent bundle $T^* \BT^n \simeq \BT^n \times\ft_n^*$ consisting of pairs of  $(n+1)$-tuples $([\theta_1, \ldots, \theta_{n+1}], (\xi_1, \ldots, \xi_{n+1}))$,
with the first taken up to simultaneous translation, and the second satisfying 
$
\sum_{a = 1}^{n+1} \xi_a = 0$.

For $\chi\in  \BR$, we have a twisted Hamiltonian reduction correspondence
$$
\xymatrix{
T^*T^{n+1} & \ar@{_(->}[l]_-{q_\chi} \mu^{-1}_\Delta(\chi) \ar@{->>}[r]^-{p_\chi}
&
T^*\BT^n
}
$$
where the  level-set $ \mu^{-1}_\Delta(\chi)$ consists of pairs of  $(n+1)$-tuples $((\theta_1, \ldots, \theta_{n+1}), (\xi_1, \ldots, \xi_{n+1}))$,
with the second satisfying $ \sum_{a = 1}^{n+1} \xi_a = \chi$.
The map $q_\chi$ is the evident inclusion,   and the map $p_\chi$ is the translated projection
$$
\xymatrix{
p_\chi(( \theta_1, \ldots, \theta_{n+1}), ( \xi_1, \ldots, \xi_{n+1})) = ([ \theta_1, \ldots, \theta_{n+1}],
(\xi_1- \hat\chi, \ldots, \xi_{n+1}- \hat\chi))
}
$$
where we set $\hat \chi = \chi/(n+1)$.
In particular, when $\chi= 0$, we recover the usual Hamiltonian reduction correspondence
$$
\xymatrix{
T^*T^{n+1} & \ar@{_(->}[l]_-{q_0}  T^*_{T^1_\Delta} T^{n+1} \ar@{->>}[r]^-{p_0}
&
T^*\BT^n
}
$$
where $  T^*_{T^1_\Delta} T^{n+1} \subset T^* T^{n+1}$ is the conormal bundle.

 Introduce the conic Lagrangian subvariety
$$
\xymatrix{
\Lambda_{1} = \{(\theta, 0) \, | \, \theta\in T^1\} \cup \{(0, \xi) \, |\, \xi \in \BR_{\geq 0}\} \subset
T^1 \times \BR \simeq  T^*T^1
}$$
and product conic Lagrangian subvariety
$$
\xymatrix{
\Lambda_{n+1} = (\Lambda_1)^{n+1} \subset (T^* T^1)^{n+1} = T^* T^{n+1}
}$$

Note that $\Lambda_{n+1}  \subset \mu_\Delta^{-1}(\BR_{\geq 0})$,
and that  $\Lambda_{n+1}$ and $\mu_\Delta^{-1}(\chi)$ are transverse, for $\chi>0$.

\begin{defn}
For $\chi>0$, define the Lagrangian subvariety
$$
\xymatrix{
\fL_{n} = p_\chi(q_\chi^{-1}( \Lambda_{n+1}) \subset T^*\BT^n
}
$$
\end{defn}

\begin{remark}
We do not include $\chi$ in the notation for $\fL_n$ since eventually we will specialize to the case  $\chi  = n+1$,
 the character of $T_\Delta^1$ arising by restricting the diagonal character
of $T^{n+1}$. For now, we keep $\chi$ as a variable since fewer appearances of $n+1$ may lead to less confusion. 
\end{remark}

To describe $\fL_{n} \subset T^* \BT^n$, let us return to the moment map
$\mu_{n+1}:T^*T^{n+1} \to \BR^{n+1}$ and  restrict it to $\Lambda_{n+1} \subset T^* T^{n+1}$.
Note first that $\mu_{n+1}(\Lambda_{n+1}) = \BR_{\geq 0}^{n+1}$.
For  $I \subset [n+1]$, consider the relatively open coordinate cone 
$$
\xymatrix{
\sigma_I  = \{\xi_a = 0, \xi_b >0 \, |\, a\in I,\,  b\not \in I\} \subset \BR_{\geq 0}^{n+1}
}
$$ 
For $x\in \sigma_I$, note that  $\mu^{-1}_{n+1}(x) \cap \Lambda_{n+1}$
is the orthogonal coordinate subtorus 
$$
\xymatrix{
T^I  = \{\theta_a = 0 \, |\, a\not \in I\} \subset T^{n+1}
}
$$

Next consider the closed simplex
$$
\xymatrix{
\tilde \Xi_n(\chi) = \{ (\xi_1, \ldots, \xi_{n+1})\, |\, \xi_a \geq 0, \mbox{ for } a\in [n+1], \sum_{a = 1}^{n+1} \xi_a = \chi \} 
\subset \BR_{\geq 0 }^{n+1}}
$$
Note that the projection $p_\chi$ restricts to an isomorphism 
$$
\xymatrix{
\mu_\Delta^{-1}(\chi)  \cap \Lambda_{n+1} = 
\mu_{n+1}^{-1}(\tilde \Xi_n(\chi)) \ar[r]^-\sim & \fL_{n} 
}
$$
since for any point of $\mu_\Delta^{-1}(\chi)  \cap \Lambda_{n+1}$, we must have
 $\xi_a > 0$ and hence $\theta_a = 0$, for some $a\in [n+1]$, so that no points are identified by
 $T^1_\Delta$-translations.

For proper  $I \subset [n+1]$, consider the relatively open subsimplex
$$
\xymatrix{
\tilde \Xi_I(\chi)  = \tilde \Xi_n(\chi) \cap \sigma_I 
}
$$ 
Then we see that $p_\chi$ restricts to an isomorphism
$$
\xymatrix{
\bigcup_{I} T^I \times \tilde \Xi_I(\chi)\ar[r]^-\sim & \fL_n
}
$$ 
where we take the union over nonempty  $I \subset [n+1]$.

To make the above description more intrinsic,  
consider 
%
for $\chi>0$, 
and proper  $I \subset [n+1]$,  the relatively open simplex
$$
\xymatrix{
\Xi_I(\chi)  
= \{ (\xi_1, \ldots, \xi_{n+1}) \, |\, \xi_a >  - \hat \chi, \mbox{ for } a\in I,
 \xi_a =  - \hat \chi, \mbox{ for } a\not \in I, \sum_{a = 1}^{n+1} \xi_a = 0 \} 
}
$$ 
and  the subtorus $\BT^I \subset \BT^n$ given by the isomorphic
  image
of $T^I\subset T^{n+1}$. 

Then the above description descends to an  identification of subspaces
$$
\xymatrix{
 \fL_n = \bigcup_{I} \BT^I \times \Xi_I(\chi) \subset \BT^n \times \ft_n^*
 \simeq  T^* \BT^n 
}
$$ 
where we take the union over proper  $I \subset [n+1]$. 

Now we are ready to  describe the geometry of the skeleton
$
 L_n \subset Q_n.
$

\begin{thm}\label{thm: symplecto}
There is an open neighborhood $U_n \subset Q_n$ of the skeleton $L_n \subset Q_n$
and an open symplectic embedding
$$
\xymatrix{
\mathfrak j:U_n \ar@{^(->}[r] & T^* \BT^n
}
$$
restricting to an isomorphism
$$
\xymatrix{
\mathfrak j |_{L_n}:L_n \ar[r]^-\sim & \fL_{n}
}
$$
\end{thm}

\begin{proof}
Consider the real points 
$$
\xymatrix{
R_n = Q_n\cap T^{n+1}_\BR =  \{(x_1, \ldots, x_{n+1}) \in (\BR^{\times})^{n+1} \, |\, \sum_{a=1}^{n+1} x_a = -1\} 
} 
$$
and in particular its component where all coordinates are negative 
$$
\xymatrix{
R_n^- = \{(x_1, \ldots, x_{n+1}) \in \BR_{<0}^{n+1} \, |\, \sum_{a=1}^{n+1} x_a = -1\} 
\subset R_n
} 
$$

Consider the closed simplex
$$
\xymatrix{
\Delta_n(\ell) = \{ (x_1, \ldots, x_{n+1}) \in R_n^- \, |\, \log |x_a| \geq -L, \mbox{ for } a\in [n+1]\}
}
$$
and for proper $I\subset [n+1]$, the relatively open subsimplices
$$
\xymatrix{
\Delta_I(\ell) = \{ (x_1, \ldots, x_{n+1}) \in R_n^- \, |\, \log |x_a| > -\ell, \mbox{ for } a\not \in I,
 \log |x_b| = -\ell, \mbox{ for } b \in I   \}
}
$$
So for example, for $I = \emptyset$, we have  $\Delta_\emptyset(\ell)$ is the interior of $\Delta_n(\ell)$,
and for $I = [n+1] \setminus \{i\}$, we have  $\Delta_{[n+1] \setminus \{i\}}(\ell)$ is a vertex of $\Delta_n(\ell)$.

Introduce as well the barycenters of the subsimplices
$$
\xymatrix{
\delta_I(\ell) = \{ \log |x_a| = \log|x_{a'}|,  \mbox{ for } a, a'\not \in I,
 \log |x_b| = -\ell, \mbox{ for } b \in I  \} \subset \Delta_I(\ell)
}
$$

For proper $I\subset [n+1]$, recall we  write $T^I \subset T^{n+1}$ for the subtorus with $\theta_a = 0$, for $a\not\in I$.
So for example, for $I = \emptyset$, we have $T^\emptyset \subset T^{n+1}$ is the identity.

By the inductive compatibility of tailored pairs-of pants, the  vanishing locus 
of  the Liouville form
$\beta_{Q_n}$
includes for proper, nonempty $I \subset [n+1]$,
the nondegenerate manifold given by the  torus orbit of the corresponding barycenter
$$
\xymatrix{
C_I(\ell) = T^I \cdot \delta_I(\ell) 
}
$$
with respective stable manifold
 the  torus orbit of the relatively open subsimplex
$$
\xymatrix{
S_I(\ell) = T^I \cdot \Delta_I(\ell)
}
$$

Recall
by~\cite[Corollaries 4.4 and 4.5]{mik} and the constructions of ~\cite[Proposition 4.6]{mik},
we may arrange so that $R_n$ is precisely  the critical points of $\Log_{n+1}|_{Q_n}$ 
and furthermore $\Log_{n+1}|_{R_n}$ is an immersion.
It follows that  the only additional vanishing of  the Liouville form
$\beta_{Q_n}$
is the nondegenerate isolated zero at the central barycenter 
$$
\xymatrix{
C_\emptyset(\ell) = \{\delta_\emptyset(\ell)\} =   \{T^\emptyset \cdot \delta_\emptyset(\ell)\}  
}
$$
Moreover, note that the one-form $d\theta_a$, 
 for $a\in [n+1]$,  vanishes on tangent vectors  to
 $R_n$, and thus the Liouville form 
 $
\beta_{Q_n}
$
does as well.
Therefore  $R_n$ is a Lagrangian subvariety
of $Q_n$ and invariant with respect to the Liouville flow.
In particular, the  stable manifold of  
the nondegenerate isolated zero
 must be the open subsimplex
$$
\xymatrix{
S_\emptyset(\ell) =  \Delta_\emptyset(\ell) = T^\emptyset \cdot \Delta_\emptyset(\ell)
}
$$

Altogether, we conclude the skeleton is the union of stable manifolds
$$
\xymatrix{
 L_n = \bigcup_{I} S_I(\ell) = \bigcup_{I} T^I \cdot \Delta_I(\ell) 
}
$$ 
where  the union is over proper  $I \subset [n+1]$.

Now let us find the neighborhood $U_n \subset Q_n$
and symplectic embedding $\mathfrak j:U_n \hookrightarrow T^* \BT^n$
 in the statement of the theorem. 

On the one hand, consider a small open neighborhood $U^\circ_n\subset Q_n$
of the closed simplex $\Delta_n(\ell)\subset R_n^-$. 
By the inductive compatibility of tailored pairs-of pants, 
near each relatively open boundary subsimplex $\Delta_I(\ell) \subset\partial \Delta_n(\ell)$, for nonempty,
proper $I \subset [n+1]$, there is the
local  coisotropic foliation
of $U_n^\circ$ with leaves defined by the collection of functions
$\log |z_a|$, for $ a\in I$.   
 Moreover, the  local  coisotropic foliations are compatible in the evident sense that for nonempty,
 proper $I\subset J 
 \subset [n+1]$, near the 
 relatively open subsimplex $\Delta_J(\ell) \subset\partial \Delta_n(\ell)$, the leaves defined 
 by the collection of functions
$\log |z_a|$, for $ a\in J$, refine those 
defined by 
$\log |z_a|$, for $ a\in I$.
Thus we specify no extra structure away from the boundary $\partial \Delta_n(\ell) \subset R_n^-$,
and for example, for $I = [n+1] \setminus \{i\}$, we have a local Lagrangian foliation on $U_n^\circ$ near the vertex
 $\Delta_{[n+1] \setminus \{i\}}(\ell) \subset \Delta_n(\ell)$.

On the other hand, consider a small open neighborhood $\fU^\circ_n\subset T^*\BT^n$
of the closed simplex
$
\Xi_n(\chi) \subset \ft^*.
$ 
Near each relatively open boundary subsimplex $\Xi_I(\chi) \subset\partial \Xi_n(\chi)$, for nonempty,
proper $I \subset [n+1]$, there is the
local  coisotropic foliation
of $\fU_n^\circ$ with leaves defined by the collection of functions
$\xi_a$, for $ a\in I$.   
 Moreover, the  local  coisotropic foliations are compatible in the evident sense that for nonempty,
 proper $I\subset J 
 \subset [n+1]$, near the 
 relatively open subsimplex $\Xi_J(\chi) \subset\partial \Xi_n(\chi)$, the leaves defined 
 by the collection of functions
$\xi_a$, for $ a\in J$, refine those 
defined by 
$\xi_a$, for $ a\in I$.
Thus we specify no extra structure away from the boundary $\partial \Xi_n(\chi) \subset  \ft^*$,
and for example, for $I = [n+1] \setminus \{i\}$, we have a local Lagrangian foliation on $\fU_n^\circ$ near the vertex
 $\Xi_{[n+1] \setminus \{i\}}(\chi) \subset \Xi_n(\chi)$.

Now choose a symplectomorphism 
$$
\xymatrix{
\mathfrak j^\circ:U_n^\circ \ar[r]^-\sim & \fU_n^\circ
}
$$
restricting to a diffeomorphism of Lagrangian submanifolds
$$
\xymatrix{
R_n^- \cap U_n^\circ \ar[r]^-\sim & \ft^* \cap \fU_n^\circ
}
$$ 
 further restricting  to an isomorphism of closed simplices
$$
\xymatrix{
\Delta_n(\ell) \ar[r]^-\sim & \Xi_n(\chi)
}
$$ 
and compatible with the above local coisotropic foliations.

Next, for proper $I \subset [n+1]$, choose  a small open neighborhood
$$
\xymatrix{
U_I^\circ \subset U_n^\circ 
}
$$
around the subsimplex $\Delta_I(\ell) \subset \Delta_n(\ell)$,
and consider the corresponding  small open neighborhood
$$
\xymatrix{
\fU_I^\circ = \mathfrak j^\circ(U_I^\circ) \subset \fU_n^\circ 
}
$$
around the subsimplex $\Xi_I(\ell) \subset \Xi_n(\chi)$.

Introduce the open neighborhoods
$$
\xymatrix{
U_n = \bigcup_I T^I \cdot U_I^\circ \subset Q_n &
\fU_n = \bigcup_I \BT^I \cdot \fU_I^\circ\subset T^*\BT^n
}
$$
of the respective Lagrangian subvarieties
$$
\xymatrix{
L_n = \bigcup_I T^I \cdot \Delta_I(\ell)  \subset Q_n &
\fL_n = \bigcup_I \BT^I \cdot \Xi_I(\chi) \subset T^*\BT^n
}
$$

Since the matched local coisotropic foliations are the moment maps for the respective subtorus actions,
the symplectomorphism $\mathfrak j^\circ$ canonically extends to a symplectomorphism
$$
\xymatrix{
\mathfrak j:U_n \ar[r]^-\sim & \fU_n
}
$$
satisfying the requirements of the theorem.
\end{proof}


\subsection{Contactification and symplectification}
By Theorem~\ref{thm: symplecto}, the symplectic geometry of a neighborhood 
$U_n\subset Q_n$  of the skeleton
$L_n \subset Q_n$ is equivalent to that of a neighborhood $\fU_n\subset T^*\BT^n$ 
of the Lagrangian subvariety $\fL_n \subset T^*\BT^n$.

To keep track of the exact symplectic geometry,
let us  introduce the  Liouville form  $\beta_n$ on the neighborhood $\fU_n\subset T^*\BT^n$  
obtained by transporting the Liouville form $\beta_{Q_n}$ restricted to
the neighborhood 
$U_n\subset Q_n$. Thus $\beta_n$  provides a primitive
to the restriction of the canonical symplectic form $\omega_{T^*\BT^n}|_{\fU_n} =  d\beta_n$,
and also  the Lagrangian subvariety $\fL_n \subset T^*\BT^n$ is conic with respect to its associated
Liouville vector field.

Our aim in this section is to give a microlocal interpretation of the exact symplectic geometry 
of the neighborhood $\fU_n\subset T^*\BT^n$ with Liouville form $\beta_n$
and symplectic form $\omega_{T^*\BT^n}|_{\fU_n}  = d\beta_n$.

To begin, let us recall some general notions for $M$ an exact symplectic manifold with  
Liouville form  $\alpha_M$ and 
symplectic form $\omega_M = d\alpha_M$.

\begin{defn}

1) The {\em circular contactification} of $M$
is the contact manifold $N = M\times T^1$, with contact form $\lambda_N = dt + \alpha_M$,
and contact structure $\xi_N = \ker (\lambda_N)$. 

2)   The {\em contactification} of $M$
is the contact manifold $\tilde N = M\times \BR$, with contact form $\lambda_{\tilde N} = dt + \alpha_M$,
and contact structure $\xi_{\tilde N} = \ker (\lambda_{\tilde N})$. 

Here and in what follows, we often write $t$ for a coordinate on $T^1$ or $\BR$.

\end{defn}

\begin{remark}
Note the natural contact $\BZ$-cover $\tilde N \to N$ induced by the $\BZ$-cover $\BR\to T^1$.
\end{remark}

\begin{remark}\label{rem: compare contact}
Suppose given two Liouville forms $d\alpha_M = \omega_M$, $d\alpha_M'= \omega_M$.

1) 
 If the difference $\alpha_M - \alpha'_M$ is integral, then any primitive $ f:M\to T^1$, $df = \alpha_M - \alpha'_M$ provides a diffeomorphism 
 $$
 \xymatrix{
 F: N  \ar[r]^-\sim &   N
 &
 F(m, t) = (m, t + f(m))
 }
 $$ 
 intertwining the respective contact forms $F^*(dt+\alpha'_M) = dt + \alpha_M$. 

2) Similarly, 
 if the difference $\alpha_M - \alpha'_M$ is exact, then any primitive $\tilde f:M\to \BR$, $d\tilde f = \alpha_M - \alpha'_M$ provides a diffeomorphism 
 $$
 \xymatrix{
 F:\tilde N  \ar[r]^-\sim &  \tilde N
 &
 F(m, t) = (m, t + \tilde f(m))
 }
 $$ 
 intertwining the respective contact forms $F^*(dt+\alpha'_M) = dt + \alpha_M$. 
\end{remark}

\begin{defn}

1) A Lagrangian subvariety $L\subset M$ is {\em integral} if there is a continuous function
$f:L\to S^1$, called the {\em integral structure}, 
such that
the restriction of $f$ to any submanifold contained within $L$ is differentiable and a primitive for the restriction of 
the Liouville form $\alpha_{M}$.

2) A Lagrangian subvariety $L\subset M$ is {\em exact} if in addition
there exists a lift of $f:L\to S^1$ to a continuous function $\tilde f:L\to \BR$, called the  {\em exact structure}.

\end{defn}


%
%
%
%

\begin{remark}\label{rem: leg lift}

1) A Lagrangian subvariety $L \subset M$ is  integral if and only if it admits a Legendrian lift $\cL \subset N$.
We can construct the lift $\cL \subset N$ as the graph
$$
\xymatrix{
\Gamma_{L, -f} =\{(x, -f(x)) \, | \, x\in L\}  \subset M \times T^1 = N
}
$$
of the negative of an integral structure $f:L\to T^1$,
or conversely, construct the integral structure $f:L\to T^1$
by regarding an isotropic lift $\cL \subset N$ as the negative of its graph.

2) Similarly, a Lagrangian subvariety $L \subset M$ is  exact if and only if it admits  a
Legendrian lift $\tilde\cL \subset \tilde N$.
%
\end{remark}

Now let us return to 
 the neighborhood $\fU_n\subset T^*\BT^n$ with symplectic form $\omega_{T^*\BT^n}|_{\fU_n}$,
 and focus on the two Liouville forms
   $\beta_n$
and $\alpha_{T^*\BT^n}|_{\fU_n}$.
On the one hand, the Lagrangian subvariety $\fL_n \subset \fU_n$ is conic with respect 
 to the 
Liouville vector field associated to $\beta_n$, and thus exact with respect to $\beta_n$ as exhibited by the trivial exact structure.
On the other hand, if we  construct $\fL_n \subset \fU_n$
using  $\chi>0$ with $\hat \chi  =\chi/(n+1)$ integral, then the function 
 $$
 \xymatrix{
 \displaystyle
\tilde  f:\mu_\Delta^{-1}(\chi) \cap \Lambda_{n+1} \ar[r] & T^1 &
 \displaystyle
\tilde  f  =\sum_{a=1}^{n+1} (\xi_a - \hat \chi)\theta_a 
 }
 $$ 
is invariant under $T^1_\Delta$-translations, and hence descends to a function $f:\fL_n\to T^1$. Moreover, 
a straightforward computation shows $f$ provides an integral structure  with respect to $\alpha_{T^*\BT^n}$.

Let us continue from here with  
  $\chi>0$ such that $\hat \chi  =\chi/(n+1)$ is integral.

  Consider the contactifications $(N_n, \lambda_n)$, $(N_n, \lambda_n')$
  where we set
  $$
  \xymatrix{
  N_n = \fU_n \times T^1
  &
   \lambda_n = dt + \alpha_{T^*\BT^n}|_{U_n}
   &
   \lambda_n' = dt + \beta_n
   }
   $$
   Following  Remark~\ref{rem: leg lift}, consider the respective Legendrian lifts
$$
\xymatrix{
\cL_n = \Gamma_{\fL_n, -f} & \cL_n' = \fL_n \times \{0\} 
}
$$

  Consider the difference $\gamma_n = \alpha_{T^*\BT^n}|_{\fU_n} - \beta_n$ and note that $d\gamma_n = 0$.
   Observe that $\gamma_n$ is integral since the inclusion $\fL_n \subset \fU_n$ is a homotopy equivalence,  and 
  the restriction $\gamma_n|_{\fL_n}$ is equal to the restriction $\alpha_{T^*\BT^n}|_{\fL_n}$ which is integral.
  Thus there is a unique function $g:\fU_n \to T^1$ such that $dg = \gamma_n$  with the normalization $g|_{\fL_n} = f$.
  Following Remark~\ref{rem: compare contact}, we obtain a contactomorphism
   $$
 \xymatrix{
 G: (N_n, \lambda_n)  \ar[r]^-\sim &   (N_n, \lambda_n')
 &
 G(m, t) = (m, t + g(m))
 }
 $$ 
  Note further that $G(\cL_n) = \cL_n'$ since  we have $g|_{\fL_n} = f$. We conclude that the contact geometry of $(N_n, \lambda_n)$
  near  $\cL_n$ is equivalent to that of  $(N_n, \lambda_n')$  near  $\cL_n'$. 
  Thus to give a microlocal interpretation of $(N_n, \lambda_n')$
  near  $\cL_n'$, it suffices to do the same
 for $(N_n, \lambda_n)$
  near $\cL_n$.

Now let us further specialize to 
  $\chi = n+1$ so that $\hat \chi  =\chi/(n+1) = 1$.

Introduce the conic open subspace 
and its spherical projectivization
$$
\xymatrix{
\Omega_{n+1} = \mu^{-1}_\Delta(\BR_{>0}) 
\subset T^* T^{n+1}
&
\Omega^\oo_{n+1} =\Omega_{n+1}/\BR_{>0}
\subset S^\oo T^{n+1}
}
$$
Note  the natural  projection gives an isomorphism of contact manifolds
$$
\xymatrix{
\mu^{-1}_\Delta(\chi) \ar[r]^-\sim &  \Omega_{n+1}^\oo
}
$$

 Recall the map $p_\chi:\mu^{-1}_\Delta(\chi) \to T^* \BT^n$ given by the translated projection
$$
\xymatrix{
p_\chi(( \theta_1, \ldots, \theta_{n+1}), ( \xi_1, \ldots, \xi_{n+1})) = ([ \theta_1, \ldots, \theta_{n+1}],
(\xi_1- \hat\chi, \ldots, \xi_{n+1}- \hat\chi))
}
$$

  \begin{lemma}\label{lem: contact cover}
For $\chi = n+1$, so that $\hat\chi = 1$,   we have a finite contact cover
  $$
  \xymatrix{
 \fp_\chi: \Omega^\oo_{n+1} \simeq \mu^{-1}_\Delta(\chi) \ar[r] & T^*\BT^n \times T^1    &
\fp_\chi =  p_\chi \times \delta
}
  $$
 where $\delta:T^{n+1}\to T^1$ is the diagonal character. 
 
 The cover is trivializable over the neighborhood $N_n\subset T^*\BT^n \times T^1$ of 
  the Legendrian $\cL_n \subset N_n$ with a canonical section $s: N_n \to \Omega^\oo_{n+1}$
  such that $s(\cL_n) = \Lambda^\oo_{n+1}$.
  \end{lemma}

\begin{proof}
Altogether, the proposed map takes the form
 $$
\xymatrix{
\fp_\chi(( \theta_1, \ldots, \theta_{n+1}), ( \xi_1, \ldots, \xi_{n+1})) = ([ \theta_1, \ldots, \theta_{n+1}],
(\xi_1- \hat\chi, \ldots, \xi_{n+1}- \hat\chi), \sum_{a=1}^{n+1} \theta_a)
}
$$
Note this is a finite $\BZ/(n+1)\BZ$-cover.

The contact structure on $T^*\BT^n \times T^1$ is the kernel of the one-form 
$$
\xymatrix{
\lambda = dt + \sum_{a=1}^{n+1} \xi_a d\theta_a
}
$$
Since $\chi = n+1$, so $\hat\chi = 1$, the pullback of $\lambda$ takes the form 
$$
\xymatrix{
\fp_\chi^*\lambda = \sum_{a=1}^{n+1} d\theta_a + \sum_{a=1}^{n+1} (\xi_a -\hat\chi) d\theta_a
= \sum_{a=1}^{n+1} \xi_a d\theta_a
}
$$
and thus its kernel defines the contact structure on $\Omega_{n+1}^\oo$. 

Finally, recall that $
\cL_n = \Gamma_{\fL_n, -f}$, with $f:\fL_n\to T^1$ represented by
 $$
 \xymatrix{
\tilde  f:\mu_\Delta^{-1}(\chi) \cap \Lambda_{n+1} \ar[r] & T^1 &
\tilde  f  =\sum_{a=1}^{n+1} (\xi_a - \hat \chi)\theta_a 
 }
 $$ 
 Define the section $s:\cL_n  \to \Omega^\oo_{n+1}$ by the formula
 $$
\xymatrix{
s ([ \theta_1, \ldots, \theta_{n+1}],
(\xi_1, \ldots, \xi_{n+1}),  -f)
=(( \theta_1, \ldots, \theta_{n+1}), [\xi_1 +\hat\chi, \ldots, \xi_{n+1}+\hat\chi]) 
}
$$
where we require $\theta_a = 0$ whenever $\xi_a + \hat \chi >0$, for $a\in [n+1]$. Note 
this arises for some $a\in [n+1]$, and thus constrains the homogeneous expression. Moreover,
since for all $a\in [n+1]$, either
$\theta_a = 0$ or $\xi_a = \hat \chi$ along $\fL_n$, the function $-f$, with $\hat\chi = 1$, is equal to $\sum_{a=1}^{n+1} \theta_a$
 along $\fL_n$. Thus we indeed have a section,
and since the inclusion $\fL_n \subset N_n$ is a homotopy equivalence,
the section extends from $\cL_n$ to $N_n$.
\end{proof}

By the lemma, the contact geometry of the circular contactification $N_n \subset T^* \BT^n \times T^1$ near 
the Legendrian lift $\cL_n$ is equivalent
to that of the open subspace $\Omega^\oo_{n+1} \subset S^\oo T^{n+1}$ near
the Legendrian subvariety $\Lambda^\oo_{n+1}$.
This is the microlocal interpretation we were seeking, but
for compatibility with standard refererences, which often adopt the  setting of exact symplectic  rather than contact geometry, it is useful to go one step further.

\begin{defn}
Let $N$ be a co-oriented contact manifold with contact form $\lambda$.
The {\em symplectification} of $N$ is the exact symplectic manifold $M  = N \times \BR$
with Liouville form $\alpha = e^t\lambda$.
\end{defn}

\begin{remark}
The contact geometry of $N$ is equivalent to the conic symplectic geometry of $M$. 
In particular, taking the inverse-image under the projection $M \to N$ induces a bijection from subspaces of $N$ to conic subspaces of $M$.
\end{remark}

\begin{remark}
We will only invoke the above definition in the following concrete situation.
Given an open subspace $\Omega^\oo \subset T^\oo Z$,
its symplectification
is the open  cone 
$$
\xymatrix{
\Omega = \{(z, \xi)  \, |\, \xi \not = 0,  (z, [\xi]) \in \Omega^\oo\} \subset T^* Z \setminus Z
}$$
with the restricted canonical Liouville form $\alpha_Z|_\Omega$
and  symplectic form $\omega_Z|_{\Omega} = d\alpha_Z|_\Omega$.
\end{remark}

Now we  can summarize the result of the above discussion in a final form.
Introduce the 
 circular contactification $Q_n \times T^1$, and then its symplectification 
$\tilde Q_n = Q_n \times T^1 \times \BR$,
with their natural projections
 $$
 \xymatrix{
\tilde Q_n = Q_n \times T^1 \times \BR \ar[r]^-s &   Q_n \times S^1 \ar[r]^-c &  Q_n
 }
 $$
 The skeleton $L_n \subset Q_n$ lifts under $c$ to the Legendrian subvariety $L_n \times\{0\}\subset  Q_n\times S^1$, and we can take its inverse-image under $s$ to obtain a conic Lagrangian subvariety
  $$
  \xymatrix{
  \tilde L_n = s^{-1}(L_n \times\{0\}) \subset\tilde Q_n
  }
  $$
 
 Then we have constructed the following lift of Theorem~\ref{thm: symplecto}.

\begin{thm}\label{thm: contacto}
Fix $\chi=n+1$.
There is a conic open neighborhood $\tilde U_n \subset \tilde Q_n$ of
the  Lagrangian subvariety  $\tilde L_n  \subset \tilde Q_n$,
 a conic open neigborhood $\Upsilon_{n+1} \subset \Omega_{n+1}$
 of
the  intersection  $\Lambda_{n+1} \cap \Omega_{n+1}$,
 and an exact symplectomorphism
$$
\xymatrix{
 \tilde{\mathfrak \jmath}:\tilde U_n\ar[r]^-\sim & \Upsilon_{n+1} 
 }
$$
restricting to an isomorphism
$$
\xymatrix{
 \tilde{\mathfrak \jmath} |_{\tilde L_n}:\tilde L_n \ar[r]^-\sim & \Lambda_{n+1} \cap \Omega_{n+1}
}
$$
\end{thm}


%
%
%
%
%


\subsection{Mirror symmetry}


Our aim in this section is to calculate wrapped microlocal sheaves on the tailored pair of pants  $Q_n$
supported along the skeleton $L_n$. 

Recall the conic Lagrangian subvariety
$$
\xymatrix{
\Lambda_{1} = \{(\theta, 0) \, | \, \theta\in T^1\} \cup \{(0, \xi) \, |\, \xi \in \BR_{\geq 0}\} \subset T^1 \times \BR \simeq
T^*T^1
}$$
and the product conic Lagrangian subvariety
$$
\xymatrix{
\Lambda_{n+1} = (\Lambda_1)^{n+1} \subset (T^* T^1)^{n+1} \simeq T^* T^{n+1}
}$$
Recall the conic open subspace
$$
\xymatrix{
\Omega_{n+1} = \{\sum_{a = 1}^{n+1} \xi_a > 0\} \subset T^* T^{n+1}
}$$

The constructions of the preceding section summarized in Theorem~\ref{thm: contacto}
justify the following starting point.

\begin{ansatz}
Set the dg category $\mu\Sh_{L_{n}}(Q_n)$ of wrapped microlocal sheaves on the tailored pair of pants $Q_n$ 
supported along  the skeleton $L_{n}$ to be the dg category
of wrapped microlocal sheaves on $\Omega_{n+1}$ supported along $\Lambda_{n+1}$.
\end{ansatz}

Let us begin with the following variant of Example~\ref{ex:cylinder}.

\begin{lemma}\label{lemma:cylinder}
There are mirror equivalences
$$
\xymatrix{
\Sh_{\Lambda_{n+1}}^\un(T^{n+1}) \simeq \QCoh(\BA^{n+1})
&
 \Sh_{\Lambda_{n+1}}(T^{n+1}) \simeq \Perf_\proper(\BA^{n+1})
}
$$
$$
\xymatrix{
\Sh_{\Lambda_{n+1}}^w(T^{n+1}) \simeq \Coh(\BA^{n+1})
}
$$
\end{lemma}

\begin{proof}

Following Example~\ref{ex:cylinder}, recall the Kronecker quiver category $K_1$
with two objects $a, b$, and two non-identity morphisms $x, y:a\to b$.
Recall
 the equivalences
$$
\xymatrix{
\Sh_{\cS_1}^\un(T^1) \simeq \Mod_k(K_1)\simeq \QCoh(\BP^1)
}
$$

Following Example~\ref{ex:cylinder}, the inclusion $\Lambda_1 \subset T^*_{\cS_1} T^1$ induces a full embedding
$
\Sh_{\Lambda_1}^\un(T^1) \subset \Sh_{\cS_1}^\un(T^1)
$
with image those modules such that the quiver map $x$ is invertible.
Thus for the inclusion $j:\BA^1= \Spec k[y/x] = \{x\not =0\} \to \BP^1$, we have the identification of full  subcategories
$$
\xymatrix{
\Sh_{\Lambda_{1}}^\un(T^1) \simeq j_*(\QCoh(\BA^1))
}$$

Taking products over the product quiver category, we obtain the first asserted equivalence. The third immediately follows by taking compact objects. For the second, one can repeat the above argument working with quiver modules with values in perfect $k$-modules.
\end{proof}

Fix a subset $I \subset [n+1]$, with complement $I^c = [n+1] \setminus I$.

 Let $T^I \subset T^{n+1}$ be the subtorus defined by $\theta_a = 0$, for $a \in I^c$.
 Let $\Lambda_I = (\Lambda_1)^I \subset (T^* T^1)^I \simeq T^* T^I$ be the product conic Lagrangian subvariety.

Consider the hyperbolic restriction 
$$
\xymatrix{
\eta_I:\Sh_{\Lambda_{n+1}}^\un(T^n)\ar[r] &  \Sh_{\Lambda_I}^\un(T^I)
& \eta_I(\cF) = p_* q^!\cF
}
$$
built from the correspondence
 $$
 \xymatrix{
 T^I & \ar[l]_-p T^I \times [0, 1/2)^{I^c} \ar[r]^-{q} & T^{n+1}
 }
 $$
 where $p$ is the evident projection and $q$ the evident inclusion.
  
  To verify $\eta_I$ indeed respects the singular support condition specified, note that it is  simply the product of copies of the hyperbolic restriction $\eta_+$
   introduced in Example~\ref{ex:cylinder}
    in the coordinate directions indexed by $I^c$,  and the identity    in the coordinate directions indexed by $I$. 


Let $f:\BA^I  = \Spec k[t_a \, |\, a\in I\}  \to \BA^n = \Spec k[t_1, \ldots, t_n]$ be the affine subspace defined by $t_a = 0$, for $a \not \in I$. 

\begin{lemma}\label{lemma:restriction}
The first and second equivalences of Lemma~\ref{lemma:cylinder} canonically extend to commutative diagrams
$$
\xymatrix{
\ar[d]_-{\eta}\Sh_{\Lambda_{n+1}}^\un(T^{n+1}) \ar[r]^-\sim &  \QCoh(\BA^{n+1})\ar[d]^-{f^*}
&
\ar[d]_-{\eta}\Sh_{\Lambda_{n+1}}(T^{n+1}) \ar[r]^-\sim &  \Perf_\proper(\BA^{n+1})\ar[d]^-{f^*}\\
\Sh_{\Lambda_{I}}(T^{I}) \ar[r]^-\sim &  \QCoh(\BA^{I})&
\Sh_{\Lambda_{I}}(T^{I}) \ar[r]^-\sim &  \Perf_\proper(\BA^{I})
}
$$
\end{lemma}

\begin{proof}
Since $\eta_I$ preserves constructible sheaves, it suffices to construct the first commutative diagram.
 Recall that the hyperbolic restriction 
 $
 \eta_+:\Sh_{\cS_1}^\un(T^1)\to  \Mod_k
 $
  introduced in Example~\ref{ex:cylinder} corresponds to the $*$-restriction $\QCoh(\BP^1) \to \Mod_k$ to the point of projective space $\{y = 0\} \subset  \BP^1$. Hence under the 
 first equivalence of Lemma~\ref{lemma:cylinder}, its restriction to the full subcategory $\Sh_{\Lambda_1}^\un(T^1) \subset \Sh_{\cS_1}^\un(T^1)$
corresponds to the 
  $*$-restriction $\QCoh(\BA^1) \to \Mod_k$ to the same point regarded in the affine space $\BA^1  = \Spec k[y/x] = \{x \not = 0\} \subset \BP^1$. 
   Now recall that $\eta_I$ is  the product of copies of the hyperbolic restriction $\eta_+$
    in the coordinate directions indexed by $I^c$,  and the identity    in the coordinate directions indexed by $I$. 
   Thus taking products over the coordinate directions, we obtain the assertion.
\end{proof}

Now we arrive at our main calculation. 
Introduce 
 the union of coordinate hyperplanes
$$
\xymatrix{
X_{n} = \{ t_1 \cdots t_{n+1} = 0\} \subset \BA^{n+1}
}
$$

\begin{thm}\label{thm:main calc}
There are mirror  equivalences
$$
\xymatrix{
\mu\Sh_{\Lambda_{n+1}}^\un(\Omega_{n+1}) \simeq \Ind \Coh(X_{n})
&
\mu\Sh_{\Lambda_{n+1}}(\Omega_{n+1}) \simeq \Perf_\proper(X_{n})
}
$$
$$
\xymatrix{
\mu\Sh_{\Lambda_{n+1}}^w(\Omega_{n+1}) \simeq \Coh(X_{n})
}
$$
\end{thm}

\begin{proof}
Let us recall the setup of the descent of  Proposition~\ref{prop:descent}.

Let $\cI_{n+1}$ be the category with objects subsets $I \subset [n+1]$, and morphisms $I\to I'$  inclusions $I\subset I'$. Let $\cI^\circ_{n+1}\subset \cI_{n+1}$ denote the full subcategory of proper subsets $I \not = [n+1]$.

 Let $\fX_k$ denote the category of affine lci $k$-schemes and closed embeddings. 
 
 Define a functor $A:\cI_{n+1}^\circ \to \fX_k$ as follows.
For an object $I\in\cI_{n+1}^\circ$,  take the affine space $A(I) = \BA^I = \Spec k[t_a \, |\, a\in I]$, 
and  for a morphism $I\subset I'$,
  take the inclusion $A(I, I'):\BA^I = \Spec k[t_a \, |\, a\in I] \to \Spec k[t_a \, |\, a\in I'] =  \BA^{I'}$, given by setting $t_a = 0$, for each $a\in I'\setminus I$,
  
 Recall the functor $\Ind\Coh^*: \fX_k^{op} \to \dgSt_k$ that assigns to a scheme its ind-coherent sheaves and to a proper morphism of schemes its $*$-pullback.
 Recall also the full subfunctor $\Perf^*_\proper : \fX_k^{op} \to \dgst_k$ of perfect complexes with proper support.
 
 Consider the composite functor  
 $\Ind\Coh^* \circ A:(\cI_{n+1}^\circ)^{op} \to \dgSt_k$, and subfunctor $\Perf^*_\proper \circ A:(\cI_{n+1}^\circ)^{op} \to \dgst_k$. By the descent of Proposition~\ref{prop:descent}, the canonical maps are equivalences
 $$
 \xymatrix{
 \Ind\Coh(X_{n}) \ar[r]^-\sim & \lim_{(\cI_{n+1}^\circ)^{op}} \Ind\Coh (\BA^I)
&
 \Perf_\proper(X_{n}) \ar[r]^-\sim & \lim_{(\cI_{n+1}^\circ)^{op}} \Perf_\proper (\BA^I)
 }
 $$
  $$
 \xymatrix{
 \Coh(X_{n}) & \ar[l]_-\sim \colim_{\cI^\circ_{n+1}} \Coh (\BA^I)
}
$$

To prove the theorem, we will similarly identify $\mu\Sh_{\Lambda_{n+1}}^\un(\Omega_{n+1})$ as the limit of a functor
 $$
 \xymatrix{
 \mu\Sh^\un:(\cI_{n+1}^\circ)^{op}\ar[r] &  \St_k
 }
 $$
 and  then  provide an equivalence of functors $\mu\Sh^\un \simeq \Ind\Coh^* \circ A$. This will immediately provide the first and third asserted equivalences. 
 
 For the second, 
 we will observe that 
 $\mu\Sh_{\Lambda_{n+1}}(\Omega_{n+1})$ is the limit of a full subfunctor
 $
 \mu\Sh\subset \mu\Sh^\un
 $
equivalent to the subfunctor  $\Perf^*_\proper \circ A \subset \Ind\Coh^*\circ A$.
  
  \medskip
  
 Now to each $I\in \cI_{n+1}^\circ$,  introduce the conic open subspace $\Omega_I \subset \Omega_{n+1}$ cut out by the additional requirement  $\xi_a \not = 0,$ for $a \not \in I$. Thus for $I\subset I'$, we have the open inclusion $\Omega_I \subset \Omega_{I'}$, and for $I = [n+1]$, we have $\Omega_I = \Omega_{n+1}$. 
 Note that the collection $\{\Omega_I\}_{I\in \cI_{n+1}^\circ}$ forms a conic open cover of $\Omega_{n+1}$
 with the property  $\Omega_{I\cap I'} = \Omega_I\cap \Omega_{I'}$.
 
 Define the functor 
 $$
 \xymatrix{
 \mu\Sh^\un:(\cI_{n+1}^\circ)^{op} \ar[r] & \dgSt_k
 &
 \mu\Sh^\un(I) = \mu\Sh^\un_{\Lambda_{n+1}}(\Omega_I)
 }
 $$ 
  with inclusions $I\subset I'$ taken to the restriction maps along the inclusions $\Omega_I \subset \Omega_{I'}$.
 
 Define the full subfunctor
  $
 \mu\Sh \subset \mu\Sh^\un
 $
 by taking
 $
 \mu\Sh(I) = \mu\Sh_{\Lambda_{n+1}}(\Omega_I)
 $

 Since $\mu\Sh_{\Lambda_{n+1}}^\un$ forms a sheaf, $ \mu\Sh_{\Lambda_{n+1}} \subset \mu\Sh^\un_{\Lambda_{n+1}}$ a full subsheaf,  and $\{\Omega_I\}_{I\in \cI_{n+1}^\circ}$ an open conic cover of $\Omega_{n+1}$, the canonical maps  are equivalences
 $$
 \xymatrix{
 \mu\Sh^\un_{\Lambda_{n+1}}(\Omega_{n+1}) \ar[r]^-\sim & \lim_{(\cI_{n+1}^\circ)^{op}} \mu\Sh^\un_{\Lambda_{n+1}}(\Omega_I)
 }
 $$
 $$
 \xymatrix{
 \mu\Sh_{\Lambda_{n+1}}(\Omega_{n+1}) \ar[r]^-\sim & \lim_{(\cI_{n+1}^\circ)^{op}} \mu\Sh_{\Lambda_{n+1}}(\Omega_I)
 }
 $$

Next let us define an  additional functor to interpolate between $\Ind \Coh\circ A$ and $ \mu\Sh^\un$.
 
 For $I\in \cI_{n+1}^\circ$,  recall the subtorus  $T^I =(T^1)^I \subset T^n$  defined by $\theta_a = 0$, for $a\in I^c$,
 and  the corresponding product  conic Lagrangian
 subvariety $\Lambda_I = (\Lambda_1)^I \subset (T^* T^1)^I \simeq  T^* T^I$.

 Define the functor
  $$
 \xymatrix{
 \Sh^\un:(\cI_{n+1}^\circ)^{op} \ar[r] & \dgSt_k
 &
 \Sh^\un(I) = \Sh^\un_{\Lambda_I}(T^I)
 }
 $$ 
 with inclusions $I\subset I'$ taken to the hyperbolic restrictions
 $$
 \xymatrix{
 \eta_{I \subset I'}:\Sh^\un_{\Lambda_{I'}}(T^{I'})\ar[r] &  \Sh^\un_{\Lambda_{I}}(T^{I})   &
\eta_{I \subset I'}(\cF) =  p_*q^!\cF
 }
 $$ 
 built from the correspondences
 $$
 \xymatrix{
 T^I & \ar[l]_-p T^I  \times [0, 1/2)^{I' \setminus I} \ar[r]^-q & T^{I'}
 }
 $$
 where $p$ is the evident projection and $q$ the evident inclusion. 
 
 Note that $ \eta_{I \subset I'}$ is simply the product of  hyperbolic restrictions in the coordinate directions indexed by $I'\setminus I$, and the identity in the coordinate directions indexed by $ I$.
In particular,   it satisfies evident compatibilities providing $ \Sh^\un$ with the structure of a functor.
  
  Define the full subfunctor
   $
 \Sh \subset \Sh^\un
 $ 
  by taking
 $
 \Sh(I) = \Sh_{\Lambda_I}(T^I).
 $ 
 
 Now Lemmas~\ref{lemma:cylinder} and~\ref{lemma:restriction}, along with evident compatibilities, 
 confirm we have equivalences 
 $$
 \xymatrix{
  \Sh^\un \simeq \Ind\Coh^*\circ A
 &
  \Sh \simeq \Perf^*_\proper\circ A
  }
 $$
 
 It remains to establish equivalences of functors
 $$
 \xymatrix{
  \Sh^\un \simeq \mu\Sh^\un
 &
  \Sh \simeq \mu\Sh
  }
 $$

For any $I\in \cI_{n+1}^\circ$, let us return to the hyperbolic restriction
 $$
 \xymatrix{
 \eta_I:\Sh^\un_{\Lambda_{n+1}}(T^{n+1})\ar[r] &  \Sh^\un_{\Lambda_{I}}(T^{I})   &
\eta_I(\cF) =  p_*i^!\cF
 }
 $$ 
 built from the correspondence
 $$
 \xymatrix{
 T^I & \ar[l]_-p T^I  \times [0, 1/2)^{I^c} \ar[r]^-q & T^{n+1}
 }
 $$
 
 First, observe that $\eta$ factors through the microlocalization
 $$
 \xymatrix{
 \Sh^\un_{\Lambda_{n+1}}(T^{n+1}) 
 \ar[r] & \mu\Sh^\un_{\Lambda_I}(\Omega_I)\ar[r]^-{\tilde\eta_I} &  \Sh^\un_{\Lambda_{I}}(T^{I})   &
 }
 $$ 
since the hyperbolic restriction in the coordinate direction indexed by $a\in I^c$ vanishes on sheaves whose singular support does not intersect the locus $\{\xi_a >0\} \subset T^* T^{n+1}$.

 Next, for $I \subset I'$, observe that the induced functors extend to natural commutative diagrams
 $$
 \xymatrix{
 \ar[d]_{\rho_{I \subset I'}}\mu\Sh^\un_{\Lambda_{I'}}(\Omega_{I'})\ar[r]^-{\tilde \eta_{I'}} &  \Sh^\un_{\Lambda_{{I'}}}(T^{I'})  \ar[d]^{\eta_{I \subset I'}}  \\
 \mu\Sh^\un_{\Lambda_I}(\Omega_I)\ar[r]^-{\tilde \eta_{I}} &  \Sh^\un_{\Lambda_{I}}(T^{I})   
 }
 $$ 
following the identities of Example~\ref{ex cod 1}.

Thus we have a map of functors $\tilde \eta:\mu\Sh^\un\to \Sh^\un$, restricting to a map of subfunctors $\mu\Sh\to \Sh$. It remains to show  $\tilde \eta:\mu\Sh^\un\to \Sh^\un$  is an equivalence. 
 For this, it suffices to  fix $I\in \cI_{n+1}^\circ$ and show that
 $$
 \xymatrix{
\tilde \eta_I :\mu\Sh^\un_{\Lambda_I}(\Omega_I)\ar[r] &  \Sh^\un_{\Lambda_I}(T^I)
 }
 $$ 
 is an equivalence. But it admits an inverse induced by the pushforward
 $$
 \xymatrix{
 j_{I*}:\Sh^\un_{\Lambda_I}(T^I) \ar[r] &  \Sh^\un_{\Lambda_{n+1}}(T^{n+1})
 }
 $$ 
 along the natural inclusion $j_I: T^I \to T^{n+1}$.  To see this, 
 note that $ j_{I}$ is simply the product of  the inclusions in the coordinate directions indexed by $I^c$,
 and the identity in the coordinate directions indexed by $I$.

 This concludes the proof of the theorem.
\end{proof}

Composing the last equivalence of Theorem~\ref{thm:main calc} with the equivalence of Proposition~\ref{prop:mfcohequiv}, we obtain the following. Recall that to a dg category $\cC$, we assign its folding $\cC_{\BZ/2}$ by forgetting structure to obtain a $\BZ/2$-dg category.

\begin{corollary}\label{cor:mirrorsym}
There is an  equivalence of $\BZ/2$-dg categories
$$
\xymatrix{
\mu\Sh_{\Lambda_{n+1}}^w(\Omega_{n+1})_{\BZ/2} \simeq \MF(\BA^{n+2}, W_{n+2})
}
$$
\end{corollary}




\end{document}